\documentclass[suppldata]{interact}
\usepackage{latexsym}
\usepackage{anyfontsize} 
\usepackage{color}
\usepackage[hidelinks]{hyperref}
\usepackage{url}
\usepackage{breakurl}
\usepackage[table]{xcolor}
\newcommand{\bburl}[1]{\textcolor{blue}{\url{#1}}}
\usepackage[utf8]{inputenc}

\usepackage{subfigure}
\usepackage{xcolor}
\usepackage{dsfont}
\usepackage{placeins} 

\newcommand{\mymod}[1]{\mod #1} 

\makeatletter
\newcommand{\monthyear}[1]{%
  \def\@monthyear{\uppercase{#1}}}
\newcommand{\volnumber}[1]{%
  \def\@volnumber{\uppercase{#1}}}
\makeatother

\newcommand{\N}{{\mathbb N}}
\newcommand{\Z}{{\mathbb Z}}
\newcommand{\1}[1]{{\mathds 1}_{\Z\setminus N\Z}\left(#1\right)}

\theoremstyle{plain}
\newtheorem{thm}{Theorem}[section] 
\newtheorem{theorem}[thm]{Theorem}

\newtheorem{definition}[thm]{Definition}
\newtheorem{proposition}[thm]{Proposition}

\numberwithin{table}{section} 
\numberwithin{figure}{section}

\newcommand{\Llra}{\mathbf{\Longleftrightarrow~}}
\newcommand{\Lra}{\mathbf{\hspace{0.08cm}\Longrightarrow~}}

\definecolor{mygreen}{RGB}{198,239,206}
\definecolor{myblue}{RGB}{221,235,247}
\definecolor{myyellow}{RGB}{255,242,204}
\definecolor{myred}{RGB}{255,200,200}  

\definecolor{mygreenD}{RGB}{120,215,150}
\definecolor{myblueD}{RGB}{140,200,255}
\definecolor{myyellowD}{RGB}{255,235,120} 
\definecolor{myredD}{RGB}{255,150,120} 

\newcommand{\gcell}[1]{\cellcolor{mygreen}#1}
\newcommand{\bcell}[1]{\cellcolor{myblue}#1}
\newcommand{\ycell}[1]{\cellcolor{myyellow}#1}

\newcommand{\gcellD}[1]{\cellcolor{mygreenD}{\bf #1}}
\newcommand{\bcellD}[1]{\cellcolor{myblueD}{\bf #1}}
\newcommand{\ycellD}[1]{\cellcolor{myyellowD}{\bf #1}}
\newcommand{\rcellD}[1]{\cellcolor{myredD}{\bf #1}}

\newcommand{\gcelld}[1]{\cellcolor{mygreen}{\bf #1}}

\newcommand{\ycelld}[1]{\cellcolor{myyellow}{\bf #1}}
\newcommand{\rcelld}[1]{\cellcolor{myred}{\bf #1}}

\begin{document}

\setcounter{page}{1}

\title{Pascal tiling and congruences modulo $N$ in Pascal's triangle.
}

\author{
\name{Etienne Rousseau\textsuperscript{a}\thanks{Email address: xxx@xxx.fr} and Guillaume
Rousseau\textsuperscript{b}\thanks{Email address: xxx@xxx.fr}}
\affil{\textsuperscript{a} Corresponding author;
        \textsuperscript{b} Laboratoire Matières et Systèmes Complexes, UMR 7057, CNRS and Université Paris Cité, 10 rue Alice Domon et Léonie Duquet, F-75013 Paris Cedex 13, France
                }
}

\maketitle

\bigskip

\begin{abstract}
We investigate the properties of matrices obtained from a geometric
transformation of the first $N$ rows of Pascal's triangle. For $N > 2$,
their congruence properties form a \emph{Pascal tiling}, that is,
a perfect alternation between entries 
congruent to $0 \mymod{N}$ and the others, if and only if $N$ is prime.

This result yields an alternative proof of the classical congruence
$L_N-1\equiv 0 \mymod{N}$ for prime $N$, where $L_N$ denotes the
$N$th Lucas number. Within the framework of the {\it Pascal tiling theorem},
this congruence can be expressed as a sum of entries lying along a diagonal
of one of the matrices under consideration; when $N$ is prime, each of these
entries is congruent to $0 \mymod{N}$. By contrast, for Fibonacci pseudoprimes,
the sum remains congruent to $0 \mymod{N}$ while at least one of its terms is not.

Finally, these results are interpreted in terms of decompositions of binomial
coefficients and extended to multinomial coefficients, leading to a study of
the associated symmetries. This perspective highlights the case
where $N$ is a prime power and clarifies the conditions under which a
Pascal tiling arises.
\end{abstract}


\section{Introduction}

\label{sec:introduction}

Pascal's triangle occupies a central place in the history of
mathematics due to its rich combinatorial and arithmetic
structure, all of which arise from a remarkably simple definition.
Since the foundational work of Edouard Lucas---whose results led to
the theorem that now bears his name---and the early studies on the
congruence properties of the binomial coefficients forming the
triangle~\cite{lucas_sur_1878,glaisher_residue_1899,
fine_binomial_1947}, numerous works have further investigated these
properties modulo a prime or a 
prime power~\cite{granville_arithmetic_1997, mestrovic_lucas_2014}.

Pascal's triangle reduced modulo a prime exhibits a fractal
structure, as shown by Wolfram~\cite{wolfram_geometry_1984} and
Berg~\cite{berg_decouvertes_1993}. More recent studies have also
highlighted various symmetries related to the properties of the
triangle considered modulo $p$~\cite{cinkir_extension_2023} or
modulo $p^2$~\cite{rowland_lucas_2022}, where $p$ is a prime.
More generally, symmetry has been studied in the Pascal--Kummer
triangle, defined by assigning to each binomial coefficient its 
$p$-adic valuation~\cite{prunescu_symmetries_2022}.

We study the congruence properties of matrices constructed from the
first $N$ rows of Pascal's triangle.
Figure~\ref{fig:building} illustrates how one such matrix 
can be obtained geometrically by rotation, 
symmetrization and addition.
The resulting matrix $S^N$ consists of $N \times N$ cells and is symmetric
with respect to both diagonals. Each cell is colored according to the
divisibility by $N$ of the corresponding entry.
\begin{figure}[tb]
  \centering
  \includegraphics[width=1.00\linewidth]{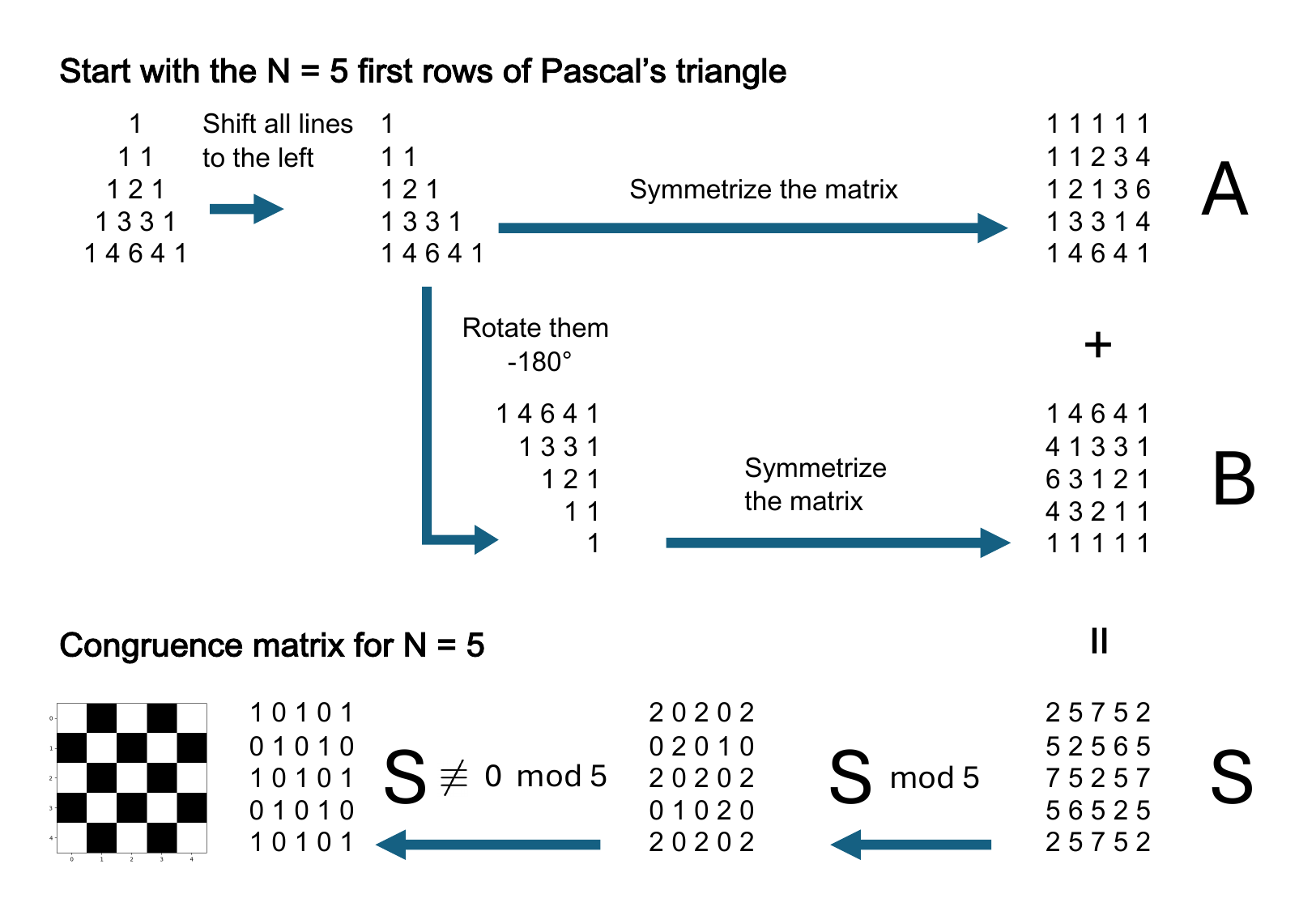}
  \caption{
  Example of the geometric construction of the matrix $S^N$ for $N = 5$,
  based on the first $N$ rows of Pascal's triangle. The congruence
  properties are visualized by coloring the $(i,j)$ entries in black
  when $S^N_{i,j} \equiv 0 \mymod{N}$, and in white otherwise.
}
  \label{fig:building}
\end{figure}

We state and prove Theorem~\ref{theo:PTT} (Section~\ref{sec:theorem}),
showing that the congruence properties of 
the matrices $S^N$ and $W^N$ exhibit a periodic pattern
if and only if $N$ is prime. We call this pattern a {\it Pascal tiling}:
in a given index grid, a perfect alternation between entries congruent to $0 \mymod{N}$
and the others, which can be viewed 
as a checkerboard structure in the 
corresponding congruence matrix.\footnote{
This notion will be refined later through a more general construction
(Theorem~\ref{theo:PTTUVK}).}

This result is then related to properties of Lucas
numbers~\cite{lucas_theorie_1878, lucas_theorie_1878-1,
lucas_theorie_1878-2} and to Fibonacci pseudoprimes
(Section~\ref{sec:fibonacci}), where we prove that the condition
$L_N - 1 \equiv 0 \mymod{N}$ can be expressed as a sum whose terms are
individually congruent to $0 \mymod{N}$ if and only if $N$ is prime
(Theorem~\ref{theo:pseudoprime}).

These results are naturally interpreted in terms of decompositions of
binomial coefficients, leading to the canonical pair $(X^N, Y^N)$
and Theorem~\ref{theo:PTTXY}, 
derived from the original decomposition
$(S^N, W^N)$. This construction extends to a
multinomial setting through the pair $(X^{N,r}_{\vec{k}},
Y^{N,r}_{\vec{k}})$, Theorem~\ref{theo:PTTmultinomial}
(Section~\ref{sec:multinomial}), and is further generalized via a
parametric linear map $\mathcal{F}_K$, leading to the pair
$(U_{\vec{k}}, V_{\vec{k}})$ and Theorem~\ref{theo:PTTUVK}. 

Section~\ref{sec:symmetry} develops this framework, 
which enables the study of the associated symmetries, 
showing how each canonical decomposition in the binomial 
case gives rise to a corresponding symmetry in the 
congruence properties of Pascal's triangle.

Moreover, in canonical decompositions, one of the three 
equivalent properties only holds in one direction, 
as prime powers provide counterexamples to the converse. 
The reconstructed decomposition restores full equivalences, 
yielding characterizations of the {\it Pascal tiling} 
that hold if and only if $N$ is prime.

Table~\ref{tab:summary} provides a concise overview of the main definitions 
and results used in the sequel.

\section{Pascal tiling theorem}
\label{sec:theorem}

In this section, we first define the matrices under study, then
present examples of their congruence properties for both prime and
composite values of $N$, and finally state the Pascal tiling theorem
(Theorem~\ref{theo:PTT}).

\begin{definition}
Let $N \geq 2$, and let $\mathcal{D}$ denote the set of
pairs $(i,j)$ such that $i,j \in \{0,\dots,N-1\}$. The matrices
$S^N$ and $W^N$ are two $N \times N$ symmetric matrices with
entries $S^N_{i,j}$ and $W^N_{i,j}$ defined for indices
$0 \leq j \leq i < N$, and extended to all indices by setting
$S^N_{j,i} = S^N_{i,j}$ and $W^N_{j,i} = W^N_{i,j}$. They are
given by
\begin{eqnarray}
S^N_{i,j} & = & A^N_{i,j} + B^N_{i,j}, \label{Sij}\\
W^N_{i,j} & = & A^N_{i,j} - B^N_{i,j},
\end{eqnarray}
where $A^N_{i,j} = \binom{i}{j}$ and
$B^N_{i,j} = \binom{N-1-j}{N-1-i}$ are binomial coefficients.
We denote by $\mathcal{D}_{j \leq i}$ the subset of
$(i,j) \in \mathcal{D}$ such that $j \leq i$.
\label{def:SijWij}
\end{definition}

\begin{definition}
The congruence properties of the entries $S^N_{i,j}$ and $W^N_{i,j}$
modulo $N$ are described by the indicator functions
\begin{eqnarray}
I^S_{N,i,j} & = & \1{S^N_{i,j}}, \label{IS}\\
I^W_{N,i,j} & = & \1{W^N_{i,j}}. \label{IW}
\end{eqnarray}
We denote by $I^S_N$ and $I^W_N$ the
corresponding congruence matrices.
\end{definition}

Figure~\ref{fig:damier} shows the matrix $I^S_N$ obtained for several
values of $N$, some of which are prime and others composite.
\begin{figure}[tb]
  \centering
  \includegraphics[width=1.00\linewidth]{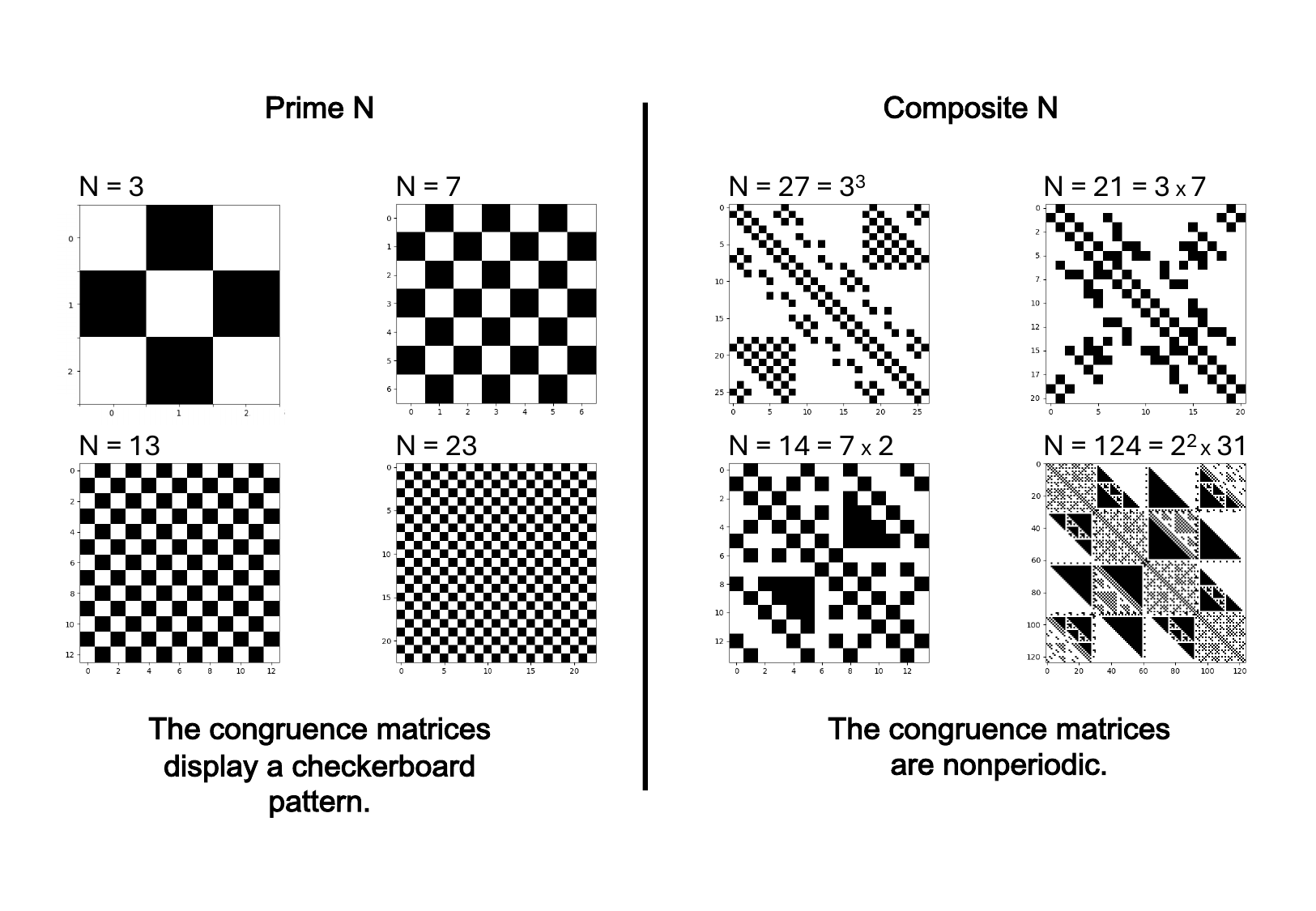}
  \caption{Representation of the congruence properties of the matrix
  $S^N$ for various values of $N$, with prime values on the left and
  composite values on the right. The Pascal tiling
  Theorem~\ref{theo:PTT} states that the congruence matrix $I^S_N$
  forms a perfect alternation of black and white cells (a Pascal
  tiling) if and only if $N$ is prime.}
  \label{fig:damier}
\end{figure}
The Pascal tiling theorem states that, for any  $N > 2$, the
congruence matrix $I^S_N$ forms a tiling 
consisting of a perfect alternation of white
($I^S_{N,i,j} = 1$) and black ($I^S_{N,i,j} = 0$) cells, if and only
if $N$ is prime. The dependence
of this alternation on the arrangement of entries in the $(i,j)$ index
grid is discussed in Section~\ref{sec:symmetry}.

\begin{theorem}[Pascal tiling theorem] Let $N > 2$. The following statements are equivalent: 
\begin{align}
(i)\quad& N \text{ is prime};\nonumber\\
(ii)\quad& \forall (i,j) \in \mathcal{D}_{j \leq i}:\; \1{S^N_{i,j}} = \frac{1+(-1)^{i+j}}{2};\label{eqn:PTT}\\
(iii)\quad& \forall (i,j) \in \mathcal{D}_{j \leq i}:\; \1{W^N_{i,j}} = \frac{1-(-1)^{i+j}}{2};\label{eqn:PTTmiror}\\
(iv)\quad& \forall (i,j) \in \mathcal{D}_{j \leq i}:\; \left|\1{S^N_{i,j}}-\1{W^N_{i,j}}\right| = 1. \label{eqn:PTTSW}
\end{align}
\label{theo:PTT}
\end{theorem}

In the following, we restrict to $(i,j) \in \mathcal{D}_{j \leq i}$ since the symmetry
$(i,j) \mapsto (j,i)$ is imposed by construction in Definition~\ref{def:SijWij}
of the matrices $S^N$ and $W^N$.

We now turn to the proof of this theorem.

\subsection{Preliminaries}

\begin{proposition}
Let $N > 2$ be a prime. Then
\begin{align}
\binom{i}{j}&\not\equiv 0 \mymod{N}\quad(0 \leq j \leq i \leq N{-}1),\label{prop:C}\\
\binom{N}{j}&\equiv 0 \mymod{N}\quad(1 \leq j \leq N-1),\label{prop:D}\\
\binom{N{-}1}{j}&\equiv (-1)^j \mymod{N}\quad(0 \leq j \leq N{-}1).\label{prop:N1j}
\end{align}
\end{proposition}
\begin{proof}
These are classical results.  See for example Theorems 8 and 9 
in~\cite{berg_decouvertes_1993}.
\end{proof}

\begin{proposition}
Let $N = mq > 3$ be a composite, 
where $m > 1$ and $2 \leq q \leq m$ is the
smallest prime factor of $N$. Then the following statements hold:
\begin{align}
\binom{mq}{0}&= \binom{mq-1}{0} = 1,\label{S1}\\
\binom{mq}{q}&= m\binom{mq{-}1}{q{-}1},\label{S2}\\
\binom{mq}{q+1}&= \frac{mq}{q{+}1}\binom{mq{-}1}{q},\label{S2bis}\\
\binom{mq}{k}&\equiv 0 \mymod{mq}\quad\text{(for  $0<k<q$)},\label{S3}\\
\binom{mq{-}1}{k}&\equiv (-1)^k \mymod{mq}\quad\text{(for  $0<k<q$)},\label{S4}\\
\binom{mq}{q}&\equiv m \mymod{mq},\label{S5}\\
\binom{mq{-}1}{q}&\equiv m+(-1)^q \mymod{mq}.\label{S6}
\end{align}
\label{prop:congruence}
\end{proposition}
\begin{proof}
Identities (\ref{S1}), (\ref{S2}) and (\ref{S2bis}) are immediate consequences of the 
usual properties of binomial coefficients.

To prove (\ref{S3}), we expand the factorial
form of binomial coefficients.
Since $q$ is the smallest prime factor of $N = mq$, none of the values $k<q$ is a divisor of $N$.
Thus, none of the factors in $k!$ is a divisor of the term $N$ that appears in the expression
\begin{equation*}
\binom{N}{k} = \frac{N!}{k!(N{-}k)!} = N\frac{(N{-}1)...(N{-}k{+}1)}{k!}.
\end{equation*}
$\binom{N}{k}$ is an integer, and thus the remaining term in the numerator must 
necessarily be divisible by $k!$.
It follows that for all $1 \leq k \leq q - 1$, $\binom{mq}{k} = N H$ for some 
integer $H$, and we obtain the desired expression.

By induction using (\ref{S3}) and (\ref{S1}),
we obtain (\ref{S4}),
\begin{align*}
\binom{mq{-}1}{k}& = \binom{mq}{k}-\binom{mq{-}1}{k-1}\\
& \equiv 0-(-1)^{k-1}\equiv (-1)^k \mymod{mq}
\end{align*}
for all $k \in \{1,\dots,q-1\}$.

Applying (\ref{S2}) together with Lucas's theorem yields (\ref{S5}), namely
$\binom{mq-1}{q-1} \equiv 1 \mymod{q}$.
Hence, there exists $H \in \mathbb{Z}$ such that $\binom{mq-1}{q-1} = 1 + Hq$, and
\begin{align*}
\binom{mq}{q}& = m\binom{mq-1}{q-1} = m(1+Hq) = m+Hmq\\
&\equiv m \mymod{mq}.
\end{align*}

Combining (\ref{S4}) and (\ref{S5}), we obtain (\ref{S6}),
\begin{align*}
\binom{mq{-}1}{q}& = \binom{mq}{q}-\binom{mq{-}1}{q{-}1}\\
&\equiv m-(-1)^{q-1} \mymod{mq}.\qedhere
\end{align*}
\end{proof}

\begin{proposition}
Let $N \geq 2$ and $(i,j) \in \mathcal{D}_{j \leq i}$. Then
\begin{equation*}
\binom{N{-}1}{i}\binom{i}{j} = \binom{N{-}1}{j}\binom{N{-}1{-}j}{N{-}1{-}i}.
\end{equation*}
\label{prop:compobin}
\end{proposition}
\begin{proof}
This result follows from the usual properties of binomial coefficients.
\end{proof}

\begin{proposition}
Let $N \geq 2$ be a prime and $(i,j) \in \mathcal{D}_{j \leq i}$. Then
\begin{equation*}
\binom{i}{j}-(-1)^{i+j}\binom{N{-}1{-}j}{N{-}1{-}i}\equiv 0 \mymod{N}.
\end{equation*}
\label{prop:compomod}
\end{proposition}
\begin{proof}
The demonstration follows from (\ref{prop:N1j}) and Proposition~\ref{prop:compobin}, using
$\binom{N-1}{i} \equiv (-1)^i \mymod{N}$ and $\binom{N-1}{j} \equiv(-1)^j \mymod{N}$ 
for any odd prime $N$ and $0\leq j \leq N-1$. $N=2$ is trivial since $\binom{i}{j}=1$ for all 
$(i,j)\in\mathcal{D}_{j \leq i}$.

This argument is closely related to that of~\cite{cinkir_extension_2023}.
Another proof can be found in~\cite{rowland_lucas_2022} (Corollary~7).
A third proof, based on the expansion of
$\binom{p-a-1}{b} \mymod{p}$ for prime $p$, has also been suggested.
\end{proof}

\subsection{Proof of the Pascal tiling theorem}

The proof of Theorem~\ref{theo:PTT} relies on establishing the
corresponding implications when $N$ is prime
(Proposition~\ref{prop:PTTprime}), and on identifying at least one
pair $(i,j)$ for which they fail when $N$ is composite
(Proposition~\ref{prop:PTTnonprime}).

\begin{proposition}
Let $N > 2$ be a prime and $(i,j) \in \mathcal{D}_{j \leq i}$. Then
\begin{align}
\1{S^N_{i,j}}& = \frac{1+(-1)^{i+j}}{2},\label{PTTprime1}\\
\1{W^N_{i,j}}& = \frac{1-(-1)^{i+j}}{2},\label{PTTprime2}\\
\left|\1{S^N_{i,j}}-\1{W^N_{i,j}}\right|& = 1.\label{PTTprime3}
\end{align}
\label{prop:PTTprime}
\end{proposition}
\begin{proof}
Identity (\ref{PTTprime1}) is a direct consequence of Proposition~\ref{prop:compomod}.
\begin{align*}
S^N_{i,j}& = \binom{i}{j}+\binom{N{-}1{-}j}{N{-}1{-}i}\\
&\equiv\left(1+(-1)^{i+j}\right)\binom{i}{j} \mymod{N}.
\end{align*}
Thus, $S^N_{i,j} \equiv 0 \mymod{N}$ when $i+j$ is odd.
By (\ref{prop:C}), $\binom{i}{j} \not\equiv 0 \mymod{N}$, and 
since $N > 2$, $S^N_{i,j} \equiv 2\binom{i}{j} \not\equiv 0 \mymod{N}$
when $i+j$ is even. Consequently,
$\1{S^N_{i,j}} = \frac{1}{2}\left(1 + (-1)^{i+j}\right)$.

The proof of 
(\ref{PTTprime2}) is identical,
except for the sign separating the two terms,
which results in a permutation of the congruence properties between
even and odd values of $i+j$.
A longer proof of 
(\ref{PTTprime1}) and
(\ref{PTTprime2})
can be obtained from
Proposition~\ref{prop:Srecurrence} (Appendix~\ref{sec:appendix}) and appears in a technical
report~\cite{enneith}.

Combining (\ref{PTTprime1}) and (\ref{PTTprime2}), 
we obtain (\ref{PTTprime3}).
For all $(i,j) \in \mathcal{D}_{j \leq i}$,
\begin{equation*}
\left|\1{S^N_{i,j}} - \1{W^N_{i,j}}\right|
= \left|\frac{1 + (-1)^{i+j}}{2}
      - \frac{1 - (-1)^{i+j}}{2}\right|
= \left|(-1)^{i+j}\right|
= 1.
\qedhere
\end{equation*}
\end{proof}

\begin{proposition}
Let $N = mq > 2$ be a composite, where $m > 1$ and $q$ is the
smallest prime factor of $N$.

\medskip

\noindent\textbf{(i) Suppose that $\mathbf q$ is odd} (i.e., $q \ge 3$, $N \ge 9$). Then
\begin{align}
\1{S^N_{q,0}}&~\ne~\frac{1+(-1)^q}{2},\label{eq:PTTnonprime1}\\
\1{W^N_{q+1,0}}&~\ne~\frac{1-(-1)^{q+1}}{2},\label{eq:PTTnonprime2}\\
\left|\1{S^N_{q,0}}-\1{W^N_{q,0}}\right|&~\ne~1.\label{eq:PTTnonprime3}
\end{align}

\medskip

\noindent\textbf{(ii) Suppose that $\mathbf q$ is even} (i.e., $q = 2$, $N \ge 4$). Then
\begin{align}
\1{S^N_{m-1,0}}\ne\frac{1+(-1)^{m-1}}{2} &\quad{\rm or}\quad \1{S^N_{m,0}}\ne\frac{1+(-1)^{m}}{2},\label{eq:PTTnonprime4}\\
\1{W^N_{m-1,0}}\ne\frac{1-(-1)^{m-1}}{2} &\quad{\rm or}\quad \1{W^N_{m,0}}\ne\frac{1-(-1)^{m}}{2},\label{eq:PTTnonprime5}\\
\left|\1{S^N_{q,0}}-\1{W^N_{q,0}}\right|&~\ne~1\quad({\mathbf N \ne 4}),\label{eq:PTTnonprime6}\\ 
\left|\1{S^4_{2,1}}-\1{W^4_{2,1}}\right|&~\ne~1\quad({\mathbf  N = 4}).\label{eq:PTTnonprime7}
\end{align}
\label{prop:PTTnonprime}
\end{proposition}
\begin{proof}
Using Definition~\ref{def:SijWij} of $S^N_{q,0}$ and
$W^N_{q,0}$, together with Proposition~\ref{prop:congruence}, we obtain
the following properties, valid whether $q$ is even
or odd.
\begin{align*}
S^{mq}_{q,0}& = \binom{q}{0}+\binom{mq{-}1}{q}
\quad\equiv\quad 1+m+(-1)^q
\mymod{mq},\\
W^{mq}_{q,0}& = \binom{q}{0}-\binom{mq{-}1}{q}
\quad\equiv\quad1-m-(-1)^q
\mymod{mq},\\
W^{mq}_{q+1,0}&= \binom{q{+}1}{0}-\binom{mq{-}1}{q{+}1}
\quad = \quad1+\binom{mq{-}1}{q}-\binom{mq}{q{+}1}\\
& =1+\binom{mq{-}1}{q}-\frac{mq}{q+1}\binom{mq{-}1}{q}.
\end{align*}

\medskip

We consider separately the cases where the smallest prime factor $q$ is odd or even.

\noindent\textbf{(i) Case $\mathbf{q}$ odd.}
Then $mq$ is also odd and $m > 2$. Thus,
\begin{equation*}
S^{mq}_{q,0} \equiv m \not\equiv 0 \mymod{mq},
\quad
W^{mq}_{q,0} \equiv 2 - m \not\equiv 0 \mymod{mq}.
\end{equation*}

Hence, (\ref{eq:PTTnonprime1}) holds, since
$\1{S^N_{q,0}} = 1$, whereas $\frac{1}{2}\left(1+(-1)^q\right) = 0$.

Similarly, (\ref{eq:PTTnonprime3}) follows from
$\1{S^N_{q,0}} = 1$ and $\1{W^N_{q,0}} = 1$.

Regarding (\ref{eq:PTTnonprime2}), 
$q+1$ is even and does not divide $mq$,
since otherwise $N$ must be even, and $q$ which is odd 
would not be the smallest prime factor of $N$.
In particular, $q+1$ divides $\binom{mq-1}{q}$, and according to (\ref{S2bis})
\begin{align*}
\binom{mq}{q+1}& = \frac{mq}{q+1}\binom{mq-1}{q} \equiv 0 \mymod{mq},\\
\vspace{-0.2cm}\\
W^{mq}_{q+1,0} &\equiv S^{mq}_{q,0} \equiv m\not\equiv 0 \mymod{mq}.
\end{align*}

Therefore, $\1{W^{N}_{q+1,0}} = 1$, whereas $\frac{1}{2}\left(1-(-1)^{q+1}\right) = 0$.

\medskip

\noindent\textbf{(ii) Case $\mathbf{q}$ even} (i.e., $q = 2$).
By symmetry of Pascal's triangle,
\begin{equation*}
\binom{2m-1}{m} = \binom{2m-1}{m-1},\quad
S^{2m}_{m,0} = S^{2m}_{m-1,0},\quad
W^{2m}_{m,0} = W^{2m}_{m-1,0}.
\end{equation*}

This gives (\ref{eq:PTTnonprime4}), since these values cannot
simultaneously equal $\frac{1}{2}\left(1+(-1)^m\right)$ and $\frac{1}{2}\left(1+(-1)^{m-1}\right)$.
The argument is similar for 
(\ref{eq:PTTnonprime5}).

The same argument as for (\ref{eq:PTTnonprime3}) yields (\ref{eq:PTTnonprime6}),
\begin{equation*}
S^{2m}_{2,0} \equiv 2 + m \mymod{2m},
\quad
W^{2m}_{2,0} \equiv -m \mymod{2m}.
\end{equation*}

Thus, if $m \ne 2$ (i.e., \textbf{case $\mathbf{N \ne 4}$}),
\begin{equation*}
\1{S^{2m}_{2,0}} = \1{W^{2m}_{2,0}} = 1.
\end{equation*}

The final \textbf{case $\mathbf{N = 4}$ yields} 
(\ref{eq:PTTnonprime7}), since 
\begin{equation*}
S^4_{2,1} = 2+2 \equiv 0 \mymod{4},
\quad
W^4_{2,1} = 2-2 \equiv 0 \mymod{4}.
\qedhere
\end{equation*}
\end{proof}

\section{Congruence properties of Lucas numbers and Fibonacci pseudoprimes}
\label{sec:fibonacci}
Given the connection between Pascal's triangle and the Fibonacci
and Lucas numbers, we use the Pascal tiling theorem to study their
congruence properties.
\subsection{Preliminaries}
\begin{definition}
The Fibonacci numbers are defined by the following sequence $F_0=0$, $F_1=1$, and $F_n=F_{n-1}+F_{n-2}$.
\label{def:Fn}
\end{definition}
They admit the following representation:
\begin{equation}
F_n=\sum_{k=0}^{\infty}\binom{n-1-k}{k}.
\label{eq:Fn}
\end{equation}

Since $\binom{n-1-k}{k} = 0$ whenever $k > n-1-k$, we consider separately the cases where $n$ is even or odd:

-- If $n~(=2m+1)$ is odd, there are $m+1$ nonzero terms in $F_{2m+1}$,
\begin{equation}
F_{2m+1}=\sum_{k=0}^{m}\binom{2m-k}{k}=\sum_{l=0}^{m}\binom{m+l}{m-l}.
\label{eq:Fodd}
\end{equation}

-- If $n~(=2m)$ is even, there are $m$ nonzero terms in $F_{2m}$,
\begin{equation}
F_{2m}=\sum_{k=0}^{m-1}\binom{2m-1-k}{k}=\sum_{l=0}^{m-1}\binom{m+l}{m-l-1}.
\label{eq:Feven}
\end{equation}

\begin{definition}
The Lucas numbers are defined by the following sequence $L_0=2$, $L_1=1$, and $L_n=L_{n-1} + L_{n-2}$.
\label{def:Ln}
\end{definition}

They are related to the Fibonacci numbers by several identities.
We mention here only one we will use later:
\begin{equation}
L_n = F_{n+1} + F_{n-1}
\label{eq:Ln}
\end{equation}

\begin{proposition}
Let $N > 2$. Then the Fibonacci number $F_N$ and the
Lucas number $L_N$ can be expressed as sums of $S^N_{i,j}$ or $W^N_{i,j}$, over pairs
$(i,j) \in \mathcal{D}_{j \leq i}$  such that $i+j=N-1$ and $i+j=N-2$,
respectively:
\begin{align}
\sum_{i+j=N-1}S^N_{i,j}&
= 2 F_N,\label{eq:PTTfibo}\\[4pt]
\sum_{i+j=N-2}S^N_{i,j}&=L_N-1,\label{eq:PTTlucas}\\
\sum_{i+j=N-2}W^N_{i,j}&=1-F_N.\label{eq:PTTfibobis}
\end{align}
\end{proposition}
\begin{proof}
According to Definition~\ref{def:SijWij},
\begin{align*}
\sum_{i+j=N-1}S^N_{i,j}
&=\sum_{j=N-1-i}\binom{N-1-j}{j}+\binom{N-1-j}{N-1-(N-1-j)}\\
&=2\sum_{j=N-1-i}\binom{N-1-j}{j}.
\end{align*}

The right-hand side coincides with 
(\ref{eq:Fn}). The conditions
$i+j=N-1$ and $(i,j) \in \mathcal{D}_{j \leq i}$ describe precisely the
index set used in 
(\ref{eq:Fodd}) and (\ref{eq:Feven}), for both odd
and even $N$, yielding 
(\ref{eq:PTTfibo}).

We now derive (\ref{eq:PTTlucas}) from the definition:
\begin{align*}
\sum_{i+j=N-2}S^N_{i,j}~
&=\sum_{j=N-2-i}\binom{N-2-j}{j}+\binom{N-1-j}{N-1-(N-2-j)}\\
&=
\sum_{j=(N-1)-1-i}\binom{(N-1)-1-j}{j}\\
&+\sum_{j+1=(N+1)-1-(i+1)}\binom{(N+1)-1-(j+1)}{j+1}.
\end{align*}
All terms in the first sum such that $N'=N-1$, $i+j=N'-1$ and $\binom{N'-1-j}{j}\ne 0$ belong to $\mathcal{D}_{j \leq i}$.
Hence, the first sum equals $F_{N-1}$ by Definition~\ref{def:Fn}. This is not the case for the second sum.
Let $N'=N+1$,  $j'=0$, then $i'+j'=N'-1$ implies $i'=N'-1=N$. Since $(N,0)\notin\mathcal{D}_{j \leq i}$, the term
$\binom{N'-1}{0}$ in 
(\ref{eq:Fn}) for $N'=N+1$ is missing. All other terms belong to $\mathcal{D}_{j \leq i}$.
From (\ref{eq:Ln}), we obtain 
\begin{equation*}
\sum_{i+j=N-2}S^N_{i,j}
=F_{N-1}+ F_{N+1}-\binom{N}{0}
=F_{N-1}+ F_{N+1}-1
=L_N-1.
\end{equation*}
Proof of \eqref{eq:PTTfibobis} is similar to proof of \eqref{eq:PTTlucas},
\begin{equation*}
\sum_{i+j=N-2}W^N_{i,j}
=F_{N-1}-F_{N+1}+\binom{N}{0}
=1-F_{N}.
\qedhere
\end{equation*}
\end{proof}

An example illustrating the links between the Lucas numbers and the terms $A^N_{i,j}$, $B^N_{i,j}$ and $S^N_{i,j}$ is provided in Table~\ref{tab:fibo-grid} for $N=11$.

\begin{theorem}
Let $N$ be prime. Then
$L_N-1 \equiv 0 \mymod{N}$.
\label{theo:LN}
\end{theorem}
\begin{proof}
Several proofs of this classical result are known. We show here how
the Pascal tiling theorem provides a straightforward proof.

The case $N=2$ is immediate. Assume now that $N > 2$ is prime.
Then both $N$ and $N-2$ are odd. According to (\ref{eqn:PTT}) of
the Pascal tiling theorem, all the terms $S^N_{i,j}$ with $i+j = N-2$
appearing in (\ref{eq:PTTlucas}) satisfy
$S^N_{i,j} \equiv 0 \mymod{N}$. 
Hence, the sum is congruent to $0$ modulo $N$.
\end{proof}

The converse of Theorem~\ref{theo:LN} is false. Fibonacci
pseudoprimes are composite $N$ such that
$L_N \equiv 1 \mymod{N}$. As established in previous
works~\cite{di_porto_nonexistence_1993,somer_even_1991}, all such
numbers are odd.
The Pascal tiling theorem shows that the congruence
$\sum_{i+j=N-2} S^N_{i,j} \equiv 0 \mymod{N}$ for prime $N$ actually holds
term by term. For Fibonacci pseudoprimes, the sum modulo $N$ is also zero,
but at least one of these terms is nonzero, as will be shown.

\subsection{Fibonacci pseudoprimes}
Regardless of the particular case of Fibonacci pseudoprimes, the
study of the congruences of the $S^N_{i,j}$ with $i+j = N-2$
shows that, for large values of composite $N$ 
with few prime factors, many of the $S^N_{N-2-j,j}$ are 
congruent to $0$ modulo $N$.
The following proposition provides an explanation for this observation.
\begin{proposition}
Let $N > 2$ and $(i,j) \in \mathcal{D}_{j \leq i}$ with
$i+j = N-2$. Then
\begin{align}
S^N_{N-2-j,j}&=\frac{N}{j+1}\binom{N-2-j}{j},\label{eqn:fibozero0}\\[4pt]
\gcd(j{+}1,N)=1&\Longrightarrow S^N_{N{-}2{-}j,j}\equiv 0 \mymod{N},\label{eqn:fibozero1}\\[4pt]
L_N-1&\equiv\sum_{\substack{0 \leq j \leq N{-}2\\ \gcd(j{+}1,N)\ne1}}S^N_{N{-}2{-}j,j} \mymod{N}.\label{eqn:fibozero2}
\end{align}
\label{prop:fibozero}
\end{proposition}
\begin{proof}
Let $(i,j) \in \mathcal{D}_{j \leq i}$ be such that $i+j = N-2$. Then
\begin{align*}
S^N_{N-2-j,j}&=\binom{N-2-j}{j}+\binom{N-1-j}{j+1}
=\left(1+\frac{N-1-j}{j+1}\right)\binom{N-2-j}{j}\\
&=\frac{N}{j+1}\binom{N-2-j}{j}.
\end{align*}
Since $\gcd(j+1,N)=1$, none of the prime divisors of $j+1$ divides $N$.
Thus, $j+1$ divides $\binom{N-2-j}{j}$,
and hence, 
$S^N_{N-2-j,j} \equiv 0 \mymod{N}$, proving (\ref{eqn:fibozero1}).

Restricting (\ref{eq:PTTlucas}) to the nonzero terms using (\ref{eqn:fibozero1}) 
gives (\ref{eqn:fibozero2}).
\end{proof}

This result is consistent with the Pascal tiling theorem,
since for a prime $N$, none of the $j+1$ with
$0 \leq j \leq N-2$ divides $N$.

\begin{theorem}
Let $N$ be a Fibonacci pseudoprime. Then there exists
$(i,j) \in \mathcal{D}_{j \leq i}$ with $i+j = N-2$
such that $S^N_{i,j}\not\equiv 0 \mymod{N}$.
\label{theo:pseudoprime}
\end{theorem}
\begin{proof}
The proof is not restricted to Fibonacci pseudoprimes, but applies to all composite $N$.
Let $N = mq$ be a composite number with $q$ one of the prime factors of $N$, and define
\begin{equation*}
i_q = (m-1)q-1, \qquad
j_q = q-1.
\end{equation*}
Since $N = mq$ is composite and $q$ is a prime factor of $N$, we have
$q \geq 2$ and $m \geq 2$.
Using (\ref{eqn:fibozero0}),
\begin{align*}
S^{mq}_{(m-1)q-1,q-1}&=\frac{mq}{q}\binom{mq-2-q+1}{q-1}=m\binom{(m-1)q-1}{q-1}\\
&=m\binom{(m-1)q-1}{(m-2)q}=m\binom{(m-2)q+q-1}{(m-2)q}.
\end{align*}
By Lucas's theorem, $\binom{(m-2)q+q-1}{(m-2)q}\equiv 1 \mymod{q}$, so 
there exists $H$ such that $S^{mq}_{i_q,j_q}=m(1+Hq)\equiv m \mymod{mq}$.
This completes the proof.

Another proof is based on $i'_q = (N+q)/2 - 1$ and
$j'_q = (N-q)/2 - 1$, where $N$ is still a composite number,
$q$ is a prime factor of $N$ (odd or even), and $m = N/q$
is odd.
Under these conditions, which hold for all Fibonacci pseudoprimes,
it can be proved as shown in Appendix~\ref{sec:appendix},
Proposition~\ref{prop:composite}, that
$S^{mq}_{i'_q,j'_q} \equiv m \not\equiv 0 \mymod{mq}$.
Since $i'_q + j'_q = N - 2$, this term contributes another
nonzero term to the expression for $L_N - 1$.
\end{proof}

Thus, if $N$ is a Fibonacci pseudoprime, the sum of the terms $S^N_{N-2-j,j}$ 
appearing in the computation of $L_N - 1$
in (\ref{eqn:fibozero2}) is congruent to $0$ modulo $N$, 
while at least one term is nonzero. 
Thus, at least one other term must cancel it, and hence there are at least 
two nonzero terms modulo $N$.

Since the proofs of $S^{mq}_{i_q,j_q}\equiv S^{mq}_{i'_q,j'_q}\equiv m \mymod{mq}$ 
are true for any prime factor 
of a Fibonacci pseudoprime $N$, the number of
nonzero terms in the sum in (\ref{eqn:fibozero2})
is bounded below by twice the number of its prime factors when all $(i_q,j_q)$ 
and $(i'_q,j'_q)$ have distinct values.\footnote{
For an even composite, the lower bound on the number of nonzero terms with 
$i+j=N-2$ is given by the number of prime factors, plus those primes $q$ 
for which $m=N/q$ is odd, still corrected by duplicated values.}
Table~\ref{tab:fibo705} lists the values of $S_{N-2-j,j}$ for 
$N = 705 = 3 \times 5 \times 47$, the smallest Fibonacci
pseudoprime.

Without restricting to Fibonacci pseudoprimes, one can further analyze 
the congruence properties of $L_N - 1$ modulo each prime factor 
of a composite $N$, and bound the number of nonzero terms as well 
as the number of distinct values of $S^N_{N-2-j,j} \mymod{N}$ 
using Proposition~\ref{prop:fibozero}. A detailed study of 
these aspects is left for future work.

\section{Canonical decompositions and multinomial extensions}
\label{sec:multinomial}
\subsection{Decomposition of binomial coefficients}
\label{subsec:decomposition}
\begin{proposition}
Let $0 \leq j \leq i$, and let $p>2$ be a prime such that $p>i$.
Then the binomial coefficient $\binom{i}{j}$ can be written as
the weighted sum
\begin{equation*}
\binom{i}{j}=\frac{S^p_{i,j}+W^p_{i,j}}{2},
\end{equation*}
where exactly one of the two terms is divisible by $p$.
\end{proposition}

\begin{proof}
From Definition~\ref{def:SijWij}, for all $0 \leq j \leq i$ and $N>i$, we have 
$\binom{i}{j}=\frac{1}{2}\left(S^N_{i,j}+W^N_{i,j}\right)$. 
Exactly one of the two terms is divisible by $p$, as a consequence of the 
Pascal tiling theorem (Theorem~\ref{theo:PTT}).
\end{proof}

This decomposition is not equivalent to Euclidean division. By construction, $S^N_{i,j}>0$, whereas $W^N_{i,j}$ need not be and there is no guarantee that the term not congruent to $0$ modulo $p$ is bounded by $p$.

To determine which term is congruent to $0$ modulo $p$, independently of the indices considered, one may derive an alternative expression that explicitly separates the terms congruent to $0$ modulo $p$ from those that are not.

\begin{definition}
Let $N \geq 2$, and let $(i,j) \in \mathcal{D}_{j\leq i}$.
The entries $X^N_{i,j}$ and $Y^N_{i,j}$ are
given by
\begin{align}
X^N_{i,j}&=\frac{1-(-1)^{i+j}}{2}S^N_{i,j}+\frac{1+(-1)^{i+j}}{2}W^N_{i,j},
\label{Xij}\\
Y^N_{i,j}&=\frac{1+(-1)^{i+j}}{2}S^N_{i,j}+\frac{1-(-1)^{i+j}}{2}W^N_{i,j}.
\label{Yij}
\end{align}
\label{def:XijYijfun}
We denote by $(X^N,Y^N)$ the corresponding decomposition and by $\mathcal{F}_z$ the parametric 
linear map that transforms
$(S^N_{i,j},W^N_{i,j})$ to $(X^N_{i,j},Y^N_{i,j})$ for $z=i+j$.  
\end{definition}

It can be checked that $(S^N_{i,j},W^N_{i,j})=\mathcal{F}_{i+j}\left(X^N_{i,j},Y^N_{i,j}\right)$, and simply that
$\mathcal{F}_z\circ \mathcal{F}_z = \mathrm{Id}$.

The Pascal tiling theorem (Theorem~\ref{theo:PTT}) can then be extended to this decomposition.
\begin{theorem}
Let $N > 2$. 
The following statements are equivalent:
\begin{align}
(i)\quad&N \text{ is prime}; \nonumber\\
(ii)\quad&\forall (i,j) \in \mathcal{D}_{j \leq i}:\; \1{X^N_{i,j}} = 0; \label{eqn:PTTX}\\
(iii)\quad&\forall (i,j) \in \mathcal{D}_{j \leq i}:\; \left|\1{X^N_{i,j}} - \1{Y^N_{i,j}}\right| = 1. \label{eqn:PTTXY}
\end{align}

\medskip
Moreover, if $N$ is prime, then
\begin{align}
& \forall (i,j) \in \mathcal{D}_{j \leq i}:\; \1{Y^N_{i,j}} = 1. \label{eqn:PTTY}
\end{align}
\label{theo:PTTXY}
\end{theorem}
\begin{proof}
Expanding Definitions~\ref{def:XijYijfun} and~\ref{def:SijWij}, we obtain
\begin{align}
X^N_{i,j}&=
\binom{i}{j}-(-1)^{i+j}\binom{N-1-j}{N-1-i},\label{XijBinome}\\
Y^N_{i,j}&=
\binom{i}{j}+(-1)^{i+j}\binom{N-1-j}{N-1-i}.\label{YijBinome}
\end{align}

Using Proposition~\ref{prop:compomod} together with 
(\ref{XijBinome}), 
if $N$ is prime, $X^N_{i,j}\equiv 0 \mymod{N}$,
which gives (\ref{eqn:PTTX}).

Since $N > 2$ is prime, Proposition~\ref{prop:compomod},
(\ref{prop:C}) and (\ref{YijBinome}) imply that
$Y^N_{i,j} \equiv 2\binom{i}{j} \not\equiv 0 \mymod{N}$.
This gives (\ref{eqn:PTTY}) for odd prime.

Combining (\ref{eqn:PTTX}) and (\ref{eqn:PTTY}), we obtain (\ref{eqn:PTTXY}).

Some of the counterexamples for composite $N=mq$ 
where $q$ is the smallest prime factor of $N$
are the same as those used in the
proof of Proposition~\ref{prop:PTTnonprime}:

\medskip 

\noindent\textbf{(i) Case $\mathbf{q}$ odd.}
\begin{align*}
X^{mq}_{q,0}&\equiv 1-(-1)^q\left(m+(-1)^q\right)\equiv m\not\equiv 0 \mymod{mq},\\
Y^{mq}_{q,0}&\equiv 1+(-1)^q\left(m+(-1)^q\right)\equiv 2-m\not\equiv 0 \mymod{mq}.
\end{align*}
This shows that \ref{eqn:PTTX}) and (\ref{eqn:PTTXY}) fail for odd $q$. 

\medskip 

\noindent\textbf{(ii) Case $\mathbf{q}$ even.}
\begin{align*}
X^{2m}_{2m-1,0}&=\binom{2m-1}{0}-(-1)^{2m-1}\binom{2m-1-0}{2m-1-2m+1}=2\not\equiv 0 \mymod{2m},\\
Y^{2m}_{2m-1,0}&=\binom{2m-1}{0}+(-1)^{2m-1}\binom{2m-1-0}{2m-1-2m+1}=0\equiv 0 \mymod{2m},\\
X^{2m}_{q,0}&\equiv -m \mymod{2m},\quad
Y^{2m}_{q,0}\equiv 2+m \mymod{2m},\quad
X^4_{2,1}\equiv Y^4_{2,1}\equiv 0 \mymod{4}.
\end{align*}

This shows that (\ref{eqn:PTTX}) and (\ref{eqn:PTTXY}) fail for even $q$.

The converse of \eqref{eqn:PTTY} does not hold.
Counterexamples of composite $N$ such that $Y^N_{i,j}\not\equiv 0 \mymod{N}$ for all $(i,j)\in\mathcal{D}_{j \leq i}$
are prime powers of odd primes.
Table~\ref{tab:swxy-grid-9} shows it for $N=3^2$, and a proof for all such $N$ is given in
Appendix~\ref{sec:primepowers}.
\end{proof}

Tables~\ref{tab:swxy-grid-7} and~\ref{tab:swxy-grid-10} illustrate the different decompositions and 
the congruence properties, respectively for $N=7$ and $N=10$.

\begin{proposition}
Let $0 \leq j \leq i$, and let $p>2$ be a prime such that $p>i$.
Then the binomial coefficients $\binom{i}{j}$ and
$\binom{p-1-j}{p-1-i}$ can be decomposed in terms of $X^p_{i,j}$ and
$Y^p_{i,j}$ as follows:
\begin{align*}
\binom{i}{j}&=\frac{1}{2}\left(p~\frac{X^p_{i,j}}{p}+Y^p_{i,j}\right),\\
\binom{p-1-j}{p-1-i}&=\frac{(-1)^{i+j}}{2}\left(-p~\frac{X^p_{i,j}}{p}+Y^p_{i,j}\right).
\end{align*}
\end{proposition}
\begin{proof}
The decompositions follow immediately from Definition~\ref{def:XijYijfun} and Theorem~\ref{theo:PTTXY}.
\end{proof}

\begin{proposition}
Let $N > 2$ and $(i,j) \in \mathcal{D}_{j\leq i}$. Then
the sequences $X^N_{i,j}$ and $Y^N_{i,j}$ satisfy the same recurrence
as the binomial coefficients in Pascal's triangle,
\[\binom{i}{j}=\binom{i-1}{j-1}+\binom{i-1}{j},\] 
but with different initial conditions. Namely,
\begin{align}
X^N_{i,j} &= X^N_{i-1,j-1} + X^N_{i-1,j}, \label{eqn:decompXYrec1}\\
Y^N_{i,j} &= Y^N_{i-1,j-1} + Y^N_{i-1,j}.\label{eqn:decompXYrec2}
\end{align}

The initial conditions are given by
\begin{align}
X^N_{i,0} &= 1 - \binom{i-N}{i},\quad X^N_{i,i}=0,\\
Y^N_{i,0} &= 1 + \binom{i-N}{i}, \quad Y^N_{i,i}=2.
\end{align}

Here $i-N<0$, and the binomial coefficient with negative upper index is defined 
for $n \in \N^*$ and $m \in \N$, by
\[\binom{-n}{m}=(-1)^m\binom{n+m-1}{m}.\]
\label{prop:decompXYrec}
\end{proposition}
\begin{proof}
From (\ref{XijBinome}), (\ref{YijBinome}), and Definition~\ref{def:SijWij}, 
for all $(i,j) \in \mathcal{D}_{j\leq i}$ with  $i \geq 1$ and $j \geq 1$,
\begin{align*}
X^N_{i,j}&=A^N_{i,j}-(-1)^{i+j}B^N_{i,j}\\
&=A^N_{i-1,j-1}+A^N_{i-1,j}-(-1)^{i+j}\left(B^N_{i-1,j-1}-B^N_{i-1,j}\right)\\
&=A^N_{i-1,j-1}-(-1)^{i-1+j-1}B^N_{i-1,j-1}+
A^N_{i-1,j}-(-1)^{i-1+j}B^N_{i-1,j}\\
&=X^N_{i-1,j-1}+X^N_{i-1,j}.
\end{align*}

To apply this recurrence, we need the explicit values of $X^N_{i,i}=0$ and $X^N_{i,0}$:
\begin{equation*}
X^N_{i,0}=\binom{i}{0}-(-1)^i\binom{N-1}{N-1-i}=1-(-1)^i\binom{N-1}{i}=1-\binom{i-N}{i}.
\end{equation*}
This establishes (\ref{eqn:decompXYrec1}).
The proof of (\ref{eqn:decompXYrec2}) is identical.
\end{proof}

Proposition~\ref{prop:decompXYrec} shows that the
decomposition $(X^N,Y^N)$ has the advantage of admitting a recursive
definition of the terms $X^N_{i,j}$ and $Y^N_{i,j}$ independently,
whereas $S^N_{i,j}$ and $W^N_{i,j}$ depend on each other in this
formulation as shown in Proposition~\ref{prop:Srecurrence}
(Appendix~\ref{sec:appendix}).

Both decompositions $(X^N,Y^N)$ and $(S^N,W^N)$ are related to the
symmetry $(i,j)\mapsto(N-1-j,N-1-i)$. Other symmetries have been
identified, such as $(i,j)\mapsto(N-1-i+j,j)$~\cite{rowland_lucas_2022,
cinkir_extension_2023}, and can be used to derive
further decompositions.

The underlying question is how many symmetries of this type exist and
whether the Pascal tiling theorem extends to them. To address this,
we relate our results to properties of multinomial coefficients and
to higher-dimensional generalizations of Pascal's triangle, namely
Pascal simplices.

\subsection{$M$-dimensional Pascal simplices}

We now move from the binomial indexing $(i,j)$ to the multinomial
indexing $\vec{k}=(k_1,\dots,k_M)$, where $K=\sum_{l=1}^M k_l$.
This allows us to extend the previous constructions to the multinomial setting.

\begin{definition}
The multinomial coefficient of order $M$ is defined for $k_l \ge 0$ with $\sum_l k_l = K$ by
\begin{equation*}
\binom{K}{k_{1},k_{2},k_{3},\ldots ,k_{M}}=\binom{K}{{\vec {k}}}={\frac {K!}{k_{1}!k_{2}!k_{3}!\dots k_M!}}
={\frac {K!}{\prod _{l=1}^Mk_{l}!}}.
\end{equation*}
The usual binomial coefficient corresponds to the case $M=2$, with $k_1=j$ and $k_2=i-j$,
$$\binom{i}{j}=\frac{i!}{j!(i-j)!}=\binom{i}{j,i-j}.$$

\label{def:multinome}
\end{definition}

There is no ambiguity in the definitions and notations: 
the presence of a comma or a vector in the lower coefficient always refers to the multinomial 
coefficient, whereas its absence refers to the usual binomial coefficient.

We now show how invariances at order $M+1$ yield the $M$ congruence symmetries
that arise in an M-dimensional Pascal simplex.
We then deduce a generalization of the Pascal tiling theorem to Pascal's simplices
and conclude with a discussion of these results.
\begin{proposition}
Let $M \ge 2$, let $N\ge2$ be a prime and  let $\mathcal{D}_M$ denote the set of vectors
$\vec{k}=(k_1,\dots,k_M)$ such that $k_l\ge 0$, $\sum_{l=1}^M k_l=K$ and $0\leq K\leq N-1$. 
Then, for all $\vec{k} \in \mathcal{D}_M$ and all $r \in \{1,\dots,M\}$, we have 
\begin{equation*}
\binom{K}{\vec{k}}\equiv(-1)^{K+k_r}\binom{N-1-k_r}{k_1,\dots,k_{r-1},N-1-K,k_{r+1},\dots,k_M} \mymod{N}.
\end{equation*}
\label{prop:kmod}
\end{proposition}
\begin{proof}
Consider the Pascal simplex obtained by introducing the index $k_{M+1}=N-1-K$, 
so that $\sum_{l=1}^{M+1} k_l = N-1$. 

Using Definition~\ref{def:multinome}, we obtain
\begin{align*}
\binom{N-1}{k_1,\dots,k_M,N-1-K}
&=\binom{N-1}{k_r}\binom{N-1-k_r}{k_1,\dots,k_{r-1},k_{r+1},\dots,k_M,N-1-K}\\
&=\binom{N-1}{N-1-K}\binom{N-1-(N-1-K)}{k_1,\dots,k_M}\\
&=\binom{N-1}{K}\binom{K}{k_1,\dots,k_M}.
\end{align*}

Rearranging the terms, and 
using (\ref{prop:N1j}), we obtain
$\binom{N-1}{k_l}\equiv (-1)^{k_l} \mymod{N}$ for odd prime. 
Case $N=2$ is trivial since $K=0$ or $K=1$, and $-1\equiv 1\mymod{2}$. 
Multiplying by $(-1)^K$ yields the desired result for all primes.
\end{proof}

This proposition and its proof reduce to Propositions~\ref{prop:compomod} and~\ref{prop:compobin} when $M=2$, $k_1=j$, $k_2=i-j$ and $r=1$.

\begin{definition}
Let $N \ge 2$, $M \ge 2$, and let $r \in \{1,\dots,M\}$.
We define a map that assigns to each vector $\vec{k} \in \mathcal{D}_M$ the 
multinomial coefficient of order $M$ given by 
\begin{equation*}
T^{N,r}_{\vec{k}}=\binom{N-1-k_r}{k_1,\dots,k_{r-1},N-1-K,k_{r+1},\dots,k_M}.
\end{equation*}
\label{def:Tnr}
\end{definition}

For $N=K+1+k_r$, $T^{N,r}_{\vec{k}}$ coincides with the multinomial coefficient $\binom{K}{\vec{k}}$.

\begin{definition}
For all $N \ge 2$, $M \geq 2$, $\vec{k} \in \mathcal{D}_M$, and
$r \in \{1,\dots,M\}$, we define
\begin{align*}
X^{N,r}_{\vec{k}}&=\binom{K}{\vec{k}}-(-1)^{K+k_r}~T^{N,r}_{\vec{k}},\\
Y^{N,r}_{\vec{k}}&=\binom{K}{\vec{k}}+(-1)^{K+k_r}~T^{N,r}_{\vec{k}}.
\end{align*}
\label{def:XYnr}
\end{definition}

The notation is unambiguous (see the discussion above). 
In particular, $X^{N,r}_{\vec{k}}\ne X^N_{i,j}$.

\subsection{Generalized Pascal tiling theorem}

\begin{theorem}
Let $N > 2$, $M \geq 2$, and $r \in \{1,\dots,M\}$.
The following statements are equivalent:
\begin{align}
(i)\quad&N \text{ is prime};\nonumber\\
(ii)\quad&\forall \vec{k} \in \mathcal{D}_M:\;
\1{X^{N,r}_{\vec{k}}}=0;\label{eqn:PTTmultinomialX}\\
(iii)\quad&\forall \vec{k} \in \mathcal{D}_M:\;
\left|\1{X^{N,r}_{\vec{k}}}-\1{Y^{N,r}_{\vec{k}}}\right|=1.\label{eqn:PTTmultinomialXY}
\end{align}

Moreover, if $N$ is prime, then
\begin{align}
&\forall \vec{k} \in \mathcal{D}_M:\;
\1{Y^{N,r}_{\vec{k}}}=1.\label{eqn:PTTmultinomialY}
\end{align}
\label{theo:PTTmultinomial}
\end{theorem}
\begin{proof}
The implications of \eqref{eqn:PTTmultinomialX}, 
\eqref{eqn:PTTmultinomialXY} and \eqref{eqn:PTTmultinomialY}
for prime $N$ are consequences of 
Proposition~\ref{prop:kmod} together with the extension of
(\ref{prop:N1j}) to multinomial coefficients.


To prove the converses of \eqref{eqn:PTTmultinomialX} and \eqref{eqn:PTTmultinomialXY}, 
we show that these conditions fail whenever $N$ is composite.

Let $(k_1,\dots,k_l,\dots,k_M)$ be such that $k_1+\dots+k_M = K \leq N-1$, and let
$r \in \{1,\dots,M\}$. Fix $s \neq r$, and define $\vec{i}_s$ by
$k_s = i$ and $k_l = 0$ for all $l \neq s$. Then $k_r = 0$ and
$K = i$.

In this case, the expressions reduce to the binomial setting:
\begin{align*}
\binom{i}{\vec{i}_s}&=\frac{i!}{0!\dots i!\dots 0!}
=\binom{i}{0}=A^N_{i,0},\\
T^{N,r}_{\vec{i}_s}&=\binom{N-1}{0,\dots,i,\dots,N-1-i,\dots,0}
=\binom{N-1}{i}=B^N_{i,0},\\
X^{N,r}_{\vec{i}_s}&=A^N_{i,0}-(-1)^{i+0}B^N_{i,0}=X^N_{i,0},\\
Y^{N,r}_{\vec{i}_s}&=Y^N_{i,0}.
\end{align*}

Therefore, any counterexample with $j=0$ in the binomial case
(Theorem~\ref{theo:PTTXY}) yields a counterexample in the multinomial case.

This covers all composite $N \neq 4$, since there exists a counterexample 
with $j=0$ in the binomial case for all such $N$.
In the remaining case $N=4$,
we use instead a vector $\vec{k}_4$ such that $K=2$, $k_r=1$ and exactly 
one other component equals $1$, giving
\begin{align*}
X^{4,r}_{\vec{k}_4}&=\binom{2}{1,1}-(-1)^{2+1}\binom{4-1-1}{1,4-1-2}=4\equiv 0 \mymod{4},\\
Y^{4,r}_{\vec{k}_4}&=\binom{2}{1,1}+(-1)^{2+1}\binom{4-1-1}{1,4-1-2}=0\equiv 0 \mymod{4},
\end{align*}
which shows that the required property fails in this case as well.

Hence, if $N$ is composite, the stated conditions do not hold, which
completes the proof of \eqref{eqn:PTTmultinomialX} and \eqref{eqn:PTTmultinomialXY}.
\end{proof}

\section{Discussion}
\label{sec:symmetry}

We now investigate how these multinomial decompositions translate back
to the binomial setting by considering the case $M=2$ (i.e., binomial coefficients),
where additional symmetries arise.

\begin{proposition} 
For $M=2$, the generalized Pascal tiling theorem
(Theorem~\ref{theo:PTTmultinomial}) yields two decompositions,
corresponding respectively to the symmetries $(i,j)\mapsto(N-1-j,N-1-i)$
and $(i,j)\mapsto(N-1-i+j,j)$.
\end{proposition}
\begin{proof}
Set $M=2$, $K=i$ and $k_1=j$ (i.e., $k_2=K-k_1$).
Recall that the notation $X^N_{i,j}$ differs from $X^{N,r}_{\vec{k}}=X^{N,r}_{k_1,k_2}$ defined for $M=2$.

We treat separately the cases $r=1$ and $r=2$.

\medskip

\noindent\textbf{(i) Case $\mathbf{r=1}$.}
\begin{align*}
X^{N,1}_{k_1,K-k_1}&=\binom{K}{k_1,K-k_1}-(-1)^{K+k_1}\binom{N-1-k_1}{N-1-K,K-k_1}\\
&=\binom{i}{j,i-j}-(-1)^{i+j}\binom{N-1-j}{N-1-i,i-j}\\
&=\binom{i}{j}-(-1)^{i+j}\binom{N-1-j}{N-1-i}
=\binom{i}{i-j}-(-1)^{i+j}\binom{N-1-j}{i-j}\\
&=X^N_{i,j}.
\end{align*}
For $M=2$ (binomial coefficients), the case $r=1$ corresponds to the symmetry $(i,j)\mapsto(N-1-j,N-1-i)$.

\medskip

\noindent\textbf{(ii) Case $\mathbf{r=2}$.}
\begin{align*}
X^{N,2}_{k_1,K-k_1}&=\binom{K}{k_1,K-k_1}-(-1)^{K+K-k_1}\binom{N-1-K+k_1}{k_1,N-1-K}\\
&=\binom{i}{j,i-j}-(-1)^{2i-j}\binom{N-1-i+j}{j,N-1-i}\\
&=\binom{i}{j}-(-1)^{j}\binom{N-1-i+j}{j}
=\binom{i}{i-j}-(-1)^{j}\binom{N-1-i+j}{N-1-i}.
\end{align*}
For $M=2$ (binomial coefficients),  the case $r=2$ corresponds to the symmetry $(i,j)\mapsto(N-1-i+j,j)$.
\end{proof}

The pairs $\left(X^{N,r}_{\vec{k}}, Y^{N,r}_{\vec{k}}\right)$ define the $M$ canonical 
decompositions of the $M$-dimensional Pascal simplex. These decompositions arise 
in particular from the invariance of multinomial coefficients of order $M+1$ 
under permutations of the $k_l$ (Definition~\ref{def:multinome}).

The generalization to $\left(S^{N,r,z}_{\vec{k}}, W^{N,r,z}_{\vec{k}}\right)$
is given by the parametric linear map $\mathcal{F}_z$, which,
in the case $M = 2$, yields $\left(S^N_{i,j}, W^N_{i,j}\right)$ for $k_1 = j$,
$k_2 = i - j$, $r = 1$, and $z = i + j$.

\begin{definition}
For all $N \ge 2$, $M \geq 2$, $\vec{k} \in \mathcal{D}_M$, 
$r \in \{1,\dots,M\}$ and $z=\mathcal{Z}\left(r,\vec{k}\right)$, 
we define $S^{N,r,z}_{\vec{k}}$ and $W^{N,r,z}_{\vec{k}}$
(denoted $U_{\vec{k}}$ and $V_{\vec{k}}$ to simplify the notation)
such that
\begin{equation*}
\left(U_{\vec{k}},V_{\vec{k}}\right) =
\left(S^{N,r,z}_{\vec{k}},W^{N,r,z}_{\vec{k}}\right)=\mathcal{F}_{z}\left(X^{N,r}_{\vec{k}},Y^{N,r}_{\vec{k}}\right),
\end{equation*}
\label{def:UV}
\end{definition} 

Equivalently,
\begin{align*}
U_{\vec{k}}=S^{N,r,z}_{\vec{k}}&=\binom{K}{\vec{k}}+(-1)^{z+K+k_r}~T^{N,r}_{\vec{k}},\\
V_{\vec{k}}=W^{N,r,z}_{\vec{k}}&=\binom{K}{\vec{k}}-(-1)^{z+K+k_r}~T^{N,r}_{\vec{k}}.
\end{align*}

In the case $M=2$, $r=1$, $K=i$, $k_1=j$ and $z=K+k_r=i+j$, 
$U_{\vec{k}}=S^N_{i,j}$ and $V_{\vec{k}}=W^N_{i,j}$.
This yields a {\it Pascal tiling} for $U_{\vec{k}}$ and $V_{\vec{k}}$ 
if and only if $N$ is prime according to Theorem~\ref{theo:PTT}. 
The case $r = 2$ leads to distinct decomposition. We have $z=K+k_r=i+(i-j)=2i-j, \quad (-1)^z=(-1)^j$, 
and the congruence properties modulo a prime $N$ are then identical along each column $j$. Hence, 
this does not yield a {\it Pascal tiling} for these $U_{\vec{k}}$ and $V_{\vec{k}}$ .

In the general case of Pascal simplices, a perfect alternation of the congruence 
properties corresponds to varying a single $k_l$ while $K$ changes by $\pm 1$.
The corresponding parametric linear map is $\mathcal{F}_z$ with $z=K$.
With this combination, the binomial case yields a Pascal tiling for $\left(U_{\vec{k}},V_{\vec{k}}\right)$
if and only if $N$ is prime, for both $r=1$ and $r=2$, provided that the corresponding matrices are 
expressed in terms of $(k_1,k_2)=(j,i-j)$ rather than $(i,j)=(k_1+k_2,k_1)$.

\begin{theorem}
Let $M \ge 2$, $N > 2$, $r \in \{1,\dots,M\}$, and
let
$\left(U_{\vec{k}},V_{\vec{k}}\right)
=\mathcal{F}_K\left(X^{N,r}_{\vec{k}},Y^{N,r}_{\vec{k}}\right)$, where 
$z=\mathcal{Z}\left(r,\vec{k}\right)=\sum_{l=1}^M k_l=K$.
Then the following statements are equivalent:
\begin{align}
(i)\quad&N \text{ is prime};\nonumber\\
(ii)\quad& \forall \vec{k} \in \mathcal{D}_M:\; \1{U_{\vec{k}}}=\frac{1+(-1)^K}{2};\label{eqn:PTTUK}\\
(iii)\quad& \forall \vec{k} \in \mathcal{D}_M:\; \1{V_{\vec{k}}}=\frac{1-(-1)^K}{2};\label{eqn:PTTVK}\\
(iv)\quad& \forall \vec{k} \in \mathcal{D}_M:\; \left|\1{U_{\vec{k}}}-\1{V_{\vec{k}}}\right|=1.\label{eqn:PTTUVK}
\end{align}
\label{theo:PTTUVK}
\end{theorem}
\begin{proof}
We follow the same line of argument  as in 
Theorem~\ref{theo:PTT}, after expressing the decomposition
corresponding to $z=K$.

By Definition~\ref{def:UV}, setting $z=K$,
\begin{align*}
\left(U_{\vec{k}},V_{\vec{k}}\right)&=\mathcal{F}_{K}\left(X^{N,r}_{\vec{k}},Y^{N,r}_{\vec{k}}\right)\\
U_{\vec{k}}=S^{N,r,K}_{\vec{k}}&=\binom{K}{\vec{k}}+(-1)^{k_r}~T^{N,r}_{\vec{k}},\\
V_{\vec{k}}=W^{N,r,K}_{\vec{k}}&=\binom{K}{\vec{k}}-(-1)^{k_r}~T^{N,r}_{\vec{k}}.
\end{align*}

For prime $N$, 
(\ref{eqn:PTTUK}) and (\ref{eqn:PTTVK})
follow from Proposition~\ref{prop:kmod},
while (\ref{eqn:PTTUVK}) follows
from 
(\ref{eqn:PTTUK}) and (\ref{eqn:PTTVK}).

To prove the converse, we reduce to the binomial case.
Fix $\vec{i}_s$ with $s\ne r$, so that $k_r=0$, $K=i$,
$\binom{K}{\vec{i}_s}=A^N_{i,0}$,
$T^{N,r}_{\vec{i}_s}=B^N_{i,0}$, and hence
\begin{align*}
U_{\vec{i}_s}&=A^N_{i,0}+B^N_{i,0}=S^N_{i,0},\\
V_{\vec{i}_s}&=A^N_{i,0}-B^N_{i,0}=W^N_{i,0}.
\end{align*}
Thus, the counterexamples from the proof of Theorem~\ref{theo:PTT} 
with $j=0$ also apply in the present setting, for all composite $N \ne 4$.

Counterexamples for $N=4$ are obtained from $\vec{k}_4$ as for Theorem~\ref{theo:PTTmultinomial}.
Namely, a vector $\vec{k}_4$ such that $K=2$, $k_r=1$, and exactly
one other component equals $1$, giving
\begin{align*}
S^{4,r,2}_{\vec{k}_4}&=\binom{2}{1,1}+(-1)^{1}\binom{4-1-1}{1,4-1-2}=0\equiv 0 \mymod{4},\\
W^{4,r,2}_{\vec{k}_4}&=\binom{2}{1,1}-(-1)^{1}\binom{4-1-1}{1,4-1-2}=4\equiv 0 \mymod{4},
\end{align*}
which shows that (\ref{eqn:PTTUVK}) fails in this case as well.
\end{proof}

Counterexamples where $N$ is a prime power (Theorem~\ref{theo:PTTmultinomial}) no longer arise. 
Consequently, Pascal tiling holds if and only if $N$ is prime. 
For the linear map $\mathcal{F}_z$ with $z = K$, this extends 
to both terms of the decomposition and to all values of $r$ (i.e.\ to the $M$ corresponding symmetries).

In the general case, the existence of a Pascal tiling for one term of the decomposition, for both terms, 
or for neither of them depends on the mapping $\mathcal{Z}$ and
the chosen representation (for instance $(k_1,\dots,k_M)$ or $(K,k_2,\dots,k_M)$).

By construction, the property
\[
\text{prime~}N~\Llra\forall\vec{k} \in \mathcal{D}_M:\; \left|\1{U_{\vec{k}}}-\1{V_{\vec{k}}}\right|=1
\]
remains valid for every decomposition $(U_{\vec{k}},V_{\vec{k}})$ obtained as the image 
under the map $\mathcal{F}_z$ with $z=\mathcal{Z}\left(r,\vec{k}\right)$, provided that 
$z \in \Z$ for all $\vec{k} \in \mathcal{D}_M$.
This is not necessarily the case for other maps, and additional constraints are generally 
required, as illustrated in the following counterexample.

\begin{definition}
Let $\vec{k} \in \mathcal{D}_M$. We define the decomposition 
\begin{equation*}
\left(U'_{\vec{k}},V'_{\vec{k}}\right)
=\left(\sum_{r=1}^M X^{N,r}_{\vec{k}},\sum_{r=1}^M Y^{N,r}_{\vec{k}}\right)
\end{equation*}
\label{def:XNveck}
\end{definition}
\begin{proposition}
Let $M > 2$ and set $N=M$. If $N$ is prime, then 
\begin{equation*}
\forall \vec{k} \in \mathcal{D}_M,\;
\left|\1{U'_{\vec{k}}}-\1{V'_{\vec{k}}}\right|=0.
\end{equation*}
\label{prop:XM}
\end{proposition}
\begin{proof}
According to Definitions~\ref{def:XNveck} and~\ref{def:XYnr}, and
Proposition~\ref{prop:kmod}, for odd prime $N=M$ and all $\vec{k} \in \mathcal{D}_M$, with $K=\sum_l k_l$,
\begin{align*}
U'_{\vec{k}}&\equiv \sum_{r=1}^M \binom{K}{\vec{k}}\left(1 - (-1)^{2K+2k_r}\right) \equiv 0 \mymod{N},\\[4pt]
V'_{\vec{k}}&\equiv \sum_{r=1}^M \binom{K}{\vec{k}}\left(1 + (-1)^{2K+2k_r}\right)  \equiv 2 M \binom{K}{\vec{k}} \equiv 0 \mymod{N},
\end{align*}
which proves the proposition.
\end{proof}

A more detailed analysis of linear combinations of the canonical
decompositions, as well as a generalization to negative multinomial
coefficients, which arises naturally from the expression of the
canonical decompositions, lies beyond the scope of this work and is
left for  future study.

\section{Conclusion}
\label{sec:conclusion}

We show that the symmetries underlying the congruence properties of
$M$-dimensional Pascal simplices arise from invariance properties of
multinomial coefficients of order $M+1$. This leads to several canonical
two-term decompositions of binomial and multinomial coefficients
(Theorem~\ref{theo:PTTmultinomial}).

The study of linear combinations of the terms in these canonical 
decompositions
shows that one of them, namely $z=\mathcal{Z}(r,\vec{k})=K$, makes it
possible to observe Pascal tilings for both terms 
$\left(U_{\vec{k}},V_{\vec{k}}\right)$ of the resulting decomposition,
for all $M\ge2$ and $r \in \{1,\dots,M\}$, leading to 
Theorem~\ref{theo:PTTUVK}.

Another decomposition, arising from  a simple geometrical construction
based on the first $N$ rows  of Pascal's triangle,
explains the failure of the converse of the property
$N \text{ prime}\Rightarrow L_N-1 \equiv0 \mymod{N}$.
For prime numbers, this equality can be expressed as
a sum of terms all congruent to $0 \mymod{N}$, whereas for Fibonacci
pseudoprimes, the sum remains congruent to $0 \mymod{N}$ but contains
terms that are not congruent to $0 \mymod{N}$.

It is possible to further analyze the
properties of $L_N-1$ modulo each prime factor and 
bound both the number of nonzero terms and the number
of distinct values of $S^N_{i,j} \mymod{N}$ using the prime factorization
of $N$ and Proposition~\ref{prop:fibozero}.
A detailed study of these aspects is left for future work.

Other applications are possible. For example, one may formulate
primality tests distinct from those directly derived from
(\ref{prop:N1j}) or its corollaries. 
Concretely, the matrices $X^N$ and $Y^N$ can be
partially constructed as soon as the first half of the rows of
Pascal's triangle is known, that is, from one quarter of the binomial
coefficients of the first $N$ rows. A detailed analysis of the
efficiency of such tests is beyond the scope of this work.

\bibliographystyle{plainurl}
\bibliography{DamierPascal}
\medskip
\noindent MSC2020: 11B65, 11A07, 11A41, 11B39 
\section{Appendix: Auxiliary Identities}
\label{sec:appendix}

This appendix contains supplementary results on the role of prime factors
of composite $N$ in the congruence properties of the matrices $S^N$ and $W^N$,
complementing the analysis developed in the main text, in particular Theorem~\ref{theo:pseudoprime}.

\begin{proposition}
Let $N > 2$ and let $(i,j) \in \mathcal{D}_{j \leq i}$ with $i-j \ge 2$ and $j \ge 1$. Then
\begin{equation*}
S^N_{i,j}=S^N_{i-1,j-1}+\sum_{n=0}^{j-1}S^N_{i-2-n,j-n}+W^N_{i-j-1,0}
\end{equation*}
\label{prop:Srecurrence}
\end{proposition}
\begin{proof}
\begin{align*}
S^N_{i,j}&=\binom{i}{j}+\binom{N-1-j}{N-1-i}\\
&=\binom{i-1}{j-1}+\binom{i-1}{j}+\binom{N-1-(j-1)}{N-1-(i-1)}-\binom{N-1-j}{N-1-(i-1)}\\
&=S^N_{i{-}1,j{-}1}+W^N_{i{-}1,j}\\[4pt]
W^N_{i{-}1,j}&=S^N_{i{-}2,j}+W^N_{i{-}2,j{-}1}\\[4pt]
S^N_{i,j}&=S^N_{i{-}1,j{-}1}+S^N_{i{-}2,j}+W^N_{i{-}2,j{-}1}
\end{align*}
Proposition~\ref{prop:Srecurrence} follows by induction, expanding $W^N_{i-2,j-1}$, $W^N_{i-3,j-2}$, 
$\dots$ until $W^N_{i-j-1,0}$.
\end{proof}

\begin{proposition}
Let $N = mq$ be an odd composite, where $m > 1$ and
$q \ge 3$ is the smallest prime factor of $N$. 
Let $(i,j) \in \mathcal{D}_{j \leq i}$ and let $0 \leq k \leq (q-3)/2$. Then
\begin{align}
i-j=2k+1 &\Longrightarrow S^N_{i,j} \equiv 0 \mymod{N},\label{eqn:fibomod1}\\
i-j=q&\Longrightarrow S^N_{i,j}\equiv m \mymod{N}.\label{eqn:fibomod2}
\end{align}
\label{prop:fibomod}
\end{proposition}
\begin{proof}
We first prove 
(\ref{eqn:fibomod1}), by induction.

\medskip

\noindent
\textbf{(i) Case $\mathbf{k=0}$} $\quad(i-j = 1,\quad 0 \leq j \leq N-2,\quad q \ge 3)$:
\begin{align*}
S^N_{1+j,j}&=\binom{j+1}{j}+\binom{N-1-j}{N-2-j}=j+1+N-1-j=N\\
&\equiv0 \mymod{N}
\end{align*}

\medskip

\noindent\textbf{(ii) Case $\mathbf{k=1}$} $\quad(i-j = 3,\quad 0 \leq j \leq N-4,\quad q \ge 5)$: 

By Proposition~\ref{prop:Srecurrence},
\begin{align*}
S^N_{3+j,j}&=S^N_{2+j,j-1}+\sum_{n=0}^{j-1}S^N_{3+j-2-n,j-n}+W^N_{3-1,0}\\
&=S^N_{2+j,j-1}+\sum_{n=0}^{j-1}S^N_{1+j-n,j-n}+W^N_{2,0}\\
\end{align*}
All terms $S^N_{1+j-n,j-n}$ satisfy $i'-j'=1$ and $k'=0$, hence are congruent to $0$ modulo $N$.
By (\ref{S4}), $W^N_{2,0}=1-\binom{mq-1}{2}\equiv1-(-1)^2\equiv 0 \mymod{N}$, and 
$S^N_{3,0}\equiv 1+\binom{mq-1}{3}\equiv 0 \mymod{N}$.

Thus, $S^N_{3+j,j}\equiv 0 \mymod{N}$ for $q \ge 5$ and all $0 \leq j \leq N-4$.

\medskip

\noindent\textbf{(iii) Case $\mathbf{k \ge 2}$} $\quad(i-j = 2k+1,\quad 0 \leq j \leq N-2-2k,\quad q \ge 7)$:

Assume that 
(\ref{eqn:fibomod1}) holds for $k'=k-1$. 
By Proposition~\ref{prop:Srecurrence},
\begin{equation*}
S^N_{2k+1+j,j}\equiv S^N_{2k+j,j-1}+W^N_{2k,0} \mymod{N}.
\end{equation*}

Arguing as for $k=1$, 
\begin{align*}
S^N_{2k+1,0}&\equiv0 \mymod{N},\\
W^N_{2k,0}&\equiv0 \mymod{N}.
\end{align*}
This proves (\ref{eqn:fibomod1}) by induction for all 
$2 \leq k \leq (q-3)/2$ and $0 \leq j \leq N-2k-2$.

(\ref{eqn:fibomod2}) corresponds to the case $i-j=q$.
From Proposition~\ref{prop:Srecurrence},
\begin{equation*}
S^N_{q+j,j}=S^N_{q+j-1,j-1}+\sum_{n=0}^{j-1}S^N_{q+j-2-n,j-n}+W^N_{q-1,0}.
\end{equation*}

By (\ref{eqn:fibomod1}), each term $S^N_{q+j-2-n,j-n}$ is congruent to $0$ modulo $N$.
Moreover, $W^N_{q-1,0}\equiv 0 \mymod N$, and by (\ref{S6}),
\[
S^N_{q,0}\equiv 1+m+(-1)^q\equiv m \mymod{N}.
\]

This proves (\ref{eqn:fibomod2}).
\end{proof}

\begin{proposition}
Let $N = mq > 3$ be a composite (odd or even), where $q$ is a prime factor of $N$.
If $m$ is odd, then
\[
S^{mq}_{i'_q,j'_q}\equiv m \mymod{mq},
\]
where
\[
i'_q = \frac{N+q}{2}-1, \qquad j'_q = \frac{N-q}{2}-1.
\]
\label{prop:composite}
\end{proposition}
\begin{proof}
Let $N=mq$ be composite, with $q\geq 2$. Since $m$ is odd, set
$H=(m-1)/2$.

Then
\[
(i'_q,j'_q)=\left((H+1)q-1,\,Hq-1\right),\qquad
i'_q+j'_q=mq-2.
\]

By (\ref{eqn:fibozero0}),
\begin{align*}
S^{mq}_{i'_q,j'_q}&=
\frac{mq}{Hq-1+1}\binom{mq-2-Hq+1}{Hq-1}\\
&=\frac{m}{H}\binom{Hq+(q-1)}{(H-1)q+(q-1)}.\\
\end{align*}

By Lucas's theorem, 
\[\binom{Hq+(q-1)}{(H-1)q+(q-1)}\equiv H \mymod{q}.\]

If $H\not\equiv 0 \mymod q$, then $H$ is invertible modulo $q$, hence
\[
S^{mq}_{i'_q,j'_q}\equiv m \mymod mq.
\]

If $H\equiv 0 \mymod{q}$, we use instead the identity
\begin{align*}
S^{mq}_{mq-2-j,j}&=\left(\frac{j+1}{mq-1-j}+1\right)\binom{mq-1-j}{j+1}\\
&=\frac{mq}{mq-1-j}\binom{mq-1-j}{j+1},\\
\end{align*}
which yields
\begin{align*}
S^{mq}_{i'_q,j'_q}&=
\frac{mq}{(2H+1)q-1-Hq+1}\binom{(2H+1)q-1-Hq+1}{Hq-1+1}\\
&=\frac{m}{(H+1)}\binom{(H+1)q}{Hq}.
\end{align*}

Again by Lucas's theorem,
we conclude that 
\[H+1\not\equiv 0 \mymod{q}\Rightarrow S^{mq}_{i'_q,j'_q}\equiv m \mymod{mq}.\]

Since $H$ and $H+1$ cannot be both congruent to $0$ modulo $q$, we get the expected 
property for all prime factors $q$ of a composite $N=mq$, if $m=N/q$ is odd.
\end{proof}

This property holds for all odd composite $N = mq$, since both $m$ and $q$ 
are odd in this case. In particular, it applies to Fibonacci pseudoprimes, 
which are necessarily odd composite numbers. Table~\ref{tab:fibo705} 
displays the corresponding $(i'_q, j'_q)$ for $N = 705$, the smallest 
such example.

In the more general case of an odd composite, 
Table~\ref{tab:annex} shows $S^{35}_{i,j}$, highlighting $(i_q,j_q)$,
$(i'_q,j'_q)$, the condition $i+j=N-2$, and the lines corresponding to 
Proposition~\ref{prop:fibomod}.
The table also shows that the corners locally form a Pascal tiling.

\section{Appendix: Prime Powers}
\label{sec:primepowers}

We now consider more specifically the cases for which the converse of \eqref{eqn:PTTY}
fails.
Table~\ref{tab:swxy-grid-9} provides the example $N=3^2$, for which
$Y^N_{i,j}\not\equiv 0 \mymod{N}$ for all $(i,j)\in\mathcal{D}_{j \leq i}$.
We show that this result holds for all prime powers $N=q^\alpha$ with $q\neq 2$.

\begin{proposition}
Let $N=q^\alpha$ with $q$ prime and $\alpha\in\mathbb{N}^*$, and let $(i,j)\in\mathcal{D}_{j \leq i}$. 
Then
\[
v_q\!\left(\binom{i}{j}\right)
=
v_q\!\left(\binom{N-1-j}{N-1-i}\right)<\alpha.
\]
\label{prop:vq}
\end{proposition}

\begin{proof}
By Kummer's theorem, the $q$-adic valuation $v_q$ of the binomial coefficient $\binom{i}{j}$
is equal to the number of carries when adding $j$ and $i-j$ in base $q$.

Since $i<q^\alpha$, its base-$q$ expansion involves at most $\alpha$ digits,
hence the number of carries is at most $\alpha-1$.

Write the base-$q$ expansions
\[
i=\sum_{k=0}^{\alpha-1} i_k q^k,
\qquad
j=\sum_{k=0}^{\alpha-1} j_k q^k.
\]
Then
\[
N-1-i=\sum_{k=0}^{\alpha-1} (q-1-i_k) q^k,
\qquad
N-1-j=\sum_{k=0}^{\alpha-1} (q-1-j_k) q^k.
\]

Let $d=i-j=\sum_{k=0}^{\alpha-1} d_k q^k$. Using the digit-sum formulation of Kummer's theorem,
\[
v_q\!\left(\binom{i}{j}\right)
=
\frac{\sum_k j_k+\sum_k d_k-\sum_k i_k}{q-1}.
\]

Similarly,
\begin{align*}
v_q\!\left(\binom{N-1-j}{N-1-i}\right)
&=
\frac{\sum_k (q-1-i_k)+\sum_k d_k-\sum_k (q-1-j_k)}{q-1}\\
&=
\frac{-\sum_k i_k+\sum_k d_k+\sum_k j_k}{q-1}\\
&=
v_q\!\left(\binom{i}{j}\right).
\qedhere
\end{align*}
\end{proof}

\begin{proposition}
Let $N = q^\alpha$, with $q > 2$ prime. Then
\[
\forall (i,j) \in \mathcal{D}_{j \leq i}, \quad
Y^N_{i,j}\not\equiv 0 \mymod{N}.
\]
\label{prop:Ynprimepower}
\end{proposition}

\begin{proof}
Let $\beta$  denotes the $q$-adic valuation of $\binom{i}{j}$. By Proposition~\ref{prop:vq}, 
we have $\beta < \alpha$.

We use congruence (2) in~\cite{granville_arithmetic_1997}, which gives
\begin{align*}
\frac{1}{q^\beta}\binom{i}{j}
&\equiv (-1)^\beta
\prod_{k=0}^{\alpha-1} \frac{i_k!}{j_k! d_k!}
\mymod{q}, \\[4pt]
\frac{1}{q^\beta}\binom{q^\alpha-1-j}{q^\alpha-1-i}
&\equiv (-1)^\beta
\prod_{k=0}^{\alpha-1}
\frac{(q-1-j_k)!}{(q-1-i_k)! d_k!}
\mymod{q}\\
&\equiv (-1)^\beta
\prod_{k=0}^{\alpha-1}
\binom{q-1-j_k}{q-1-i_k}
\mymod{q}.
\end{align*}

Using Proposition~\ref{prop:compomod} for terms $\binom{q-1-j_k}{q-1-i_k}$, we obtain
\[
\frac{1}{q^\beta}\binom{q^\alpha-1-j}{q^\alpha-1-i}
\equiv
\frac{1}{q^\beta}\binom{i}{j}~(-1)^{\sum_k (i_k+j_k)}
\mymod{q}.
\]

Since $q$ is odd, one checks that
\[
(-1)^{i+j} = (-1)^{\sum_k (i_k+j_k)}.
\]

Therefore, there exist $H_a, H_b \in \mathbb{Z}$ and 
$1 \le a, b \le q-1$ such that
\begin{align*}
\binom{i}{j} &= q^\beta(a + H_a q),\\
\binom{q^\alpha-1-j}{q^\alpha-1-i} &= q^\beta(b + H_b q),
\end{align*}
with
\[
a \equiv (-1)^{i+j} b \mymod{q}.
\]

We now consider
\begin{align*}
Y^{q^\alpha}_{i,j}
&= \binom{i}{j} + (-1)^{i+j}\binom{q^\alpha-1-j}{q^\alpha-1-i}\\
&= q^\beta \bigl(a + (-1)^{i+j} b\bigr)
+ q^{\beta+1} \bigl(H_a + (-1)^{i+j} H_b\bigr).
\end{align*}

We distinguish according to the parity of $i+j$.
If $i+j$ is odd, $a \equiv -b\mymod{q}$ and then $b=q-a$. 
If $i+j$ is even, then $a=b$.
In both cases,
\[
a + (-1)^{i+j} b \equiv 2a \not\equiv 0 \mymod{q},
\]
since $q\neq 2$ and $1 \le a \le q-1$.

Thus, the leading term
\[
q^\beta \bigl(a + (-1)^{i+j} b\bigr)
\]
is not divisible by $q^{\beta+1}$, hence
\[
v_q\!\left(Y^{q^\alpha}_{i,j}\right)=\beta < \alpha,
\]
which completes the proof.
\end{proof}

Kummer's theorem generalizes to multinomial coefficients, 
and can be used to prove that \eqref{eqn:PTTmultinomialY}
does not hold for all $M$ and all $N=q^\alpha$ with $q$ an odd prime.

\section{Appendix: Tables}

\begin{table}[ht]
\caption{Overview of the main definitions and results.\\}
\centering
\renewcommand{\arraystretch}{0.95}
\small
\begin{tabular}{|
>{\raggedright\arraybackslash}p{0.35\linewidth}|
>{\raggedright\arraybackslash}p{0.60\linewidth}|}
\hline

\vspace{0pt}
$
\begin{aligned}
&\textbf{Definition~\ref{def:SijWij}}.\quad(i,j)\in\mathcal{D}_{j\leq i}\\[4pt]
&S^N_{i,j}=\binom{i}{j}{+}\binom{N{-}1{-}j}{N{-}1{-}i}\\[4pt]
&W^N_{i,j}=\binom{i}{j}{-}\binom{N{-}1{-}j}{N{-}1{-}i}
\end{aligned}
$
&
\vspace{0pt}
$
\begin{aligned}
&\textbf{Theorem~\ref{theo:PTT}}.\quad N>2\\
&N~\text{is prime}\\
&\Llra\forall~(i,j):\; \1{S^N_{i,j}}=\dfrac{1+(-1)^{i+j}}{2}\\
&\Llra\forall~(i,j):\; \1{W^N_{i,j}}=\frac{1-(-1)^{i+j}}{2}\\
&\Llra\forall~(i,j):\; \left|\1{S^N_{i,j}}-\1{W^N_{i,j}}\right|=1\\[4pt]
\end{aligned}
$
\\\hline
\vspace{0pt}
$
\begin{aligned}
&\textbf{Equation~\ref{eq:PTTlucas}}\\
&\\
&\\
&L_N-1=\sum_{i+j=N-2}S^N_{i,j}\\
\end{aligned}
$
&
\vspace{0pt}
$
\begin{aligned}
&\textbf{Theorem~\ref{theo:LN}}.\quad N>2\\
&N~\text{is prime}\\
&\Lra\forall~(i,j)\mid i+j=N-2:S^N_{i,j}\equiv 0\mymod{N}\\
&\Lra L_N-1\equiv0\mymod{N}.\\[4pt]
&\textbf{Theorem~\ref{theo:pseudoprime}}.\\
&N~\text{is a Fibonacci pseudoprime}\\
&\Lra \exists~(i,j)\mid i+j=N-2: S^N_{i,j}\not\equiv0\mymod{N}\\[4pt]
\end{aligned}
$
\\\hline
\vspace{0pt}
$
\begin{aligned}
&\textbf{Definition~\ref{def:XijYijfun}}.\\
&X^N_{i,j}=\binom{i}{j}{-}(-1)^{i{+j}}\binom{N{-}1{-}j}{N{-}1{-}i}\\[4pt]
&Y^N_{i,j}=\binom{i}{j}{+}(-1)^{i{+}j}\binom{N{-}1{-}j}{N{-}1{-}i}\\[4pt]
\end{aligned}
$
&
\vspace{0pt}
$
\begin{aligned}
&\textbf{Theorem~\ref{theo:PTTXY}}.\quad N>2\\
& N~\text{is prime}\\
&\Llra\forall~(i,j):\;\1{X^N_{i,j}}=0\\
&\Lra\forall~(i,j):\;\1{Y^N_{i,j}}=1\\
&\Llra\forall~(i,j):\;\left|\1{X^N_{i,j}}-\1{Y^N_{i,j}}\right|=1\\[4pt]
\end{aligned}
$
\\\hline
\multicolumn{2}{|l|}{
\vspace{0pt}
$
\begin{aligned}
&\\[-8pt]
&\textbf{Definition~\ref{def:Tnr}}.\quad \vec{k}\in\mathcal{D}_M\\
&r\in\{1,\dots,M\}\\[4pt]
\end{aligned}
$
\hspace{1.5cm}
$
\begin{aligned}
&\\[-4pt]
T^{N,r}_{\vec{k}}=\binom{N-1-k_r}{k_1,\dots,k_{r-1},N-1-K,k_{r+1},\dots,k_M}\\[8pt]
\end{aligned}
$
} \\\hline
\vspace{0pt}
$
\begin{aligned}
&\textbf{Definition~\ref{def:XYnr}}\\
&\\
&X^{N,r}_{\vec{k}}=\binom{K}{\vec{k}}{-}(-1)^{K+k_r}~T^{N,r}_{\vec{k}}\\[4pt]
&Y^{N,r}_{\vec{k}}=\binom{K}{\vec{k}}{+}(-1)^{K+k_r}~T^{N,r}_{\vec{k}}\\
&
\end{aligned}
$
&
\vspace{0pt}
$
\begin{aligned}
&\textbf{Theorem~\ref{theo:PTTmultinomial}}.\quad M\ge2,~N>2\\
&N~\text{is prime}\\
&\Llra\forall~\vec{k}:\;\1{X^{N,r}_{\vec{k}}}=0\\
&\Lra\forall~\vec{k}:\;\1{Y^{N,r}_{\vec{k}}}=1\\
&\Llra\forall~\vec{k}:\;\left|\1{X^{N,r}_{\vec{k}}}-\1{Y^{N,r}_{\vec{k}}}\right|=1\\[4pt]
\end{aligned}
$
\\\hline
\multicolumn{2}{|l|}{
\vspace{0pt}
$
\begin{aligned}
&\\[-8pt]
&\textbf{Definition~\ref{def:UV}}.\quad z=\mathcal{Z}(r,\vec{k})\\
&\left(U_{\vec{k}},V_{\vec{k}}\right)=\mathcal{F}_z\left(X^{N,r}_{\vec{k}},Y^{N,r}_{\vec{k}}\right)\\[4pt]
\end{aligned}
$
\hspace{1.5cm}
$
\begin{aligned}
&\\[-4pt]
U_{\vec{k}}=\frac{1-(-1)^z}{2}X^{N,r}_{\vec{k}}+\frac{1+(-1)^z}{2}Y^{N,r}_{\vec{k}}\\[4pt]
V_{\vec{k}}=\frac{1+(-1)^z}{2}X^{N,r}_{\vec{k}}+\frac{1-(-1)^z}{2}Y^{N,r}_{\vec{k}}\\[4pt]
\end{aligned}
$
} \\\hline
\vspace{0pt}
$
\begin{aligned}
&\textbf{Definitions~\ref{def:XYnr} and~\ref{def:UV}.}\\
&z=K\\[4pt]
&U_{\vec{k}}=\binom{K}{\vec{k}}+(-1)^{k_r}T^{N,r}_{\vec{k}}\\[4pt]
&V_{\vec{k}}=\binom{K}{\vec{k}}-(-1)^{k_r}T^{N,r}_{\vec{k}}\\
&
\end{aligned}
$
&
\vspace{0pt}
$
\begin{aligned}
&\textbf{Theorem~\ref{theo:PTTUVK}}.\quad M\ge 2,~N>2\\
&N~\text{is prime}\\
&\Llra \forall~\vec{k}:\; \1{U_{\vec{k}}}=\frac{1+(-1)^K}{2}\\
&\Llra \forall~\vec{k}:\; \1{V_{\vec{k}}}=\frac{1-(-1)^K}{2}\\
&\Llra \forall~\vec{k}:\; \left|\1{U_{\vec{k}}}-\1{V_{\vec{k}}}\right|=1\\[4pt]
\end{aligned}
$
\\
\hline
\end{tabular}
\label{tab:summary}
\end{table}

\begin{table}[ht]
\caption{Example for $N = 11$ illustrating the links between Pascal's triangle, Fibonacci and Lucas numbers, and the terms $A^N_{i,j}$, $B^N_{i,j}$ and $S^N_{i,j}$. The lower-right table highlights the missing term (bold blue cell) in the summation of the $S^N_{i,j}$ with $i+j = N-2$, namely the first term in the decomposition of $F_{N+1}$, leading to (\ref{eq:PTTlucas}), which states that $L_N = F_{N+1}+F_{N-1} = 1+\sum S^N_{N-2-j,j} = 1+\sum A^N_{N-2-j,j}+\sum B^N_{N-2-j,j}$.}
\centering
\begin{minipage}[t][6cm]{0.4\textwidth}
\vspace{0pt}
\centering
\scriptsize
\setlength{\tabcolsep}{2pt}
\renewcommand{\arraystretch}{0.9}
\begin{tabular}{@{}r|rrrrrrrrrrrrr@{}}
$i$ \\
$0$ & 1 & \ \  & \ \  & \ \  & \ \  & \ \  & \ \  & \ \  & \ \  & \ \  & \ \  & \ \  & \ \  \\
$1$ & 1 & 1 & \ \  & \ \  & \ \  & \ \  & \ \  & \ \  & \ \  & \ \  & \ \  & \ \  & \ \  \\
$2$ & 1 & 2 & 1 & \ \  & \ \  & \ \  & \ \  & \ \  & \ \  & \ \  & \ \  & \ \  & \ \  \\
$3$ & 1 & 3 & 3 & 1 & \ \  & \ \  & \ \  & \ \  & \ \  & \ \  & \ \  & \ \  & \ \  \\
$4$ & 1 & 4 & 6 & 4 & 1 & \ \  & \ \  & \ \  & \ \  & \ \  & \ \  & \ \  & \ \  \\
$5$ & 1 & 5 & 10 & 10 & \rcelld{5} & 1 & \ \  & \ \  & \ \  & \ \  & \ \  & \ \  & \ \  \\
$6$ & 1 & 6 & 15 & \rcelld{20} & 15 & \rcelld{6} & 1 & \ \  & \ \  & \ \  & \ \  & \ \  & \ \  \\
$7$ & 1 & 7 & \rcelld{21} & 35 & \rcelld{35} & 21 & 7 & 1 & \ \  & \ \  & \ \  & \ \  & \ \  \\
$8$ & 1 & \rcelld{8} & 28 & \rcelld{56} & 70 & 56 & 28 & 8 & 1 & \ \  & \ \  & \ \  & \ \  \\
$9$ & \rcelld{1} & 9 & \rcelld{36} & 84 & 126 & 126 & 84 & 36 & 9 & 1 & \ \  & \ \  & \ \  \\
$10$ & 1 & \rcelld{10} & 45 & 120 & 210 & 252 & 210 & 120 & 45 & 10 & 1 & \ \  & \ \  \\
$11$ & \bcellD{1} & 11 & 55 & 165 & 330 & 462 & 462 & 330 & 165 & 55 & 11 & 1 & \ \  \\
$12$ & 1 & 12 & 66 & 220 & 495 & 792 & 924 & 792 & 495 & 220 & 66 & 12 & 1 \\
\hline
\multicolumn{1}{r}{$j$} & 0 & 1 & 2 & 3 & 4 & 5 & 6 & 7 & 8 & 9 & 10 & 11 & 12
\end{tabular}
\parbox[c][2.6em][c]{\linewidth}{\centering $\binom{i}{j}=\frac{i!}{j!(i-j)!}$}
\end{minipage}
\hfill
\begin{minipage}[t][6cm]{0.5\textwidth}
\vspace{0pt}
\centering
\scriptsize

\scriptsize
\setlength{\tabcolsep}{1.7pt}
\renewcommand{\arraystretch}{0.9}
\begin{tabular}{r|lcccccccccccccc}
$n$  \\
0 & $~F_0$&=&0 &=& 0 \\
1 & $~F_1$&=&1 &=& 1 \\
2 & $~F_2$&=&1 &=& 1 \\
3 & $~F_3$&=&2 &=& 1 &+& 1 \\
4 & $~F_4$&=&3 &=& 1 &+& 2 \\
5 & $~F_5$&=&5 &=& 1 &+& 3 &+& 1 \\
6 & $~F_6$&=&8 &=& 1 &+& 4 &+& 3 \\
7 & $~F_7$&=&13 &=& 1 &+& 5 &+& 6 &+& 1 \\
8 & $~F_8$&=&21 &=& 1 &+& 6 &+& 10 &+& 4 \\
9 & $~F_9$&=&34 &=& 1 &+& 7 &+& 15 &+& 10 &+& 1 \\
10 & $~F_{10}$&=&55 &=& \rcelld{1} &+& \rcelld{8} &+& \rcelld{21} &+& \rcelld{20} &+& \rcelld{5} \\
11 & $~F_{11}$&=&89 &=& 1 &+& 9 &+& 28 &+& 35 &+& 15 &+& 1 \\
12 & $~F_{12}$&=&144 &=& \bcellD{1} &+& \rcelld{10} &+& \rcelld{36} &+& \rcelld{56} &+& \rcelld{35} &+& \rcelld{6} \\
\hline
\multicolumn{5}{r}{$k$~~}& 0 && 1 && 2 && 3 && 4 && 5
\end{tabular}

\parbox[c][2.6em][c]{\linewidth}{\centering $F_n = \sum_{k=0}^{\infty}\binom{n-1-k}{k}$}
\end{minipage}

\begin{minipage}[t][6cm]{0.4\textwidth}
\vspace{0pt}
\centering
\scriptsize
\setlength{\tabcolsep}{2pt}
\renewcommand{\arraystretch}{0.9}
\begin{tabular}{@{}r|rrrrrrrrrrr@{}}
$i$ \\
$0$ & 1 & \ \  & \ \  & \ \  & \ \  & \ \  & \ \  & \ \  & \ \  & \ \  & \ \  \\
$1$ & 1 & 1 & \ \  & \ \  & \ \  & \ \  & \ \  & \ \  & \ \  & \ \  & \ \  \\
$2$ & 1 & 2 & 1 & \ \  & \ \  & \ \  & \ \  & \ \  & \ \  & \ \  & \ \  \\
$3$ & 1 & 3 & 3 & 1 & \ \  & \ \  & \ \  & \ \  & \ \  & \ \  & \ \  \\
$4$ & 1 & 4 & 6 & 4 & 1 & \ \  & \ \  & \ \  & \ \  & \ \  & \ \  \\
$5$ & 1 & 5 & 10 & 10 & \rcelld{5} & 1 & \ \  & \ \  & \ \  & \ \  & \ \  \\
$6$ & 1 & 6 & 15 & \rcelld{20} & 15 & 6 & 1 & \ \  & \ \  & \ \  & \ \  \\
$7$ & 1 & 7 & \rcelld{21} & 35 & 35 & 21 & 7 & 1 & \ \  & \ \  & \ \  \\
$8$ & 1 & \rcelld{8} & 28 & 56 & 70 & 56 & 28 & 8 & 1 & \ \  & \ \  \\
$9$ & \rcelld{1} & 9 & 36 & 84 & 126 & 126 & 84 & 36 & 9 & 1 & \ \  \\
$10$ & 1 & 10 & 45 & 120 & 210 & 252 & 210 & 120 & 45 & 10 & 1 \\
\hline
\multicolumn{1}{r}{$j$} & 0 & 1 & 2 & 3 & 4 & 5 & 6 & 7 & 8 & 9 & 10
\end{tabular}
\parbox[c][2.6em][c]{\linewidth}{\centering $A^{\hspace{0.02cm}11}_{i,j}$}
\end{minipage}
\hfill
\begin{minipage}[t][6cm]{0.5\textwidth}
\vspace{0pt}
\centering
\scriptsize
\setlength{\tabcolsep}{2pt}
\renewcommand{\arraystretch}{0.9}
\begin{tabular}{@{}r|rrrrrrrrrrr@{}}
$i$ \\
$0$ & 1 & \ \  & \ \  & \ \  & \ \  & \ \  & \ \  & \ \  & \ \  & \ \  & \ \  \\
$1$ & 10 & 1 & \ \  & \ \  & \ \  & \ \  & \ \  & \ \  & \ \  & \ \  & \ \  \\
$2$ & 45 & 9 & 1 & \ \  & \ \  & \ \  & \ \  & \ \  & \ \  & \ \  & \ \  \\
$3$ & 120 & 36 & 8 & 1 & \ \  & \ \  & \ \  & \ \  & \ \  & \ \  & \ \  \\
$4$ & 210 & 84 & 28 & 7 & 1 & \ \  & \ \  & \ \  & \ \  & \ \  & \ \  \\
$5$ & 252 & 126 & 56 & 21 & \rcelld{6} & 1 & \ \  & \ \  & \ \  & \ \  & \ \  \\
$6$ & 210 & 126 & 70 & \rcelld{35} & 15 & 5 & 1 & \ \  & \ \  & \ \  & \ \  \\
$7$ & 120 & 84 & \rcelld{56} & 35 & 20 & 10 & 4 & 1 & \ \  & \ \  & \ \  \\
$8$ & 45 & \rcelld{36} & 28 & 21 & 15 & 10 & 6 & 3 & 1 & \ \  & \ \  \\
$9$ & \rcelld{10} & 9 & 8 & 7 & 6 & 5 & 4 & 3 & 2 & 1 & \ \  \\
$10$ & 1 & 1 & 1 & 1 & 1 & 1 & 1 & 1 & 1 & 1 & 1 \\
\hline
\multicolumn{1}{r}{$j$} & 0 & 1 & 2 & 3 & 4 & 5 & 6 & 7 & 8 & 9 & 10
\end{tabular}
\parbox[c][2.6em][c]{\linewidth}{\centering $B^{\hspace{0.02cm}11}_{i,j}$}
\end{minipage}

\begin{minipage}[t][6cm]{0.4\textwidth}
\vspace{0pt}
\centering
\scriptsize
\setlength{\tabcolsep}{2pt}
\renewcommand{\arraystretch}{0.9}
\begin{tabular}{@{}r|rrrrrrrrrrr@{}}
$i$ \\
$0$ & 2 & \ \  & \ \  & \ \  & \ \  & \ \  & \ \  & \ \  & \ \  & \ \  & \ \  \\
$1$ & 11 & 2 & \ \  & \ \  & \ \  & \ \  & \ \  & \ \  & \ \  & \ \  & \ \  \\
$2$ & 46 & 11 & 2 & \ \  & \ \  & \ \  & \ \  & \ \  & \ \  & \ \  & \ \  \\
$3$ & 121 & 39 & 11 & 2 & \ \  & \ \  & \ \  & \ \  & \ \  & \ \  & \ \  \\
$4$ & 211 & 88 & 34 & 11 & 2 & \ \  & \ \  & \ \  & \ \  & \ \  & \ \  \\
$5$ & 253 & 131 & 66 & 31 & \rcellD{11} & 2 & \ \  & \ \  & \ \  & \ \  & \ \  \\
$6$ & 211 & 132 & 85 & \rcellD{55} & 30 & 11 & 2 & \ \  & \ \  & \ \  & \ \  \\
$7$ & 121 & 91 & \rcellD{77} & 70 & 55 & 31 & 11 & 2 & \ \  & \ \  & \ \  \\
$8$ & 46 & \rcellD{44} & 56 & 77 & 85 & 66 & 34 & 11 & 2 & \ \  & \ \  \\
$9$ & \rcellD{11} & 18 & 44 & 91 & 132 & 131 & 88 & 39 & 11 & 2 & \ \  \\
$10$ & 2 & 11 & 46 & 121 & 211 & 253 & 211 & 121 & 46 & 11 & 2 \\
\hline
\multicolumn{1}{r}{$j$} & 0 & 1 & 2 & 3 & 4 & 5 & 6 & 7 & 8 & 9 & 10
\end{tabular}
\parbox[c][2.6em][c]{\linewidth}{\centering $S^{\hspace{0.02cm}11}_{i,j}=A^{\hspace{0.02cm}11}_{i,j}+B^{\hspace{0.02cm}11}_{i,j}$}
\end{minipage}
\hfill
\begin{minipage}[t][6cm]{0.5\textwidth}
\vspace{0pt}
\centering
\scriptsize

\scriptsize
\setlength{\tabcolsep}{1pt}
\renewcommand{\arraystretch}{0.9}
\begin{tabular}{rlccccccccccccccccc}
$L_{11}$&=& &$F_{10}$&=& & & &\rcelld{1} &+& \rcelld{8} &+& \rcelld{21} &+& \rcelld{20} &+& \rcelld{5} \\ \\
      & &+&$F_{12}$& &+~&\bcellD{1} &+& \rcelld{10} &+& \rcelld{36} &+& \rcelld{56} &+& \rcelld{35} &+& \rcelld{6} \\ \\ \hline\\

$L_{11}$&=&\multicolumn{2}{r}{$1+\sum_{i+j=9}S^{11}_{i,j}$}&=&& 1 &+&\rcellD{11} &+& \rcellD{44} &+& \rcellD{77} &+& \rcellD{55} &+& \rcellD{11} \\ \\ \hline\\
$L_{11}$&=& &1                         &=& &1& \\ \\
      & &+&$\sum_{i+j=9}A^{11}_{i,j}$& &+~& & &\rcelld{1} &+& \rcelld{8} &+& \rcelld{21} &+& \rcelld{20} &+& \rcelld{5}\\ \\
      & &+&$\sum_{i+j=9}B^{11}_{i,j}$& &+~& & &\rcelld{10} &+& \rcelld{36} &+& \rcelld{56} &+& \rcelld{35} &+& \rcelld{6}
\end{tabular}

\parbox[c][2.6em][c]{\linewidth}{}
\end{minipage}
\label{tab:fibo-grid}
\end{table}

\begin{table}[ht]
\centering
\tbl{
Values of $j$, $\gcd(j{+}1,N)$, and $S^N_{i,j}\mymod{N}$, where $i = N-2-j$, for $N = 705$, the smallest Fibonacci pseudoprime. Although $705 = 3\times5\times47$ is composite, it satisfies $L_{705} - 1 \equiv 0 \mymod{705}$, as do all prime numbers (Theorem~\ref{theo:LN}). If $\gcd(j{+}1,N)=1$ (for example, $j=135$, {\bf Green}), then $S^N_{N-2-j,j} \equiv 0 \mymod{N}$ (Proposition~\ref{prop:fibozero}), whereas the converse fails (for example, $j = 101$, {\bf Blue}). The cases $j = 2$, $j = 4$, $j = 46$ ,  $j = 350$, $j = 349$, $j = 328$ (i.e., $j = q-1$ and $j = (N-q)/2-1$ where $q$ is one of the prime factors of $N$) correspond to $(i_q,j_q)$ and  $(i'_q,j'_q)$ used in the proof of Theorem~\ref{theo:pseudoprime}, and are example   for which $\gcd(j{+}1)\ne 1$ and $S^N_{i,j}\not\equiv 0 \mymod{N}$ ({\bf Yellow}). 
}
{
\tiny
\setlength{\tabcolsep}{3pt}
\renewcommand{\arraystretch}{1.15}

\begin{tabular}{lcccccccccccccccccccc}
\toprule

{\boldmath $j$} & \bf\gcell{0} & \bf\gcell{1} & \bf\ycellD{2} & \bf\gcell{3} & \bf\ycellD{4} & \bf\ycell{5} & \bf\gcell{6} & \bf\gcell{7} & \bf\ycell{8} & \bf\ycell{9} & \bf\gcell{10} & \bf\ycell{11} & \bf\gcell{12} & \bf\gcell{13} & \bf\ycell{14} & \bf\gcell{15} & \bf\gcell{16} & \bf\bcell{17} & \bf\gcell{18} & \bf\bcell{19} \\
{\boldmath $\gcd(j{+}1,N)$} & \gcell{1} & \gcell{1} & \ycellD{3} & \gcell{1} & \ycellD{5} & \ycell{3} & \gcell{1} & \gcell{1} & \ycell{3} & \ycell{5} & \gcell{1} & \ycell{3} & \gcell{1} & \gcell{1} & \ycell{15} & \gcell{1} & \gcell{1} & \bcell{3} & \gcell{1} & \bcell{5} \\
{\boldmath $S_{i,j}\bmod N$} & \gcell{0} & \gcell{0} & \ycellD{235} & \gcell{0} & \ycellD{141} & \ycell{470} & \gcell{0} & \gcell{0} & \ycell{470} & \ycell{564} & \gcell{0} & \ycell{235} & \gcell{0} & \gcell{0} & \ycell{47} & \gcell{0} & \gcell{0} & \bcell{0} & \gcell{0} & \bcell{0} \\

\midrule

{\boldmath $j$} & \bf\bcell{20} & \bf\gcell{21} & \bf\gcell{22} & \bf\bcell{23} & \bf\ycell{24} & \bf\gcell{25} & \bf\ycell{26} & \bf\gcell{27} & \bf\gcell{28} & \bf\ycell{29} & \bf\gcell{30} & \bf\gcell{31} & \bf\bcell{32} & \bf\gcell{33} & \bf\ycell{34} & \bf\bcell{35} & \bf\gcell{36} & \bf\gcell{37} & \bf\bcell{38} & \bf\ycell{39} \\
{\boldmath $\gcd(j{+}1,N)$} & \bcell{3} & \gcell{1} & \gcell{1} & \bcell{3} & \ycell{5} & \gcell{1} & \ycell{3} & \gcell{1} & \gcell{1} & \ycell{15} & \gcell{1} & \gcell{1} & \bcell{3} & \gcell{1} & \ycell{5} & \bcell{3} & \gcell{1} & \gcell{1} & \bcell{3} & \ycell{5} \\
{\boldmath $S_{i,j}\bmod N$} & \bcell{0} & \gcell{0} & \gcell{0} & \bcell{0} & \ycell{282} & \gcell{0} & \ycell{235} & \gcell{0} & \gcell{0} & \ycell{141} & \gcell{0} & \gcell{0} & \bcell{0} & \gcell{0} & \ycell{564} & \bcell{0} & \gcell{0} & \gcell{0} & \bcell{0} & \ycell{282} \\

\midrule

{\boldmath $j$} & \bf\gcell{40} & \bf\bcell{41} & \bf\gcell{42} & \bf\gcell{43} & \bf\bcell{44} & \bf\gcell{45} & \bf\ycellD{46} & \bf\bcell{47} & \bf\gcell{48} & \bf\bcell{49} & \bf\bcell{50} & \bf\gcell{51} & \bf\gcell{52} & \bf\bcell{53} & \bf\bcell{54} & \bf\gcell{55} & \bf\ycell{56} & \bf\gcell{57} & \bf\gcell{58} & \bf\ycell{59} \\
{\boldmath $\gcd(j{+}1,N)$} & \gcell{1} & \bcell{3} & \gcell{1} & \gcell{1} & \bcell{15} & \gcell{1} & \ycellD{47} & \bcell{3} & \gcell{1} & \bcell{5} & \bcell{3} & \gcell{1} & \gcell{1} & \bcell{3} & \bcell{5} & \gcell{1} & \ycell{3} & \gcell{1} & \gcell{1} & \ycell{15} \\
{\boldmath $S_{i,j}\bmod N$} & \gcell{0} & \bcell{0} & \gcell{0} & \gcell{0} & \bcell{0} & \gcell{0} & \ycellD{15} & \bcell{0} & \gcell{0} & \bcell{0} & \bcell{0} & \gcell{0} & \gcell{0} & \bcell{0} & \bcell{0} & \gcell{0} & \ycell{235} & \gcell{0} & \gcell{0} & \ycell{470} \\

\midrule

{\boldmath $j$} & \bf\gcell{60} & \bf\gcell{61} & \bf\ycell{62} & \bf\gcell{63} & \bf\bcell{64} & \bf\ycell{65} & \bf\gcell{66} & \bf\gcell{67} & \bf\ycell{68} & \bf\bcell{69} & \bf\gcell{70} & \bf\bcell{71} & \bf\gcell{72} & \bf\gcell{73} & \bf\bcell{74} & \bf\gcell{75} & \bf\gcell{76} & \bf\bcell{77} & \bf\gcell{78} & \bf\ycell{79} \\
{\boldmath $\gcd(j{+}1,N)$} & \gcell{1} & \gcell{1} & \ycell{3} & \gcell{1} & \bcell{5} & \ycell{3} & \gcell{1} & \gcell{1} & \ycell{3} & \bcell{5} & \gcell{1} & \bcell{3} & \gcell{1} & \gcell{1} & \bcell{15} & \gcell{1} & \gcell{1} & \bcell{3} & \gcell{1} & \ycell{5} \\
{\boldmath $S_{i,j}\bmod N$} & \gcell{0} & \gcell{0} & \ycell{470} & \gcell{0} & \bcell{0} & \ycell{235} & \gcell{0} & \gcell{0} & \ycell{470} & \bcell{0} & \gcell{0} & \bcell{0} & \gcell{0} & \gcell{0} & \bcell{0} & \gcell{0} & \gcell{0} & \bcell{0} & \gcell{0} & \ycell{564} \\

\midrule

{\boldmath $j$} & \bf\ycell{80} & \bf\gcell{81} & \bf\gcell{82} & \bf\ycell{83} & \bf\ycell{84} & \bf\gcell{85} & \bf\ycell{86} & \bf\gcell{87} & \bf\gcell{88} & \bf\ycell{89} & \bf\gcell{90} & \bf\gcell{91} & \bf\ycell{92} & \bf\ycell{93} & \bf\bcell{94} & \bf\ycell{95} & \bf\gcell{96} & \bf\gcell{97} & \bf\bcell{98} & \bf\ycell{99} \\
{\boldmath $\gcd(j{+}1,N)$} & \ycell{3} & \gcell{1} & \gcell{1} & \ycell{3} & \ycell{5} & \gcell{1} & \ycell{3} & \gcell{1} & \gcell{1} & \ycell{15} & \gcell{1} & \gcell{1} & \ycell{3} & \ycell{47} & \bcell{5} & \ycell{3} & \gcell{1} & \gcell{1} & \bcell{3} & \ycell{5} \\
{\boldmath $S_{i,j}\bmod N$} & \ycell{235} & \gcell{0} & \gcell{0} & \ycell{235} & \ycell{141} & \gcell{0} & \ycell{470} & \gcell{0} & \gcell{0} & \ycell{188} & \gcell{0} & \gcell{0} & \ycell{235} & \ycell{90} & \bcell{0} & \ycell{470} & \gcell{0} & \gcell{0} & \bcell{0} & \ycell{141} \\

\midrule

{\boldmath $j$} & \bf\gcell{100} & \bf\bcellD{101} & \bf\gcell{102} & \bf\gcell{103} & \bf\bcell{104} & \bf\gcell{105} & \bf\gcell{106} & \bf\ycell{107} & \bf\gcell{108} & \bf\bcell{109} & \bf\bcell{110} & \bf\gcell{111} & \bf\gcell{112} & \bf\bcell{113} & \bf\bcell{114} & \bf\gcell{115} & \bf\bcell{116} & \bf\gcell{117} & \bf\gcell{118} & \bf\bcell{119} \\
{\boldmath $\gcd(j{+}1,N)$} & \gcell{1} & \bcellD{3} & \gcell{1} & \gcell{1} & \bcell{15} & \gcell{1} & \gcell{1} & \ycell{3} & \gcell{1} & \bcell{5} & \bcell{3} & \gcell{1} & \gcell{1} & \bcell{3} & \bcell{5} & \gcell{1} & \bcell{3} & \gcell{1} & \gcell{1} & \bcell{15} \\
{\boldmath $S_{i,j}\bmod N$} & \gcell{0} & \bcellD{0} & \gcell{0} & \gcell{0} & \bcell{0} & \gcell{0} & \gcell{0} & \ycell{235} & \gcell{0} & \bcell{0} & \bcell{0} & \gcell{0} & \gcell{0} & \bcell{0} & \bcell{0} & \gcell{0} & \bcell{0} & \gcell{0} & \gcell{0} & \bcell{0} \\

\midrule

{\boldmath $j$} & \bf\gcell{120} & \bf\gcell{121} & \bf\bcell{122} & \bf\gcell{123} & \bf\ycell{124} & \bf\bcell{125} & \bf\gcell{126} & \bf\gcell{127} & \bf\bcell{128} & \bf\ycell{129} & \bf\gcell{130} & \bf\bcell{131} & \bf\gcell{132} & \bf\gcell{133} & \bf\ycell{134} & \bf\gcellD{135} & \bf\gcell{136} & \bf\bcell{137} & \bf\gcell{138} & \bf\ycell{139} \\
{\boldmath $\gcd(j{+}1,N)$} & \gcell{1} & \gcell{1} & \bcell{3} & \gcell{1} & \ycell{5} & \bcell{3} & \gcell{1} & \gcell{1} & \bcell{3} & \ycell{5} & \gcell{1} & \bcell{3} & \gcell{1} & \gcell{1} & \ycell{15} & \gcellD{1} & \gcell{1} & \bcell{3} & \gcell{1} & \ycell{5} \\
{\boldmath $S_{i,j}\bmod N$} & \gcell{0} & \gcell{0} & \bcell{0} & \gcell{0} & \ycell{564} & \bcell{0} & \gcell{0} & \gcell{0} & \bcell{0} & \ycell{564} & \gcell{0} & \bcell{0} & \gcell{0} & \gcell{0} & \ycell{141} & \gcellD{0} & \gcell{0} & \bcell{0} & \gcell{0} & \ycell{423} \\

\midrule

{\boldmath $j$} & \bf\ycell{140} & \bf\gcell{141} & \bf\gcell{142} & \bf\bcell{143} & \bf\bcell{144} & \bf\gcell{145} & \bf\bcell{146} & \bf\gcell{147} & \bf\gcell{148} & \bf\ycell{149} & \bf\gcell{150} & \bf\gcell{151} & \bf\bcell{152} & \bf\gcell{153} & \bf\ycell{154} & \bf\bcell{155} & \bf\gcell{156} & \bf\gcell{157} & \bf\bcell{158} & \bf\ycell{159} \\
{\boldmath $\gcd(j{+}1,N)$} & \ycell{141} & \gcell{1} & \gcell{1} & \bcell{3} & \bcell{5} & \gcell{1} & \bcell{3} & \gcell{1} & \gcell{1} & \ycell{15} & \gcell{1} & \gcell{1} & \bcell{3} & \gcell{1} & \ycell{5} & \bcell{3} & \gcell{1} & \gcell{1} & \bcell{3} & \ycell{5} \\
{\boldmath $S_{i,j}\bmod N$} & \ycell{510} & \gcell{0} & \gcell{0} & \bcell{0} & \bcell{0} & \gcell{0} & \bcell{0} & \gcell{0} & \gcell{0} & \ycell{423} & \gcell{0} & \gcell{0} & \bcell{0} & \gcell{0} & \ycell{564} & \bcell{0} & \gcell{0} & \gcell{0} & \bcell{0} & \ycell{141} \\

\midrule

{\boldmath $j$} & \bf\gcell{160} & \bf\bcell{161} & \bf\gcell{162} & \bf\gcell{163} & \bf\ycell{164} & \bf\gcell{165} & \bf\gcell{166} & \bf\bcell{167} & \bf\gcell{168} & \bf\bcell{169} & \bf\bcell{170} & \bf\gcell{171} & \bf\gcell{172} & \bf\bcell{173} & \bf\bcell{174} & \bf\gcell{175} & \bf\bcell{176} & \bf\gcell{177} & \bf\gcell{178} & \bf\bcell{179} \\
{\boldmath $\gcd(j{+}1,N)$} & \gcell{1} & \bcell{3} & \gcell{1} & \gcell{1} & \ycell{15} & \gcell{1} & \gcell{1} & \bcell{3} & \gcell{1} & \bcell{5} & \bcell{3} & \gcell{1} & \gcell{1} & \bcell{3} & \bcell{5} & \gcell{1} & \bcell{3} & \gcell{1} & \gcell{1} & \bcell{15} \\
{\boldmath $S_{i,j}\bmod N$} & \gcell{0} & \bcell{0} & \gcell{0} & \gcell{0} & \ycell{423} & \gcell{0} & \gcell{0} & \bcell{0} & \gcell{0} & \bcell{0} & \bcell{0} & \gcell{0} & \gcell{0} & \bcell{0} & \bcell{0} & \gcell{0} & \bcell{0} & \gcell{0} & \gcell{0} & \bcell{0} \\

\midrule

{\boldmath $j$} & \bf\gcell{180} & \bf\gcell{181} & \bf\bcell{182} & \bf\gcell{183} & \bf\bcell{184} & \bf\bcell{185} & \bf\gcell{186} & \bf\ycell{187} & \bf\bcell{188} & \bf\bcell{189} & \bf\gcell{190} & \bf\bcell{191} & \bf\gcell{192} & \bf\gcell{193} & \bf\bcell{194} & \bf\gcell{195} & \bf\gcell{196} & \bf\bcell{197} & \bf\gcell{198} & \bf\bcell{199} \\
{\boldmath $\gcd(j{+}1,N)$} & \gcell{1} & \gcell{1} & \bcell{3} & \gcell{1} & \bcell{5} & \bcell{3} & \gcell{1} & \ycell{47} & \bcell{3} & \bcell{5} & \gcell{1} & \bcell{3} & \gcell{1} & \gcell{1} & \bcell{15} & \gcell{1} & \gcell{1} & \bcell{3} & \gcell{1} & \bcell{5} \\
{\boldmath $S_{i,j}\bmod N$} & \gcell{0} & \gcell{0} & \bcell{0} & \gcell{0} & \bcell{0} & \bcell{0} & \gcell{0} & \ycell{450} & \bcell{0} & \bcell{0} & \gcell{0} & \bcell{0} & \gcell{0} & \gcell{0} & \bcell{0} & \gcell{0} & \gcell{0} & \bcell{0} & \gcell{0} & \bcell{0} \\

\midrule

{\boldmath $j$} & \bf\bcell{200} & \bf\gcell{201} & \bf\gcell{202} & \bf\bcell{203} & \bf\ycell{204} & \bf\gcell{205} & \bf\bcell{206} & \bf\gcell{207} & \bf\gcell{208} & \bf\ycell{209} & \bf\gcell{210} & \bf\gcell{211} & \bf\bcell{212} & \bf\gcell{213} & \bf\ycell{214} & \bf\bcell{215} & \bf\gcell{216} & \bf\gcell{217} & \bf\ycell{218} & \bf\bcell{219} \\
{\boldmath $\gcd(j{+}1,N)$} & \bcell{3} & \gcell{1} & \gcell{1} & \bcell{3} & \ycell{5} & \gcell{1} & \bcell{3} & \gcell{1} & \gcell{1} & \ycell{15} & \gcell{1} & \gcell{1} & \bcell{3} & \gcell{1} & \ycell{5} & \bcell{3} & \gcell{1} & \gcell{1} & \ycell{3} & \bcell{5} \\
{\boldmath $S_{i,j}\bmod N$} & \bcell{0} & \gcell{0} & \gcell{0} & \bcell{0} & \ycell{282} & \gcell{0} & \bcell{0} & \gcell{0} & \gcell{0} & \ycell{423} & \gcell{0} & \gcell{0} & \bcell{0} & \gcell{0} & \ycell{564} & \bcell{0} & \gcell{0} & \gcell{0} & \ycell{235} & \bcell{0} \\

\midrule

{\boldmath $j$} & \bf\gcell{220} & \bf\ycell{221} & \bf\gcell{222} & \bf\gcell{223} & \bf\ycell{224} & \bf\gcell{225} & \bf\gcell{226} & \bf\ycell{227} & \bf\gcell{228} & \bf\bcell{229} & \bf\ycell{230} & \bf\gcell{231} & \bf\gcell{232} & \bf\bcell{233} & \bf\ycell{234} & \bf\gcell{235} & \bf\bcell{236} & \bf\gcell{237} & \bf\gcell{238} & \bf\bcell{239} \\
{\boldmath $\gcd(j{+}1,N)$} & \gcell{1} & \ycell{3} & \gcell{1} & \gcell{1} & \ycell{15} & \gcell{1} & \gcell{1} & \ycell{3} & \gcell{1} & \bcell{5} & \ycell{3} & \gcell{1} & \gcell{1} & \bcell{3} & \ycell{235} & \gcell{1} & \bcell{3} & \gcell{1} & \gcell{1} & \bcell{15} \\
{\boldmath $S_{i,j}\bmod N$} & \gcell{0} & \ycell{470} & \gcell{0} & \gcell{0} & \ycell{188} & \gcell{0} & \gcell{0} & \ycell{235} & \gcell{0} & \bcell{0} & \ycell{470} & \gcell{0} & \gcell{0} & \bcell{0} & \ycell{660} & \gcell{0} & \bcell{0} & \gcell{0} & \gcell{0} & \bcell{0} \\

\midrule

{\boldmath $j$} & \bf\gcell{240} & \bf\gcell{241} & \bf\ycell{242} & \bf\gcell{243} & \bf\bcell{244} & \bf\ycell{245} & \bf\gcell{246} & \bf\gcell{247} & \bf\ycell{248} & \bf\ycell{249} & \bf\gcell{250} & \bf\ycell{251} & \bf\gcell{252} & \bf\gcell{253} & \bf\ycell{254} & \bf\gcell{255} & \bf\gcell{256} & \bf\ycell{257} & \bf\gcell{258} & \bf\ycell{259} \\
{\boldmath $\gcd(j{+}1,N)$} & \gcell{1} & \gcell{1} & \ycell{3} & \gcell{1} & \bcell{5} & \ycell{3} & \gcell{1} & \gcell{1} & \ycell{3} & \ycell{5} & \gcell{1} & \ycell{3} & \gcell{1} & \gcell{1} & \ycell{15} & \gcell{1} & \gcell{1} & \ycell{3} & \gcell{1} & \ycell{5} \\
{\boldmath $S_{i,j}\bmod N$} & \gcell{0} & \gcell{0} & \ycell{235} & \gcell{0} & \bcell{0} & \ycell{235} & \gcell{0} & \gcell{0} & \ycell{470} & \ycell{423} & \gcell{0} & \ycell{470} & \gcell{0} & \gcell{0} & \ycell{658} & \gcell{0} & \gcell{0} & \ycell{470} & \gcell{0} & \ycell{282} \\

\midrule

{\boldmath $j$} & \bf\bcell{260} & \bf\gcell{261} & \bf\gcell{262} & \bf\bcell{263} & \bf\ycell{264} & \bf\gcell{265} & \bf\bcell{266} & \bf\gcell{267} & \bf\gcell{268} & \bf\ycell{269} & \bf\gcell{270} & \bf\gcell{271} & \bf\bcell{272} & \bf\gcell{273} & \bf\ycell{274} & \bf\bcell{275} & \bf\gcell{276} & \bf\gcell{277} & \bf\bcell{278} & \bf\ycell{279} \\
{\boldmath $\gcd(j{+}1,N)$} & \bcell{3} & \gcell{1} & \gcell{1} & \bcell{3} & \ycell{5} & \gcell{1} & \bcell{3} & \gcell{1} & \gcell{1} & \ycell{15} & \gcell{1} & \gcell{1} & \bcell{3} & \gcell{1} & \ycell{5} & \bcell{3} & \gcell{1} & \gcell{1} & \bcell{3} & \ycell{5} \\
{\boldmath $S_{i,j}\bmod N$} & \bcell{0} & \gcell{0} & \gcell{0} & \bcell{0} & \ycell{141} & \gcell{0} & \bcell{0} & \gcell{0} & \gcell{0} & \ycell{235} & \gcell{0} & \gcell{0} & \bcell{0} & \gcell{0} & \ycell{141} & \bcell{0} & \gcell{0} & \gcell{0} & \bcell{0} & \ycell{423} \\

\midrule

{\boldmath $j$} & \bf\gcell{280} & \bf\ycell{281} & \bf\gcell{282} & \bf\gcell{283} & \bf\ycell{284} & \bf\gcell{285} & \bf\gcell{286} & \bf\bcell{287} & \bf\gcell{288} & \bf\ycell{289} & \bf\bcell{290} & \bf\gcell{291} & \bf\gcell{292} & \bf\bcell{293} & \bf\bcell{294} & \bf\gcell{295} & \bf\bcell{296} & \bf\gcell{297} & \bf\gcell{298} & \bf\ycell{299} \\
{\boldmath $\gcd(j{+}1,N)$} & \gcell{1} & \ycell{141} & \gcell{1} & \gcell{1} & \ycell{15} & \gcell{1} & \gcell{1} & \bcell{3} & \gcell{1} & \ycell{5} & \bcell{3} & \gcell{1} & \gcell{1} & \bcell{3} & \bcell{5} & \gcell{1} & \bcell{3} & \gcell{1} & \gcell{1} & \ycell{15} \\
{\boldmath $S_{i,j}\bmod N$} & \gcell{0} & \ycell{375} & \gcell{0} & \gcell{0} & \ycell{282} & \gcell{0} & \gcell{0} & \bcell{0} & \gcell{0} & \ycell{141} & \bcell{0} & \gcell{0} & \gcell{0} & \bcell{0} & \bcell{0} & \gcell{0} & \bcell{0} & \gcell{0} & \gcell{0} & \ycell{235} \\

\midrule

{\boldmath $j$} & \bf\gcell{300} & \bf\gcell{301} & \bf\ycell{302} & \bf\gcell{303} & \bf\bcell{304} & \bf\ycell{305} & \bf\gcell{306} & \bf\gcell{307} & \bf\ycell{308} & \bf\bcell{309} & \bf\gcell{310} & \bf\ycell{311} & \bf\gcell{312} & \bf\gcell{313} & \bf\bcell{314} & \bf\gcell{315} & \bf\gcell{316} & \bf\bcell{317} & \bf\gcell{318} & \bf\bcell{319} \\
{\boldmath $\gcd(j{+}1,N)$} & \gcell{1} & \gcell{1} & \ycell{3} & \gcell{1} & \bcell{5} & \ycell{3} & \gcell{1} & \gcell{1} & \ycell{3} & \bcell{5} & \gcell{1} & \ycell{3} & \gcell{1} & \gcell{1} & \bcell{15} & \gcell{1} & \gcell{1} & \bcell{3} & \gcell{1} & \bcell{5} \\
{\boldmath $S_{i,j}\bmod N$} & \gcell{0} & \gcell{0} & \ycell{470} & \gcell{0} & \bcell{0} & \ycell{470} & \gcell{0} & \gcell{0} & \ycell{235} & \bcell{0} & \gcell{0} & \ycell{470} & \gcell{0} & \gcell{0} & \bcell{0} & \gcell{0} & \gcell{0} & \bcell{0} & \gcell{0} & \bcell{0} \\

\midrule

{\boldmath $j$} & \bf\bcell{320} & \bf\gcell{321} & \bf\gcell{322} & \bf\ycell{323} & \bf\bcell{324} & \bf\gcell{325} & \bf\ycell{326} & \bf\gcell{327} & \bf\ycellD{328} & \bf\ycell{329} & \bf\gcell{330} & \bf\gcell{331} & \bf\ycell{332} & \bf\gcell{333} & \bf\ycell{334} & \bf\ycell{335} & \bf\gcell{336} & \bf\gcell{337} & \bf\ycell{338} & \bf\ycell{339} \\
{\boldmath $\gcd(j{+}1,N)$} & \bcell{3} & \gcell{1} & \gcell{1} & \ycell{3} & \bcell{5} & \gcell{1} & \ycell{3} & \gcell{1} & \ycellD{47} & \ycell{15} & \gcell{1} & \gcell{1} & \ycell{3} & \gcell{1} & \ycell{5} & \ycell{3} & \gcell{1} & \gcell{1} & \ycell{3} & \ycell{5} \\
{\boldmath $S_{i,j}\bmod N$} & \bcell{0} & \gcell{0} & \gcell{0} & \ycell{235} & \bcell{0} & \gcell{0} & \ycell{235} & \gcell{0} & \ycellD{15} & \ycell{329} & \gcell{0} & \gcell{0} & \ycell{470} & \gcell{0} & \ycell{141} & \ycell{235} & \gcell{0} & \gcell{0} & \ycell{470} & \ycell{423} \\

\midrule

{\boldmath $j$} & \bf\gcell{340} & \bf\bcell{341} & \bf\gcell{342} & \bf\gcell{343} & \bf\bcell{344} & \bf\gcell{345} & \bf\gcell{346} & \bf\bcell{347} & \bf\gcell{348} & \bf\ycellD{349} & \bf\ycellD{350} & \bf\gcell{351} \\
{\boldmath $\gcd(j{+}1,N)$} & \gcell{1} & \bcell{3} & \gcell{1} & \gcell{1} & \bcell{15} & \gcell{1} & \gcell{1} & \bcell{3} & \gcell{1} & \ycellD{5} & \ycellD{3} & \gcell{1} \\
{\boldmath $S_{i,j}\bmod N$} & \gcell{0} & \bcell{0} & \gcell{0} & \gcell{0} & \bcell{0} & \gcell{0} & \gcell{0} & \bcell{0} & \gcell{0} & \ycellD{141} & \ycellD{235} & \gcell{0} \\

\bottomrule
\end{tabular}
}
\label{tab:fibo705}
\end{table}

\begin{table}[ht]
\caption{Example for $N = 7$. Representation of the different decompositions and the associated congruence properties ({\bf Green} if congruent to $0$ modulo $N$, {\bf Yellow} otherwise). Bold cells ({\bf Red}) correspond to indices such that $i+j = N-2$, appearing in the expression of $L_N-1$ (\ref{eq:PTTlucas}).}
\centering
\begin{minipage}[t][4.5cm]{0.3\textwidth}
\vspace{0pt}
\centering
\scriptsize
\setlength{\tabcolsep}{2pt}
\renewcommand{\arraystretch}{0.9}
\begin{tabular}{@{}r|rrrrrrr@{}}
$i$ \\
$0$ & 1 \\
$1$ & 1 & 1 \\
$2$ & 1 & 2 & 1 \\
$3$ & 1 & 3 & \rcelld{3} & 1 \\
$4$ & 1 & \rcelld{4} & 6 & 4 & 1 \\
$5$ & \rcelld{1} & 5 & 10 & 10 & 5 & 1 \\
$6$ & 1 & 6 & 15 & 20 & 15 & 6 & 1 \\
\hline
\multicolumn{1}{r}{$j$} & 0 & 1 & 2 & 3 & 4 & 5 & 6
\end{tabular}
\parbox[c][2.6em][c]{\linewidth}{\centering $A^{\hspace{0.02cm}7}_{i,j}=\binom{i}{j}$}
\end{minipage}
\hfill
\begin{minipage}[t][4.5cm]{0.3\textwidth}
\vspace{0pt}
\centering
\scriptsize
\setlength{\tabcolsep}{2pt}
\renewcommand{\arraystretch}{0.9}
\begin{tabular}{@{}r|rrrrrrr@{}}
$i$ \\
$0$ & 2 \\
$1$ & 7 & 2 \\
$2$ & 16 & 7 & 2 \\
$3$ & 21 & 13 & \rcellD{7} & 2 \\
$4$ & 16 & \rcellD{14} & 12 & 7 & 2 \\
$5$ & \rcellD{7} & 10 & 14 & 13 & 7 & 2 \\
$6$ & 2 & 7 & 16 & 21 & 16 & 7 & 2 \\
\hline
\multicolumn{1}{r}{$j$} & 0 & 1 & 2 & 3 & 4 & 5 & 6
\end{tabular}
\parbox[c][2.6em][c]{\linewidth}{\centering $S^{\hspace{0.02cm}7}_{i,j}=A^{\hspace{0.02cm}7}_{i,j}+B^{\hspace{0.02cm}7}_{i,j}$}
\end{minipage}
\hfill
\begin{minipage}[t][4.5cm]{0.3\textwidth}
\vspace{0pt}
\centering
\scriptsize
\setlength{\tabcolsep}{2pt}
\renewcommand{\arraystretch}{0.9}
\begin{tabular}{@{}r|rrrrrrr@{}}
$i$ \\
$0$ & 0 \\
$1$ & 7 & 0 \\
$2$ & -14 & 7 & 0 \\
$3$ & 21 & -7 & 7 & 0 \\
$4$ & -14 & 14 & 0 & 7 & 0 \\
$5$ & 7 & 0 & 14 & 7 & 7 & 0 \\
$6$ & 0 & 7 & 14 & 21 & 14 & 7 & 0 \\
\hline
\multicolumn{1}{r}{$j$} & 0 & 1 & 2 & 3 & 4 & 5 & 6
\end{tabular}
\parbox[c][2.6em][c]{\linewidth}{\centering $X^{\hspace{0.02cm}7}_{i,j}=A^{\hspace{0.02cm}7}_{i,j}{-}(-1)^{i{+}j}B^{\hspace{0.02cm}7}_{i,j}$}
\end{minipage}
\begin{minipage}[t][4.5cm]{0.3\textwidth}
\vspace{0pt}
\centering
\scriptsize
\setlength{\tabcolsep}{2pt}
\renewcommand{\arraystretch}{0.9}
\begin{tabular}{@{}r|rrrrrrr@{}}
$i$ \\
$0$ & 1 \\
$1$ & 6 & 1 \\
$2$ & 15 & 5 & 1 \\
$3$ & 20 & 10 & \rcelld{4} & 1 \\
$4$ & 15 & \rcelld{10} & 6 & 3 & 1 \\
$5$ & \rcelld{6} & 5 & 4 & 3 & 2 & 1 \\
$6$ & 1 & 1 & 1 & 1 & 1 & 1 & 1 \\
\hline
\multicolumn{1}{r}{$j$} & 0 & 1 & 2 & 3 & 4 & 5 & 6
\end{tabular}
\parbox[c][2.6em][c]{\linewidth}{\centering $B^{\hspace{0.02cm}7}_{i,j}=\binom{{\hspace{0.02cm}7}-1-j}{{\hspace{0.02cm}7}-1-i}$}
\end{minipage}
\hfill
\begin{minipage}[t][4.5cm]{0.3\textwidth}
\vspace{0pt}
\centering
\scriptsize
\setlength{\tabcolsep}{2pt}
\renewcommand{\arraystretch}{0.9}
\begin{tabular}{@{}r|rrrrrrr@{}}
$i$ \\
$0$ & 0 \\
$1$ & -5 & 0 \\
$2$ & -14 & -3 & 0 \\
$3$ & -19 & -7 & -1 & 0 \\
$4$ & -14 & -6 & 0 & 1 & 0 \\
$5$ & -5 & 0 & 6 & 7 & 3 & 0 \\
$6$ & 0 & 5 & 14 & 19 & 14 & 5 & 0 \\
\hline
\multicolumn{1}{r}{$j$} & 0 & 1 & 2 & 3 & 4 & 5 & 6
\end{tabular}
\parbox[c][2.6em][c]{\linewidth}{\centering $W^{\hspace{0.02cm}7}_{i,j}=A^{\hspace{0.02cm}7}_{i,j}-B^{\hspace{0.02cm}7}_{i,j}$}
\end{minipage}
\hfill
\begin{minipage}[t][4.5cm]{0.3\textwidth}
\vspace{0pt}
\centering
\scriptsize
\setlength{\tabcolsep}{2pt}
\renewcommand{\arraystretch}{0.9}
\begin{tabular}{@{}r|rrrrrrr@{}}
$i$ \\
$0$ & 2 \\
$1$ & -5 & 2 \\
$2$ & 16 & -3 & 2 \\
$3$ & -19 & 13 & -1 & 2 \\
$4$ & 16 & -6 & 12 & 1 & 2 \\
$5$ & -5 & 10 & 6 & 13 & 3 & 2 \\
$6$ & 2 & 5 & 16 & 19 & 16 & 5 & 2 \\
\hline
\multicolumn{1}{r}{$j$} & 0 & 1 & 2 & 3 & 4 & 5 & 6
\end{tabular}
\parbox[c][2.6em][c]{\linewidth}{\centering $Y^{\hspace{0.02cm}7}_{i,j}=A^{\hspace{0.02cm}7}_{i,j}{+}(-1)^{i{+}j}B^{\hspace{0.02cm}7}_{i,j}$}
\end{minipage}
\begin{minipage}[t][4.5cm]{0.3\textwidth}
\vspace{0pt}
\centering
\scriptsize
\setlength{\tabcolsep}{2pt}
\renewcommand{\arraystretch}{0.9}
\begin{tabular}{@{}r|rrrrrrr@{}}
$i$ \\
$0$ & \ycell{1} \\
$1$ & \ycell{1} & \ycell{1} \\
$2$ & \ycell{1} & \ycell{2} & \ycell{1} \\
$3$ & \ycell{1} & \ycell{3} & \rcelld{3} & \ycell{1} \\
$4$ & \ycell{1} & \rcelld{4} & \ycell{6} & \ycell{4} & \ycell{1} \\
$5$ & \rcelld{1} & \ycell{5} & \ycell{3} & \ycell{3} & \ycell{5} & \ycell{1} \\
$6$ & \ycell{1} & \ycell{6} & \ycell{1} & \ycell{6} & \ycell{1} & \ycell{6} & \ycell{1} \\
\hline
\multicolumn{1}{r}{$j$} & 0 & 1 & 2 & 3 & 4 & 5 & 6
\end{tabular}
\parbox[c][2.6em][c]{\linewidth}{\centering $A^{\hspace{0.02cm}7}_{i,j}\ \mathrm{mod}\ {\hspace{0.02cm}7}$}
\end{minipage}
\hfill
\begin{minipage}[t][4.5cm]{0.3\textwidth}
\vspace{0pt}
\centering
\scriptsize
\setlength{\tabcolsep}{2pt}
\renewcommand{\arraystretch}{0.9}
\begin{tabular}{@{}r|rrrrrrr@{}}
$i$ \\
$0$ & \ycell{2} \\
$1$ & \gcell{0} & \ycell{2} \\
$2$ & \ycell{2} & \gcell{0} & \ycell{2} \\
$3$ & \gcell{0} & \ycell{6} & \rcellD{0} & \ycell{2} \\
$4$ & \ycell{2} & \rcellD{0} & \ycell{5} & \gcell{0} & \ycell{2} \\
$5$ & \rcellD{0} & \ycell{3} & \gcell{0} & \ycell{6} & \gcell{0} & \ycell{2} \\
$6$ & \ycell{2} & \gcell{0} & \ycell{2} & \gcell{0} & \ycell{2} & \gcell{0} & \ycell{2} \\
\hline
\multicolumn{1}{r}{$j$} & 0 & 1 & 2 & 3 & 4 & 5 & 6
\end{tabular}
\parbox[c][2.6em][c]{\linewidth}{\centering $S^{\hspace{0.02cm}7}_{i,j}\ \mathrm{mod}\ {\hspace{0.02cm}7}$}
\end{minipage}
\hfill
\begin{minipage}[t][4.5cm]{0.3\textwidth}
\vspace{0pt}
\centering
\scriptsize
\setlength{\tabcolsep}{2pt}
\renewcommand{\arraystretch}{0.9}
\begin{tabular}{@{}r|rrrrrrr@{}}
$i$ \\
$0$ & \gcell{0} \\
$1$ & \gcell{0} & \gcell{0} \\
$2$ & \gcell{0} & \gcell{0} & \gcell{0} \\
$3$ & \gcell{0} & \gcell{0} & \gcell{0} & \gcell{0} \\
$4$ & \gcell{0} & \gcell{0} & \gcell{0} & \gcell{0} & \gcell{0} \\
$5$ & \gcell{0} & \gcell{0} & \gcell{0} & \gcell{0} & \gcell{0} & \gcell{0} \\
$6$ & \gcell{0} & \gcell{0} & \gcell{0} & \gcell{0} & \gcell{0} & \gcell{0} & \gcell{0} \\
\hline
\multicolumn{1}{r}{$j$} & 0 & 1 & 2 & 3 & 4 & 5 & 6
\end{tabular}
\parbox[c][2.6em][c]{\linewidth}{\centering $X^{\hspace{0.02cm}7}_{i,j}\ \mathrm{mod}\ {\hspace{0.02cm}7}$}
\end{minipage}
\begin{minipage}[t][4.5cm]{0.3\textwidth}
\vspace{0pt}
\centering
\scriptsize
\setlength{\tabcolsep}{2pt}
\renewcommand{\arraystretch}{0.9}
\begin{tabular}{@{}r|rrrrrrr@{}}
$i$ \\
$0$ & \ycell{1} \\
$1$ & \ycell{6} & \ycell{1} \\
$2$ & \ycell{1} & \ycell{5} & \ycell{1} \\
$3$ & \ycell{6} & \ycell{3} & \rcelld{4} & \ycell{1} \\
$4$ & \ycell{1} & \rcelld{3} & \ycell{6} & \ycell{3} & \ycell{1} \\
$5$ & \rcelld{6} & \ycell{5} & \ycell{4} & \ycell{3} & \ycell{2} & \ycell{1} \\
$6$ & \ycell{1} & \ycell{1} & \ycell{1} & \ycell{1} & \ycell{1} & \ycell{1} & \ycell{1} \\
\hline
\multicolumn{1}{r}{$j$} & 0 & 1 & 2 & 3 & 4 & 5 & 6
\end{tabular}
\parbox[c][2.6em][c]{\linewidth}{\centering $B^{\hspace{0.02cm}7}_{i,j}\ \mathrm{mod}\ {\hspace{0.02cm}7}$}
\end{minipage}
\hfill
\begin{minipage}[t][4.5cm]{0.3\textwidth}
\vspace{0pt}
\centering
\scriptsize
\setlength{\tabcolsep}{2pt}
\renewcommand{\arraystretch}{0.9}
\begin{tabular}{@{}r|rrrrrrr@{}}
$i$ \\
$0$ & \gcell{0} \\
$1$ & \ycell{2} & \gcell{0} \\
$2$ & \gcell{0} & \ycell{4} & \gcell{0} \\
$3$ & \ycell{2} & \gcell{0} & \ycell{6} & \gcell{0} \\
$4$ & \gcell{0} & \ycell{1} & \gcell{0} & \ycell{1} & \gcell{0} \\
$5$ & \ycell{2} & \gcell{0} & \ycell{6} & \gcell{0} & \ycell{3} & \gcell{0} \\
$6$ & \gcell{0} & \ycell{5} & \gcell{0} & \ycell{5} & \gcell{0} & \ycell{5} & \gcell{0} \\
\hline
\multicolumn{1}{r}{$j$} & 0 & 1 & 2 & 3 & 4 & 5 & 6
\end{tabular}
\parbox[c][2.6em][c]{\linewidth}{\centering $W^{\hspace{0.02cm}7}_{i,j}\ \mathrm{mod}\ {\hspace{0.02cm}7}$}
\end{minipage}
\hfill
\begin{minipage}[t][4.5cm]{0.3\textwidth}
\vspace{0pt}
\centering
\scriptsize
\setlength{\tabcolsep}{2pt}
\renewcommand{\arraystretch}{0.9}
\begin{tabular}{@{}r|rrrrrrr@{}}
$i$ \\
$0$ & \ycell{2} \\
$1$ & \ycell{2} & \ycell{2} \\
$2$ & \ycell{2} & \ycell{4} & \ycell{2} \\
$3$ & \ycell{2} & \ycell{6} & \ycell{6} & \ycell{2} \\
$4$ & \ycell{2} & \ycell{1} & \ycell{5} & \ycell{1} & \ycell{2} \\
$5$ & \ycell{2} & \ycell{3} & \ycell{6} & \ycell{6} & \ycell{3} & \ycell{2} \\
$6$ & \ycell{2} & \ycell{5} & \ycell{2} & \ycell{5} & \ycell{2} & \ycell{5} & \ycell{2} \\
\hline
\multicolumn{1}{r}{$j$} & 0 & 1 & 2 & 3 & 4 & 5 & 6
\end{tabular}
\parbox[c][2.6em][c]{\linewidth}{\centering $Y^{\hspace{0.02cm}7}_{i,j}\ \mathrm{mod}\ {\hspace{0.02cm}7}$}
\end{minipage}
\label{tab:swxy-grid-7}
\end{table} 

\begin{table}[ht]
\caption{Example for $N = 9 = 3^2$. Representation of the different decompositions and the associated congruence properties ({\bf Green} if congruent to $0$ modulo $N$, {\bf Yellow} otherwise). Bold cells ({\bf Red}) correspond to indices such that $i+j = N-2$, appearing in the expression of $L_N-1$ (\ref{eq:PTTlucas}).}
\centering
\begin{minipage}[t][4.5cm]{0.30\textwidth}
\vspace{0pt}
\centering
\scriptsize
\setlength{\tabcolsep}{2pt}
\renewcommand{\arraystretch}{0.9}
\begin{tabular}{@{}r|rrrrrrrrr@{}}
$i$ \\
$0$ & 1 \\
$1$ & 1 & 1 \\
$2$ & 1 & 2 & 1 \\
$3$ & 1 & 3 & 3 & 1 \\
$4$ & 1 & 4 & 6 & \rcelld{4} & 1 \\
$5$ & 1 & 5 & \rcelld{10} & 10 & 5 & 1 \\
$6$ & 1 & \rcelld{6} & 15 & 20 & 15 & 6 & 1 \\
$7$ & \rcelld{1} & 7 & 21 & 35 & 35 & 21 & 7 & 1 \\
$8$ & 1 & 8 & 28 & 56 & 70 & 56 & 28 & 8 & 1 \\
\hline
\multicolumn{1}{r}{$j$} & 0 & 1 & 2 & 3 & 4 & 5 & 6 & 7 & 8
\end{tabular}
\parbox[c][2.6em][c]{\linewidth}{\centering $A^{\hspace{0.02cm}9}_{i,j}=\binom{i}{j}$}
\end{minipage}
\hfill
\begin{minipage}[t][4.5cm]{0.30\textwidth}
\vspace{0pt}
\centering
\scriptsize
\setlength{\tabcolsep}{2pt}
\renewcommand{\arraystretch}{0.9}
\begin{tabular}{@{}r|rrrrrrrrr@{}}
$i$ \\
$0$ & 2 \\
$1$ & 9 & 2 \\
$2$ & 29 & 9 & 2 \\
$3$ & 57 & 24 & 9 & 2 \\
$4$ & 71 & 39 & 21 & \rcellD{9} & 2 \\
$5$ & 57 & 40 & \rcellD{30} & 20 & 9 & 2 \\
$6$ & 29 & \rcellD{27} & 30 & 30 & 21 & 9 & 2 \\
$7$ & \rcellD{9} & 14 & 27 & 40 & 39 & 24 & 9 & 2 \\
$8$ & 2 & 9 & 29 & 57 & 71 & 57 & 29 & 9 & 2 \\
\hline
\multicolumn{1}{r}{$j$} & 0 & 1 & 2 & 3 & 4 & 5 & 6 & 7 & 8
\end{tabular}
\parbox[c][2.6em][c]{\linewidth}{\centering $S^{\hspace{0.02cm}9}_{i,j}=A^{\hspace{0.02cm}9}_{i,j}+B^{\hspace{0.02cm}9}_{i,j}$}
\end{minipage}
\hfill
\begin{minipage}[t][4.5cm]{0.30\textwidth}
\vspace{0pt}
\centering
\scriptsize
\setlength{\tabcolsep}{2pt}
\renewcommand{\arraystretch}{0.9}
\begin{tabular}{@{}r|rrrrrrrrr@{}}
$i$ \\
$0$ & 0 \\
$1$ & 9 & 0 \\
$2$ & -27 & 9 & 0 \\
$3$ & 57 & -18 & 9 & 0 \\
$4$ & -69 & 39 & -9 & 9 & 0 \\
$5$ & 57 & -30 & 30 & 0 & 9 & 0 \\
$6$ & -27 & 27 & 0 & 30 & 9 & 9 & 0 \\
$7$ & 9 & 0 & 27 & 30 & 39 & 18 & 9 & 0 \\
$8$ & 0 & 9 & 27 & 57 & 69 & 57 & 27 & 9 & 0 \\
\hline
\multicolumn{1}{r}{$j$} & 0 & 1 & 2 & 3 & 4 & 5 & 6 & 7 & 8
\end{tabular}
\parbox[c][2.6em][c]{\linewidth}{\centering $X^{\hspace{0.02cm}9}_{i,j}=A^{\hspace{0.02cm}9}_{i,j}{-}(-1)^{i{+}j}B^{\hspace{0.02cm}9}_{i,j}$}
\end{minipage}
\begin{minipage}[t][4.5cm]{0.30\textwidth}
\vspace{0pt}
\centering
\scriptsize
\setlength{\tabcolsep}{2pt}
\renewcommand{\arraystretch}{0.9}
\begin{tabular}{@{}r|rrrrrrrrr@{}}
$i$ \\
$0$ & 1 \\
$1$ & 8 & 1 \\
$2$ & 28 & 7 & 1 \\
$3$ & 56 & 21 & 6 & 1 \\
$4$ & 70 & 35 & 15 & \rcelld{5} & 1 \\
$5$ & 56 & 35 & \rcelld{20} & 10 & 4 & 1 \\
$6$ & 28 & \rcelld{21} & 15 & 10 & 6 & 3 & 1 \\
$7$ & \rcelld{8} & 7 & 6 & 5 & 4 & 3 & 2 & 1 \\
$8$ & 1 & 1 & 1 & 1 & 1 & 1 & 1 & 1 & 1 \\
\hline
\multicolumn{1}{r}{$j$} & 0 & 1 & 2 & 3 & 4 & 5 & 6 & 7 & 8
\end{tabular}
\parbox[c][2.6em][c]{\linewidth}{\centering $B^{\hspace{0.02cm}9}_{i,j}=\binom{{\hspace{0.02cm}9}-1-j}{{\hspace{0.02cm}9}-1-i}$}
\end{minipage}
\hfill
\begin{minipage}[t][4.5cm]{0.30\textwidth}
\vspace{0pt}
\centering
\scriptsize
\setlength{\tabcolsep}{2pt}
\renewcommand{\arraystretch}{0.9}
\begin{tabular}{@{}r|rrrrrrrrr@{}}
$i$ \\
$0$ & 0 \\
$1$ & -7 & 0 \\
$2$ & -27 & -5 & 0 \\
$3$ & -55 & -18 & -3 & 0 \\
$4$ & -69 & -31 & -9 & -1 & 0 \\
$5$ & -55 & -30 & -10 & 0 & 1 & 0 \\
$6$ & -27 & -15 & 0 & 10 & 9 & 3 & 0 \\
$7$ & -7 & 0 & 15 & 30 & 31 & 18 & 5 & 0 \\
$8$ & 0 & 7 & 27 & 55 & 69 & 55 & 27 & 7 & 0 \\
\hline
\multicolumn{1}{r}{$j$} & 0 & 1 & 2 & 3 & 4 & 5 & 6 & 7 & 8
\end{tabular}
\parbox[c][2.6em][c]{\linewidth}{\centering $W^{\hspace{0.02cm}9}_{i,j}=A^{\hspace{0.02cm}9}_{i,j}-B^{\hspace{0.02cm}9}_{i,j}$}
\end{minipage}
\hfill
\begin{minipage}[t][4.5cm]{0.30\textwidth}
\vspace{0pt}
\centering
\scriptsize
\setlength{\tabcolsep}{2pt}
\renewcommand{\arraystretch}{0.9}
\begin{tabular}{@{}r|rrrrrrrrr@{}}
$i$ \\
$0$ & 2 \\
$1$ & -7 & 2 \\
$2$ & 29 & -5 & 2 \\
$3$ & -55 & 24 & -3 & 2 \\
$4$ & 71 & -31 & 21 & -1 & 2 \\
$5$ & -55 & 40 & -10 & 20 & 1 & 2 \\
$6$ & 29 & -15 & 30 & 10 & 21 & 3 & 2 \\
$7$ & -7 & 14 & 15 & 40 & 31 & 24 & 5 & 2 \\
$8$ & 2 & 7 & 29 & 55 & 71 & 55 & 29 & 7 & 2 \\
\hline
\multicolumn{1}{r}{$j$} & 0 & 1 & 2 & 3 & 4 & 5 & 6 & 7 & 8
\end{tabular}
\parbox[c][2.6em][c]{\linewidth}{\centering $Y^{\hspace{0.02cm}9}_{i,j}=A^{\hspace{0.02cm}9}_{i,j}{+}(-1)^{i{+}j}B^{\hspace{0.02cm}9}_{i,j}$}
\end{minipage}
\begin{minipage}[t][4.5cm]{0.30\textwidth}
\vspace{0pt}
\centering
\scriptsize
\setlength{\tabcolsep}{2pt}
\renewcommand{\arraystretch}{0.9}
\begin{tabular}{@{}r|rrrrrrrrr@{}}
$i$ \\
$0$ & \ycell{1} \\
$1$ & \ycell{1} & \ycell{1} \\
$2$ & \ycell{1} & \ycell{2} & \ycell{1} \\
$3$ & \ycell{1} & \ycell{3} & \ycell{3} & \ycell{1} \\
$4$ & \ycell{1} & \ycell{4} & \ycell{6} & \rcelld{4} & \ycell{1} \\
$5$ & \ycell{1} & \ycell{5} & \rcelld{1} & \ycell{1} & \ycell{5} & \ycell{1} \\
$6$ & \ycell{1} & \rcelld{6} & \ycell{6} & \ycell{2} & \ycell{6} & \ycell{6} & \ycell{1} \\
$7$ & \rcelld{1} & \ycell{7} & \ycell{3} & \ycell{8} & \ycell{8} & \ycell{3} & \ycell{7} & \ycell{1} \\
$8$ & \ycell{1} & \ycell{8} & \ycell{1} & \ycell{2} & \ycell{7} & \ycell{2} & \ycell{1} & \ycell{8} & \ycell{1} \\
\hline
\multicolumn{1}{r}{$j$} & 0 & 1 & 2 & 3 & 4 & 5 & 6 & 7 & 8
\end{tabular}
\parbox[c][2.6em][c]{\linewidth}{\centering $A^{\hspace{0.02cm}9}_{i,j}\ \mathrm{mod}\ {\hspace{0.02cm}9}$}
\end{minipage}
\hfill
\begin{minipage}[t][4.5cm]{0.30\textwidth}
\vspace{0pt}
\centering
\scriptsize
\setlength{\tabcolsep}{2pt}
\renewcommand{\arraystretch}{0.9}
\begin{tabular}{@{}r|rrrrrrrrr@{}}
$i$ \\
$0$ & \ycell{2} \\
$1$ & \gcell{0} & \ycell{2} \\
$2$ & \ycell{2} & \gcell{0} & \ycell{2} \\
$3$ & \ycell{3} & \ycell{6} & \gcell{0} & \ycell{2} \\
$4$ & \ycell{8} & \ycell{3} & \ycell{3} & \rcellD{0} & \ycell{2} \\
$5$ & \ycell{3} & \ycell{4} & \rcellD{3} & \ycell{2} & \gcell{0} & \ycell{2} \\
$6$ & \ycell{2} & \rcellD{0} & \ycell{3} & \ycell{3} & \ycell{3} & \gcell{0} & \ycell{2} \\
$7$ & \rcellD{0} & \ycell{5} & \gcell{0} & \ycell{4} & \ycell{3} & \ycell{6} & \gcell{0} & \ycell{2} \\
$8$ & \ycell{2} & \gcell{0} & \ycell{2} & \ycell{3} & \ycell{8} & \ycell{3} & \ycell{2} & \gcell{0} & \ycell{2} \\
\hline
\multicolumn{1}{r}{$j$} & 0 & 1 & 2 & 3 & 4 & 5 & 6 & 7 & 8
\end{tabular}
\parbox[c][2.6em][c]{\linewidth}{\centering $S^{\hspace{0.02cm}9}_{i,j}\ \mathrm{mod}\ {\hspace{0.02cm}9}$}
\end{minipage}
\hfill
\begin{minipage}[t][4.5cm]{0.30\textwidth}
\vspace{0pt}
\centering
\scriptsize
\setlength{\tabcolsep}{2pt}
\renewcommand{\arraystretch}{0.9}
\begin{tabular}{@{}r|rrrrrrrrr@{}}
$i$ \\
$0$ & \gcell{0} \\
$1$ & \gcell{0} & \gcell{0} \\
$2$ & \gcell{0} & \gcell{0} & \gcell{0} \\
$3$ & \ycell{3} & \gcell{0} & \gcell{0} & \gcell{0} \\
$4$ & \ycell{3} & \ycell{3} & \gcell{0} & \gcell{0} & \gcell{0} \\
$5$ & \ycell{3} & \ycell{6} & \ycell{3} & \gcell{0} & \gcell{0} & \gcell{0} \\
$6$ & \gcell{0} & \gcell{0} & \gcell{0} & \ycell{3} & \gcell{0} & \gcell{0} & \gcell{0} \\
$7$ & \gcell{0} & \gcell{0} & \gcell{0} & \ycell{3} & \ycell{3} & \gcell{0} & \gcell{0} & \gcell{0} \\
$8$ & \gcell{0} & \gcell{0} & \gcell{0} & \ycell{3} & \ycell{6} & \ycell{3} & \gcell{0} & \gcell{0} & \gcell{0} \\
\hline
\multicolumn{1}{r}{$j$} & 0 & 1 & 2 & 3 & 4 & 5 & 6 & 7 & 8
\end{tabular}
\parbox[c][2.6em][c]{\linewidth}{\centering $X^{\hspace{0.02cm}9}_{i,j}\ \mathrm{mod}\ {\hspace{0.02cm}9}$}
\end{minipage}
\begin{minipage}[t][4.5cm]{0.30\textwidth}
\vspace{0pt}
\centering
\scriptsize
\setlength{\tabcolsep}{2pt}
\renewcommand{\arraystretch}{0.9}
\begin{tabular}{@{}r|rrrrrrrrr@{}}
$i$ \\
$0$ & \ycell{1} \\
$1$ & \ycell{8} & \ycell{1} \\
$2$ & \ycell{1} & \ycell{7} & \ycell{1} \\
$3$ & \ycell{2} & \ycell{3} & \ycell{6} & \ycell{1} \\
$4$ & \ycell{7} & \ycell{8} & \ycell{6} & \rcelld{5} & \ycell{1} \\
$5$ & \ycell{2} & \ycell{8} & \rcelld{2} & \ycell{1} & \ycell{4} & \ycell{1} \\
$6$ & \ycell{1} & \rcelld{3} & \ycell{6} & \ycell{1} & \ycell{6} & \ycell{3} & \ycell{1} \\
$7$ & \rcelld{8} & \ycell{7} & \ycell{6} & \ycell{5} & \ycell{4} & \ycell{3} & \ycell{2} & \ycell{1} \\
$8$ & \ycell{1} & \ycell{1} & \ycell{1} & \ycell{1} & \ycell{1} & \ycell{1} & \ycell{1} & \ycell{1} & \ycell{1} \\
\hline
\multicolumn{1}{r}{$j$} & 0 & 1 & 2 & 3 & 4 & 5 & 6 & 7 & 8
\end{tabular}
\parbox[c][2.6em][c]{\linewidth}{\centering $B^{\hspace{0.02cm}9}_{i,j}\ \mathrm{mod}\ {\hspace{0.02cm}9}$}
\end{minipage}
\hfill
\begin{minipage}[t][4.5cm]{0.30\textwidth}
\vspace{0pt}
\centering
\scriptsize
\setlength{\tabcolsep}{2pt}
\renewcommand{\arraystretch}{0.9}
\begin{tabular}{@{}r|rrrrrrrrr@{}}
$i$ \\
$0$ & \gcell{0} \\
$1$ & \ycell{2} & \gcell{0} \\
$2$ & \gcell{0} & \ycell{4} & \gcell{0} \\
$3$ & \ycell{8} & \gcell{0} & \ycell{6} & \gcell{0} \\
$4$ & \ycell{3} & \ycell{5} & \gcell{0} & \ycell{8} & \gcell{0} \\
$5$ & \ycell{8} & \ycell{6} & \ycell{8} & \gcell{0} & \ycell{1} & \gcell{0} \\
$6$ & \gcell{0} & \ycell{3} & \gcell{0} & \ycell{1} & \gcell{0} & \ycell{3} & \gcell{0} \\
$7$ & \ycell{2} & \gcell{0} & \ycell{6} & \ycell{3} & \ycell{4} & \gcell{0} & \ycell{5} & \gcell{0} \\
$8$ & \gcell{0} & \ycell{7} & \gcell{0} & \ycell{1} & \ycell{6} & \ycell{1} & \gcell{0} & \ycell{7} & \gcell{0} \\
\hline
\multicolumn{1}{r}{$j$} & 0 & 1 & 2 & 3 & 4 & 5 & 6 & 7 & 8
\end{tabular}
\parbox[c][2.6em][c]{\linewidth}{\centering $W^{\hspace{0.02cm}9}_{i,j}\ \mathrm{mod}\ {\hspace{0.02cm}9}$}
\end{minipage}
\hfill
\begin{minipage}[t][4.5cm]{0.30\textwidth}
\vspace{0pt}
\centering
\scriptsize
\setlength{\tabcolsep}{2pt}
\renewcommand{\arraystretch}{0.9}
\begin{tabular}{@{}r|rrrrrrrrr@{}}
$i$ \\
$0$ & \ycell{2} \\
$1$ & \ycell{2} & \ycell{2} \\
$2$ & \ycell{2} & \ycell{4} & \ycell{2} \\
$3$ & \ycell{8} & \ycell{6} & \ycell{6} & \ycell{2} \\
$4$ & \ycell{8} & \ycell{5} & \ycell{3} & \ycell{8} & \ycell{2} \\
$5$ & \ycell{8} & \ycell{4} & \ycell{8} & \ycell{2} & \ycell{1} & \ycell{2} \\
$6$ & \ycell{2} & \ycell{3} & \ycell{3} & \ycell{1} & \ycell{3} & \ycell{3} & \ycell{2} \\
$7$ & \ycell{2} & \ycell{5} & \ycell{6} & \ycell{4} & \ycell{4} & \ycell{6} & \ycell{5} & \ycell{2} \\
$8$ & \ycell{2} & \ycell{7} & \ycell{2} & \ycell{1} & \ycell{8} & \ycell{1} & \ycell{2} & \ycell{7} & \ycell{2} \\
\hline
\multicolumn{1}{r}{$j$} & 0 & 1 & 2 & 3 & 4 & 5 & 6 & 7 & 8
\end{tabular}
\parbox[c][2.6em][c]{\linewidth}{\centering $Y^{\hspace{0.02cm}9}_{i,j}\ \mathrm{mod}\ {\hspace{0.02cm}9}$}
\end{minipage}
\label{tab:swxy-grid-9}
\end{table}

\begin{table}[ht]
\caption{Example for $N= 10 = 2 \times 5$. Representation of the different decompositions and the associated congruence properties ({\bf Green} if congruent to $0$ modulo $N$, {\bf Yellow} otherwise). Bold cells ({\bf Red}) correspond to indices such that $i+j = N-2$, appearing in the expression of $L_N-1$ (\ref{eq:PTTlucas}).}
\centering
\begin{minipage}[t][4.5cm]{0.27\textwidth}
\vspace{0pt}
\centering
\scriptsize
\setlength{\tabcolsep}{2pt}
\renewcommand{\arraystretch}{0.9}
\begin{tabular}{@{}r|rrrrrrrrrr@{}}
$i$ \\
$0$ & 1 \\
$1$ & 1 & 1 \\
$2$ & 1 & 2 & 1 \\
$3$ & 1 & 3 & 3 & 1 \\
$4$ & 1 & 4 & 6 & 4 & \rcelld{1} \\
$5$ & 1 & 5 & 10 & \rcelld{10} & 5 & 1 \\
$6$ & 1 & 6 & \rcelld{15} & 20 & 15 & 6 & 1 \\
$7$ & 1 & \rcelld{7} & 21 & 35 & 35 & 21 & 7 & 1 \\
$8$ & \rcelld{1} & 8 & 28 & 56 & 70 & 56 & 28 & 8 & 1 \\
$9$ & 1 & 9 & 36 & 84 & 126 & 126 & 84 & 36 & 9 & 1 \\
\hline
\multicolumn{1}{r}{$j$} & 0 & 1 & 2 & 3 & 4 & 5 & 6 & 7 & 8 & 9
\end{tabular}
\parbox[c][2.6em][c]{\linewidth}{\centering $A^{\hspace{0.02cm}10}_{i,j}=\binom{i}{j}$}
\end{minipage}
\hfill
\begin{minipage}[t][4.5cm]{0.315\textwidth}
\vspace{0pt}
\centering
\scriptsize
\setlength{\tabcolsep}{2pt}
\renewcommand{\arraystretch}{0.9}
\begin{tabular}{@{}r|rrrrrrrrrr@{}}
$i$ \\
$0$ & 2 \\
$1$ & 10 & 2 \\
$2$ & 37 & 10 & 2 \\
$3$ & 85 & 31 & 10 & 2 \\
$4$ & 127 & 60 & 27 & 10 & \rcellD{2} \\
$5$ & 127 & 75 & 45 & \rcellD{25} & 10 & 2 \\
$6$ & 85 & 62 & \rcellD{50} & 40 & 25 & 10 & 2 \\
$7$ & 37 & \rcellD{35} & 42 & 50 & 45 & 27 & 10 & 2 \\
$8$ & \rcellD{10} & 16 & 35 & 62 & 75 & 60 & 31 & 10 & 2 \\
$9$ & 2 & 10 & 37 & 85 & 127 & 127 & 85 & 37 & 10 & 2 \\
\hline
\multicolumn{1}{r}{$j$} & 0 & 1 & 2 & 3 & 4 & 5 & 6 & 7 & 8 & 9
\end{tabular}
\parbox[c][2.6em][c]{\linewidth}{\centering $S^{\hspace{0.02cm}10}_{i,j}=A^{\hspace{0.02cm}10}_{i,j}+B^{\hspace{0.02cm}10}_{i,j}$}
\end{minipage}
\hfill
\begin{minipage}[t][4.5cm]{0.315\textwidth}
\vspace{0pt}
\centering
\scriptsize
\setlength{\tabcolsep}{2pt}
\renewcommand{\arraystretch}{0.9}
\begin{tabular}{@{}r|rrrrrrrrrr@{}}
$i$ \\
$0$ & 0 \\
$1$ & 10 & 0 \\
$2$ & -35 & 10 & 0 \\
$3$ & 85 & -25 & 10 & 0 \\
$4$ & -125 & 60 & -15 & 10 & 0 \\
$5$ & 127 & -65 & 45 & -5 & 10 & 0 \\
$6$ & -83 & 62 & -20 & 40 & 5 & 10 & 0 \\
$7$ & 37 & -21 & 42 & 20 & 45 & 15 & 10 & 0 \\
$8$ & -8 & 16 & 21 & 62 & 65 & 60 & 25 & 10 & 0 \\
$9$ & 2 & 8 & 37 & 83 & 127 & 125 & 85 & 35 & 10 & 0 \\
\hline
\multicolumn{1}{r}{$j$} & 0 & 1 & 2 & 3 & 4 & 5 & 6 & 7 & 8 & 9
\end{tabular}
\parbox[c][2.6em][c]{\linewidth}{\centering $X^{\hspace{0.02cm}10}_{i,j}=A^{\hspace{0.02cm}10}_{i,j}{-}(-1)^{i{+}j}B^{\hspace{0.02cm}10}_{i,j}$}
\end{minipage}
\begin{minipage}[t][4.5cm]{0.27\textwidth}
\vspace{0pt}
\centering
\scriptsize
\setlength{\tabcolsep}{2pt}
\renewcommand{\arraystretch}{0.9}
\begin{tabular}{@{}r|rrrrrrrrrr@{}}
$i$ \\
$0$ & 1 \\
$1$ & 9 & 1 \\
$2$ & 36 & 8 & 1 \\
$3$ & 84 & 28 & 7 & 1 \\
$4$ & 126 & 56 & 21 & 6 & \rcelld{1} \\
$5$ & 126 & 70 & 35 & \rcelld{15} & 5 & 1 \\
$6$ & 84 & 56 & \rcelld{35} & 20 & 10 & 4 & 1 \\
$7$ & 36 & \rcelld{28} & 21 & 15 & 10 & 6 & 3 & 1 \\
$8$ & \rcelld{9} & 8 & 7 & 6 & 5 & 4 & 3 & 2 & 1 \\
$9$ & 1 & 1 & 1 & 1 & 1 & 1 & 1 & 1 & 1 & 1 \\
\hline
\multicolumn{1}{r}{$j$} & 0 & 1 & 2 & 3 & 4 & 5 & 6 & 7 & 8 & 9
\end{tabular}
\parbox[c][2.6em][c]{\linewidth}{\centering $B^{\hspace{0.02cm}10}_{i,j}=\binom{{\hspace{0.02cm}10}-1-j}{{\hspace{0.02cm}10}-1-i}$}
\end{minipage}
\hfill
\begin{minipage}[t][4.5cm]{0.351\textwidth}
\vspace{0pt}
\centering
\scriptsize
\setlength{\tabcolsep}{2pt}
\renewcommand{\arraystretch}{0.9}
\begin{tabular}{@{}r|rrrrrrrrrr@{}}
$i$ \\
$0$ & 0 \\
$1$ & -8 & 0 \\
$2$ & -35 & -6 & 0 \\
$3$ & -83 & -25 & -4 & 0 \\
$4$ & -125 & -52 & -15 & -2 & 0 \\
$5$ & -125 & -65 & -25 & -5 & 0 & 0 \\
$6$ & -83 & -50 & -20 & 0 & 5 & 2 & 0 \\
$7$ & -35 & -21 & 0 & 20 & 25 & 15 & 4 & 0 \\
$8$ & -8 & 0 & 21 & 50 & 65 & 52 & 25 & 6 & 0 \\
$9$ & 0 & 8 & 35 & 83 & 125 & 125 & 83 & 35 & 8 & 0 \\
\hline
\multicolumn{1}{r}{$j$} & 0 & 1 & 2 & 3 & 4 & 5 & 6 & 7 & 8 & 9
\end{tabular}
\parbox[c][2.6em][c]{\linewidth}{\centering $W^{\hspace{0.02cm}10}_{i,j}=A^{\hspace{0.02cm}10}_{i,j}-B^{\hspace{0.02cm}10}_{i,j}$}
\end{minipage}
\hfill
\begin{minipage}[t][4.5cm]{0.315\textwidth}
\vspace{0pt}
\centering
\scriptsize
\setlength{\tabcolsep}{2pt}
\renewcommand{\arraystretch}{0.9}
\begin{tabular}{@{}r|rrrrrrrrrr@{}}
$i$ \\
$0$ & 2 \\
$1$ & -8 & 2 \\
$2$ & 37 & -6 & 2 \\
$3$ & -83 & 31 & -4 & 2 \\
$4$ & 127 & -52 & 27 & -2 & 2 \\
$5$ & -125 & 75 & -25 & 25 & 0 & 2 \\
$6$ & 85 & -50 & 50 & 0 & 25 & 2 & 2 \\
$7$ & -35 & 35 & 0 & 50 & 25 & 27 & 4 & 2 \\
$8$ & 10 & 0 & 35 & 50 & 75 & 52 & 31 & 6 & 2 \\
$9$ & 0 & 10 & 35 & 85 & 125 & 127 & 83 & 37 & 8 & 2 \\
\hline
\multicolumn{1}{r}{$j$} & 0 & 1 & 2 & 3 & 4 & 5 & 6 & 7 & 8 & 9
\end{tabular}
\parbox[c][2.6em][c]{\linewidth}{\centering $Y^{\hspace{0.02cm}10}_{i,j}=A^{\hspace{0.02cm}10}_{i,j}{+}(-1)^{i{+}j}B^{\hspace{0.02cm}10}_{i,j}$}
\end{minipage}
\begin{minipage}[t][4.5cm]{0.27\textwidth}
\vspace{0pt}
\centering
\scriptsize
\setlength{\tabcolsep}{2pt}
\renewcommand{\arraystretch}{0.9}
\begin{tabular}{@{}r|rrrrrrrrrr@{}}
$i$ \\
$0$ & \ycell{1} \\
$1$ & \ycell{1} & \ycell{1} \\
$2$ & \ycell{1} & \ycell{2} & \ycell{1} \\
$3$ & \ycell{1} & \ycell{3} & \ycell{3} & \ycell{1} \\
$4$ & \ycell{1} & \ycell{4} & \ycell{6} & \ycell{4} & \rcelld{1} \\
$5$ & \ycell{1} & \ycell{5} & \gcell{0} & \rcelld{0} & \ycell{5} & \ycell{1} \\
$6$ & \ycell{1} & \ycell{6} & \rcelld{5} & \gcell{0} & \ycell{5} & \ycell{6} & \ycell{1} \\
$7$ & \ycell{1} & \rcelld{7} & \ycell{1} & \ycell{5} & \ycell{5} & \ycell{1} & \ycell{7} & \ycell{1} \\
$8$ & \rcelld{1} & \ycell{8} & \ycell{8} & \ycell{6} & \gcell{0} & \ycell{6} & \ycell{8} & \ycell{8} & \ycell{1} \\
$9$ & \ycell{1} & \ycell{9} & \ycell{6} & \ycell{4} & \ycell{6} & \ycell{6} & \ycell{4} & \ycell{6} & \ycell{9} & \ycell{1} \\
\hline
\multicolumn{1}{r}{$j$} & 0 & 1 & 2 & 3 & 4 & 5 & 6 & 7 & 8 & 9
\end{tabular}
\parbox[c][2.6em][c]{\linewidth}{\centering $A^{\hspace{0.02cm}10}_{i,j}\ \mathrm{mod}\ {\hspace{0.02cm}10}$}
\end{minipage}
\hfill
\begin{minipage}[t][4.5cm]{0.315\textwidth}
\vspace{0pt}
\centering
\scriptsize
\setlength{\tabcolsep}{2pt}
\renewcommand{\arraystretch}{0.9}
\begin{tabular}{@{}r|rrrrrrrrrr@{}}
$i$ \\
$0$ & \ycell{2} \\
$1$ & \gcell{0} & \ycell{2} \\
$2$ & \ycell{7} & \gcell{0} & \ycell{2} \\
$3$ & \ycell{5} & \ycell{1} & \gcell{0} & \ycell{2} \\
$4$ & \ycell{7} & \gcell{0} & \ycell{7} & \gcell{0} & \rcellD{2} \\
$5$ & \ycell{7} & \ycell{5} & \ycell{5} & \rcellD{5} & \gcell{0} & \ycell{2} \\
$6$ & \ycell{5} & \ycell{2} & \rcellD{0} & \gcell{0} & \ycell{5} & \gcell{0} & \ycell{2} \\
$7$ & \ycell{7} & \rcellD{5} & \ycell{2} & \gcell{0} & \ycell{5} & \ycell{7} & \gcell{0} & \ycell{2} \\
$8$ & \rcellD{0} & \ycell{6} & \ycell{5} & \ycell{2} & \ycell{5} & \gcell{0} & \ycell{1} & \gcell{0} & \ycell{2} \\
$9$ & \ycell{2} & \gcell{0} & \ycell{7} & \ycell{5} & \ycell{7} & \ycell{7} & \ycell{5} & \ycell{7} & \gcell{0} & \ycell{2} \\
\hline
\multicolumn{1}{r}{$j$} & 0 & 1 & 2 & 3 & 4 & 5 & 6 & 7 & 8 & 9
\end{tabular}
\parbox[c][2.6em][c]{\linewidth}{\centering $S^{\hspace{0.02cm}10}_{i,j}\ \mathrm{mod}\ {\hspace{0.02cm}10}$}
\end{minipage}
\hfill
\begin{minipage}[t][4.5cm]{0.315\textwidth}
\vspace{0pt}
\centering
\scriptsize
\setlength{\tabcolsep}{2pt}
\renewcommand{\arraystretch}{0.9}
\begin{tabular}{@{}r|rrrrrrrrrr@{}}
$i$ \\
$0$ & \gcell{0} \\
$1$ & \gcell{0} & \gcell{0} \\
$2$ & \ycell{5} & \gcell{0} & \gcell{0} \\
$3$ & \ycell{5} & \ycell{5} & \gcell{0} & \gcell{0} \\
$4$ & \ycell{5} & \gcell{0} & \ycell{5} & \gcell{0} & \gcell{0} \\
$5$ & \ycell{7} & \ycell{5} & \ycell{5} & \ycell{5} & \gcell{0} & \gcell{0} \\
$6$ & \ycell{7} & \ycell{2} & \gcell{0} & \gcell{0} & \ycell{5} & \gcell{0} & \gcell{0} \\
$7$ & \ycell{7} & \ycell{9} & \ycell{2} & \gcell{0} & \ycell{5} & \ycell{5} & \gcell{0} & \gcell{0} \\
$8$ & \ycell{2} & \ycell{6} & \ycell{1} & \ycell{2} & \ycell{5} & \gcell{0} & \ycell{5} & \gcell{0} & \gcell{0} \\
$9$ & \ycell{2} & \ycell{8} & \ycell{7} & \ycell{3} & \ycell{7} & \ycell{5} & \ycell{5} & \ycell{5} & \gcell{0} & \gcell{0} \\
\hline
\multicolumn{1}{r}{$j$} & 0 & 1 & 2 & 3 & 4 & 5 & 6 & 7 & 8 & 9
\end{tabular}
\parbox[c][2.6em][c]{\linewidth}{\centering $X^{\hspace{0.02cm}10}_{i,j}\ \mathrm{mod}\ {\hspace{0.02cm}10}$}
\end{minipage}
\begin{minipage}[t][4.5cm]{0.27\textwidth}
\vspace{0pt}
\centering
\scriptsize
\setlength{\tabcolsep}{2pt}
\renewcommand{\arraystretch}{0.9}
\begin{tabular}{@{}r|rrrrrrrrrr@{}}
$i$ \\
$0$ & \ycell{1} \\
$1$ & \ycell{9} & \ycell{1} \\
$2$ & \ycell{6} & \ycell{8} & \ycell{1} \\
$3$ & \ycell{4} & \ycell{8} & \ycell{7} & \ycell{1} \\
$4$ & \ycell{6} & \ycell{6} & \ycell{1} & \ycell{6} & \rcelld{1} \\
$5$ & \ycell{6} & \gcell{0} & \ycell{5} & \rcelld{5} & \ycell{5} & \ycell{1} \\
$6$ & \ycell{4} & \ycell{6} & \rcelld{5} & \gcell{0} & \gcell{0} & \ycell{4} & \ycell{1} \\
$7$ & \ycell{6} & \rcelld{8} & \ycell{1} & \ycell{5} & \gcell{0} & \ycell{6} & \ycell{3} & \ycell{1} \\
$8$ & \rcelld{9} & \ycell{8} & \ycell{7} & \ycell{6} & \ycell{5} & \ycell{4} & \ycell{3} & \ycell{2} & \ycell{1} \\
$9$ & \ycell{1} & \ycell{1} & \ycell{1} & \ycell{1} & \ycell{1} & \ycell{1} & \ycell{1} & \ycell{1} & \ycell{1} & \ycell{1} \\
\hline
\multicolumn{1}{r}{$j$} & 0 & 1 & 2 & 3 & 4 & 5 & 6 & 7 & 8 & 9
\end{tabular}
\parbox[c][2.6em][c]{\linewidth}{\centering $B^{\hspace{0.02cm}10}_{i,j}\ \mathrm{mod}\ {\hspace{0.02cm}10}$}
\end{minipage}
\hfill
\begin{minipage}[t][4.5cm]{0.315\textwidth}
\vspace{0pt}
\centering
\scriptsize
\setlength{\tabcolsep}{2pt}
\renewcommand{\arraystretch}{0.9}
\begin{tabular}{@{}r|rrrrrrrrrr@{}}
$i$ \\
$0$ & \gcell{0} \\
$1$ & \ycell{2} & \gcell{0} \\
$2$ & \ycell{5} & \ycell{4} & \gcell{0} \\
$3$ & \ycell{7} & \ycell{5} & \ycell{6} & \gcell{0} \\
$4$ & \ycell{5} & \ycell{8} & \ycell{5} & \ycell{8} & \gcell{0} \\
$5$ & \ycell{5} & \ycell{5} & \ycell{5} & \ycell{5} & \gcell{0} & \gcell{0} \\
$6$ & \ycell{7} & \gcell{0} & \gcell{0} & \gcell{0} & \ycell{5} & \ycell{2} & \gcell{0} \\
$7$ & \ycell{5} & \ycell{9} & \gcell{0} & \gcell{0} & \ycell{5} & \ycell{5} & \ycell{4} & \gcell{0} \\
$8$ & \ycell{2} & \gcell{0} & \ycell{1} & \gcell{0} & \ycell{5} & \ycell{2} & \ycell{5} & \ycell{6} & \gcell{0} \\
$9$ & \gcell{0} & \ycell{8} & \ycell{5} & \ycell{3} & \ycell{5} & \ycell{5} & \ycell{3} & \ycell{5} & \ycell{8} & \gcell{0} \\
\hline
\multicolumn{1}{r}{$j$} & 0 & 1 & 2 & 3 & 4 & 5 & 6 & 7 & 8 & 9
\end{tabular}
\parbox[c][2.6em][c]{\linewidth}{\centering $W^{\hspace{0.02cm}10}_{i,j}\ \mathrm{mod}\ {\hspace{0.02cm}10}$}
\end{minipage}
\hfill
\begin{minipage}[t][4.5cm]{0.315\textwidth}
\vspace{0pt}
\centering
\scriptsize
\setlength{\tabcolsep}{2pt}
\renewcommand{\arraystretch}{0.9}
\begin{tabular}{@{}r|rrrrrrrrrr@{}}
$i$ \\
$0$ & \ycell{2} \\
$1$ & \ycell{2} & \ycell{2} \\
$2$ & \ycell{7} & \ycell{4} & \ycell{2} \\
$3$ & \ycell{7} & \ycell{1} & \ycell{6} & \ycell{2} \\
$4$ & \ycell{7} & \ycell{8} & \ycell{7} & \ycell{8} & \ycell{2} \\
$5$ & \ycell{5} & \ycell{5} & \ycell{5} & \ycell{5} & \gcell{0} & \ycell{2} \\
$6$ & \ycell{5} & \gcell{0} & \gcell{0} & \gcell{0} & \ycell{5} & \ycell{2} & \ycell{2} \\
$7$ & \ycell{5} & \ycell{5} & \gcell{0} & \gcell{0} & \ycell{5} & \ycell{7} & \ycell{4} & \ycell{2} \\
$8$ & \gcell{0} & \gcell{0} & \ycell{5} & \gcell{0} & \ycell{5} & \ycell{2} & \ycell{1} & \ycell{6} & \ycell{2} \\
$9$ & \gcell{0} & \gcell{0} & \ycell{5} & \ycell{5} & \ycell{5} & \ycell{7} & \ycell{3} & \ycell{7} & \ycell{8} & \ycell{2} \\
\hline
\multicolumn{1}{r}{$j$} & 0 & 1 & 2 & 3 & 4 & 5 & 6 & 7 & 8 & 9
\end{tabular}
\parbox[c][2.6em][c]{\linewidth}{\centering $Y^{\hspace{0.02cm}10}_{i,j}\ \mathrm{mod}\ {\hspace{0.02cm}10}$}
\end{minipage}
\label{tab:swxy-grid-10}
\end{table}

\begin{table}[ht]
\caption{Example for $N = 35 = 7\times 5$, illustrating some of the properties proved for composite odd integer $N$. {\bf Blue}: $S^{mq}_{i,j}\equiv m\mymod{mq}$. {\bf Red}: $S^{mq}_{i,j}$ with $i+j = mq-2$ (terms appearing in (\ref{eq:PTTlucas})). {\bf Yellow} : $(i_q,j_q)=(N-1-q,q-1)$ and $(i'_q,j'_q)=((N+q)/2-1,(N-q)/2-1)$.  The terms $(i_q,j_q)$ and $(i'_q,j'_q)$ are those used in the proof of Theorem~\ref{theo:pseudoprime} when $N$ is a Fibonacci pseudoprime. In the general composite case, these terms delimit the corner regions of $S^N$ that still exhibit a Pascal tiling. Corners are $(i,j)\in\mathcal{D}_{j\leq i}$ such that $i<q$, $i>i_q$ and $j<j_q$, or $j>N-1-q$ where $q$ is the smallest prime factor of $N$. Terms in the corner, not yet colored, are colored according to their divisibility by $N$ ({\bf Green}: $S^{mq}_{i,j}\equiv 0\mymod{mq}$, and {\bf Yellow} otherwise).\\}
\centering
\makebox[\textwidth][c]{\vspace{0pt}
\centering
\scriptsize
\setlength{\tabcolsep}{2pt}
\renewcommand{\arraystretch}{0.9}
\begin{tabular}{r|rrrrrrrrrrrrrrrrrrrrrrrrrrrrrrrrrrr}
$i$ \\
$0$ & \ycelld{2} & \ \  & \ \  & \ \  & \ \  & \ \  & \ \  & \ \  & \ \  & \ \  & \ \  & \ \  & \ \  & \ \  & \ \  & \ \  & \ \  & \ \  & \ \  & \ \  & \ \  & \ \  & \ \  & \ \  & \ \  & \ \  & \ \  & \ \  & \ \  & \ \  & \ \  & \ \  & \ \  & \ \  & \ \  \\
$1$ & \gcellD{0} & \ycelld{2} & \ \  & \ \  & \ \  & \ \  & \ \  & \ \  & \ \  & \ \  & \ \  & \ \  & \ \  & \ \  & \ \  & \ \  & \ \  & \ \  & \ \  & \ \  & \ \  & \ \  & \ \  & \ \  & \ \  & \ \  & \ \  & \ \  & \ \  & \ \  & \ \  & \ \  & \ \  & \ \  & \ \  \\
$2$ & \ycelld{2} & \gcellD{0} & \ycelld{2} & \ \  & \ \  & \ \  & \ \  & \ \  & \ \  & \ \  & \ \  & \ \  & \ \  & \ \  & \ \  & \ \  & \ \  & \ \  & \ \  & \ \  & \ \  & \ \  & \ \  & \ \  & \ \  & \ \  & \ \  & \ \  & \ \  & \ \  & \ \  & \ \  & \ \  & \ \  & \ \  \\
$3$ & \gcellD{0} & \ycelld{6} & \gcellD{0} & \ycelld{2} & \ \  & \ \  & \ \  & \ \  & \ \  & \ \  & \ \  & \ \  & \ \  & \ \  & \ \  & \ \  & \ \  & \ \  & \ \  & \ \  & \ \  & \ \  & \ \  & \ \  & \ \  & \ \  & \ \  & \ \  & \ \  & \ \  & \ \  & \ \  & \ \  & \ \  & \ \  \\
$4$ & \ycelld{2} & \gcellD{0} & \ycelld{12} & \gcellD{0} & \ycelld{2} & \ \  & \ \  & \ \  & \ \  & \ \  & \ \  & \ \  & \ \  & \ \  & \ \  & \ \  & \ \  & \ \  & \ \  & \ \  & \ \  & \ \  & \ \  & \ \  & \ \  & \ \  & \ \  & \ \  & \ \  & \ \  & \ \  & \ \  & \ \  & \ \  & \ \  \\
$5$ & \bcellD{7} & 10 & \gcellD{0} & 20 & \gcellD{0} & 2 & \ \  & \ \  & \ \  & \ \  & \ \  & \ \  & \ \  & \ \  & \ \  & \ \  & \ \  & \ \  & \ \  & \ \  & \ \  & \ \  & \ \  & \ \  & \ \  & \ \  & \ \  & \ \  & \ \  & \ \  & \ \  & \ \  & \ \  & \ \  & \ \  \\
$6$ & 30 & \bcellD{7} & 30 & \gcellD{0} & 30 & \gcellD{0} & 2 & \ \  & \ \  & \ \  & \ \  & \ \  & \ \  & \ \  & \ \  & \ \  & \ \  & \ \  & \ \  & \ \  & \ \  & \ \  & \ \  & \ \  & \ \  & \ \  & \ \  & \ \  & \ \  & \ \  & \ \  & \ \  & \ \  & \ \  & \ \  \\
$7$ & 12 & 0 & \bcellD{7} & 0 & \gcellD{0} & 7 & \gcellD{0} & 2 & \ \  & \ \  & \ \  & \ \  & \ \  & \ \  & \ \  & \ \  & \ \  & \ \  & \ \  & \ \  & \ \  & \ \  & \ \  & \ \  & \ \  & \ \  & \ \  & \ \  & \ \  & \ \  & \ \  & \ \  & \ \  & \ \  & \ \  \\
$8$ & 25 & 26 & 0 & \bcellD{7} & 0 & \gcellD{0} & 21 & \gcellD{0} & 2 & \ \  & \ \  & \ \  & \ \  & \ \  & \ \  & \ \  & \ \  & \ \  & \ \  & \ \  & \ \  & \ \  & \ \  & \ \  & \ \  & \ \  & \ \  & \ \  & \ \  & \ \  & \ \  & \ \  & \ \  & \ \  & \ \  \\
$9$ & 12 & 15 & 12 & 0 & \bcellD{7} & 7 & \gcellD{0} & 2 & \gcellD{0} & 2 & \ \  & \ \  & \ \  & \ \  & \ \  & \ \  & \ \  & \ \  & \ \  & \ \  & \ \  & \ \  & \ \  & \ \  & \ \  & \ \  & \ \  & \ \  & \ \  & \ \  & \ \  & \ \  & \ \  & \ \  & \ \  \\
$10$ & 11 & 15 & 5 & 5 & 0 & \bcellD{7} & 0 & \gcellD{0} & 20 & \gcellD{0} & 2 & \ \  & \ \  & \ \  & \ \  & \ \  & \ \  & \ \  & \ \  & \ \  & \ \  & \ \  & \ \  & \ \  & \ \  & \ \  & \ \  & \ \  & \ \  & \ \  & \ \  & \ \  & \ \  & \ \  & \ \  \\
$11$ & 26 & 16 & 30 & 30 & 5 & 7 & \bcellD{7} & 30 & \gcellD{0} & 5 & \gcellD{0} & 2 & \ \  & \ \  & \ \  & \ \  & \ \  & \ \  & \ \  & \ \  & \ \  & \ \  & \ \  & \ \  & \ \  & \ \  & \ \  & \ \  & \ \  & \ \  & \ \  & \ \  & \ \  & \ \  & \ \  \\
$12$ & 11 & 32 & 26 & 15 & 20 & 12 & 14 & \bcellD{7} & 10 & \gcellD{0} & 27 & \gcellD{0} & 2 & \ \  & \ \  & \ \  & \ \  & \ \  & \ \  & \ \  & \ \  & \ \  & \ \  & \ \  & \ \  & \ \  & \ \  & \ \  & \ \  & \ \  & \ \  & \ \  & \ \  & \ \  & \ \  \\
$13$ & 26 & 3 & 33 & 31 & 5 & 17 & 26 & 16 & \bcellD{7} & 30 & \gcellD{0} & 16 & \gcellD{0} & 2 & \ \  & \ \  & \ \  & \ \  & \ \  & \ \  & \ \  & \ \  & \ \  & \ \  & \ \  & \ \  & \ \  & \ \  & \ \  & \ \  & \ \  & \ \  & \ \  & \ \  & \ \  \\
$14$ & 21 & 14 & 21 & 14 & 21 & 7 & 28 & 12 & 28 & \bcellD{7} & 7 & \gcellD{0} & 7 & \gcellD{0} & 2 & \ \  & \ \  & \ \  & \ \  & \ \  & \ \  & \ \  & \ \  & \ \  & \ \  & \ \  & \ \  & \ \  & \ \  & \ \  & \ \  & \ \  & \ \  & \ \  & \ \  \\
$15$ & 16 & 0 & 0 & 0 & 0 & 28 & 0 & 20 & 5 & 0 & \bcellD{7} & 0 & \gcellD{0} & 0 & \gcellD{0} & 2 & \ \  & \ \  & \ \  & \ \  & \ \  & \ \  & \ \  & \ \  & \ \  & \ \  & \ \  & \ \  & \ \  & \ \  & \ \  & \ \  & \ \  & \ \  & \ \  \\
$16$ & 21 & 11 & 0 & 0 & 0 & 28 & 28 & 5 & 5 & 5 & 14 & \bcellD{7} & 0 & \gcellD{0} & 30 & \gcellD{0} & 2 & \ \  & \ \  & \ \  & \ \  & \ \  & \ \  & \ \  & \ \  & \ \  & \ \  & \ \  & \ \  & \ \  & \ \  & \ \  & \ \  & \ \  & \ \  \\
$17$ & 16 & 7 & 6 & 0 & 0 & 28 & 21 & 13 & 15 & 25 & 12 & 28 & \bcellD{7} & 0 & \gcellD{0} & 27 & \rcellD{0} & 2 & \ \  & \ \  & \ \  & \ \  & \ \  & \ \  & \ \  & \ \  & \ \  & \ \  & \ \  & \ \  & \ \  & \ \  & \ \  & \ \  & \ \  \\
$18$ & 21 & 8 & 28 & 1 & 0 & 28 & 14 & 19 & 3 & 30 & 24 & 26 & 7 & \bcellD{7} & 30 & \rcellD{0} & 26 & \gcellD{0} & 2 & \ \  & \ \  & \ \  & \ \  & \ \  & \ \  & \ \  & \ \  & \ \  & \ \  & \ \  & \ \  & \ \  & \ \  & \ \  & \ \  \\
$19$ & 16 & 14 & 6 & 14 & 31 & 28 & 7 & 13 & 32 & 18 & 22 & 16 & 12 & 21 & \ycellD{7} & 17 & \gcellD{0} & 27 & \gcellD{0} & 2 & \ \  & \ \  & \ \  & \ \  & \ \  & \ \  & \ \  & \ \  & \ \  & \ \  & \ \  & \ \  & \ \  & \ \  & \ \  \\
$20$ & 21 & 5 & 0 & 5 & 0 & 19 & 0 & 5 & 15 & 5 & 22 & 5 & 15 & \ycellD{5} & 30 & \bcellD{7} & 30 & \gcellD{0} & 30 & \gcellD{0} & 2 & \ \  & \ \  & \ \  & \ \  & \ \  & \ \  & \ \  & \ \  & \ \  & \ \  & \ \  & \ \  & \ \  & \ \  \\
$21$ & 26 & 21 & 0 & 0 & 0 & 14 & 14 & 20 & 0 & 0 & 0 & 7 & \rcellD{0} & 0 & 5 & 21 & \bcellD{7} & 0 & \gcellD{0} & 0 & \gcellD{0} & 2 & \ \  & \ \  & \ \  & \ \  & \ \  & \ \  & \ \  & \ \  & \ \  & \ \  & \ \  & \ \  & \ \  \\
$22$ & 11 & 12 & 21 & 0 & 0 & 14 & 28 & 14 & 20 & 0 & 7 & \rcellD{0} & 7 & 0 & 15 & 12 & 7 & \bcellD{7} & 0 & \gcellD{0} & 7 & \gcellD{0} & 2 & \ \  & \ \  & \ \  & \ \  & \ \  & \ \  & \ \  & \ \  & \ \  & \ \  & \ \  & \ \  \\
$23$ & 26 & 8 & 33 & 21 & 0 & 14 & 7 & 27 & 14 & 20 & \rcellD{0} & 21 & 0 & 7 & 5 & 16 & 26 & 28 & \bcellD{7} & 0 & \gcellD{0} & 16 & \gcellD{0} & 2 & \ \  & \ \  & \ \  & \ \  & \ \  & \ \  & \ \  & \ \  & \ \  & \ \  & \ \  \\
$24$ & 11 & 29 & 26 & 19 & 21 & 14 & 21 & 14 & 11 & \rcellD{14} & 27 & 0 & 7 & 0 & 22 & 22 & 24 & 12 & 14 & \bcellD{7} & 7 & \gcellD{0} & 27 & \gcellD{0} & 2 & \ \  & \ \  & \ \  & \ \  & \ \  & \ \  & \ \  & \ \  & \ \  & \ \  \\
$25$ & 12 & 30 & 30 & 30 & 5 & 0 & 0 & 20 & \rcellD{0} & 15 & 14 & 20 & 0 & 0 & 5 & 18 & 30 & 25 & 5 & 0 & \bcellD{7} & 30 & \gcellD{0} & 5 & \gcellD{0} & 2 & \ \  & \ \  & \ \  & \ \  & \ \  & \ \  & \ \  & \ \  & \ \  \\
$26$ & 25 & 32 & 5 & 15 & 20 & 5 & 0 & \rcellD{0} & 30 & 0 & 11 & 14 & 20 & 0 & 15 & 32 & 3 & 15 & 5 & 5 & 28 & \bcellD{7} & 10 & \gcellD{0} & 20 & \gcellD{0} & 2 & \ \  & \ \  & \ \  & \ \  & \ \  & \ \  & \ \  & \ \  \\
$27$ & 12 & 10 & 12 & 10 & 5 & 10 & \ycellD{5} & 20 & 0 & 20 & 14 & 27 & 14 & 20 & 5 & 13 & 19 & 13 & 5 & 20 & 12 & 16 & \bcellD{7} & 30 & \gcellD{0} & 2 & \gcellD{0} & 2 & \ \  & \ \  & \ \  & \ \  & \ \  & \ \  & \ \  \\
$28$ & 30 & 21 & 0 & 7 & 0 & \rcellD{0} & 0 & 5 & 0 & 0 & 21 & 7 & 28 & 14 & 0 & 7 & 14 & 21 & 28 & 0 & 28 & 26 & 14 & \bcellD{7} & 0 & \gcellD{0} & 21 & \gcellD{0} & 2 & \ \  & \ \  & \ \  & \ \  & \ \  & \ \  \\
$29$ & 7 & 30 & 7 & 0 & \ycellD{7} & 0 & 0 & 10 & 5 & 0 & 14 & 14 & 14 & 14 & 19 & 28 & 28 & 28 & 28 & 28 & 7 & 17 & 12 & 7 & \bcellD{7} & 7 & \gcellD{0} & 7 & \gcellD{0} & 2 & \ \  & \ \  & \ \  & \ \  & \ \  \\
$30$ & \ycelld{2} & \gcelld{0} & \ycelld{30} & \rcellD{0} & 0 & 7 & 0 & 5 & 20 & 5 & 21 & 0 & 0 & 0 & 0 & 31 & 0 & 0 & 0 & 0 & 21 & 5 & 20 & 5 & 0 & \bcellD{7} & 0 & \gcellD{0} & 30 & \gcellD{0} & \ycelld{2} & \ \  & \ \  & \ \  & \ \  \\
$31$ & \gcelld{0} & \ycelld{27} & \rcellD{0} & \ycelld{30} & 0 & 0 & 7 & 10 & 15 & 30 & 19 & 21 & 0 & 0 & 5 & 14 & 1 & 0 & 0 & 0 & 14 & 31 & 15 & 30 & 5 & 0 & \bcellD{7} & 0 & \gcellD{0} & 20 & \gcellD{0} & \ycelld{2} & \ \  & \ \  & \ \  \\
$32$ & \ycelld{2} & \rcellD{0} & \ycelld{12} & \gcelld{0} & 30 & 7 & 0 & 12 & 5 & 30 & 26 & 33 & 21 & 0 & 0 & 6 & 28 & 6 & 0 & 0 & 21 & 33 & 26 & 30 & 5 & 12 & 0 & \bcellD{7} & 30 & \gcellD{0} & \ycelld{12} & \gcellD{0} & \ycelld{2} & \ \  & \ \  \\
$33$ & \rcellD{0} & \ycelld{31} & \gcelld{0} & \ycelld{27} & 0 & 30 & 21 & 10 & 32 & 30 & 29 & 8 & 12 & 21 & 5 & 14 & 8 & 7 & 11 & 0 & 14 & 3 & 32 & 16 & 15 & 15 & 26 & 0 & \bcellD{7} & 10 & \gcellD{0} & \ycelld{6} & \gcellD{0} & \ycelld{2} & \ \  \\
$34$ & \ycelld{2} & \gcelld{0} & \ycelld{2} & \gcelld{0} & 2 & 7 & 30 & 12 & 25 & 12 & 11 & 26 & 11 & 26 & 21 & 16 & 21 & 16 & 21 & 16 & 21 & 26 & 11 & 26 & 11 & 12 & 25 & 12 & 30 & \bcellD{7} & \ycelld{2} & \gcellD{0} & \ycelld{2} & \gcellD{0} & \ycelld{2} \\
\hline
\multicolumn{1}{r}{$j$} & 0 & 1 & 2 & 3 & 4 & 5 & 6 & 7 & 8 & 9 & 10 & 11 & 12 & 13 & 14 & 15 & 16 & 17 & 18 & 19 & 20 & 21 & 22 & 23 & 24 & 25 & 26 & 27 & 28 & 29 & 30 & 31 & 32 & 33 & 34
\end{tabular}
}
\parbox[c][2.6em][c]{\linewidth}{\centering $S^{\hspace{0.02cm}35}_{i,j}\mymod{{\hspace{0.02cm}35}}$}
\label{tab:annex}
\end{table}

\end{document}